\documentclass[11pt]{amsart}
\usepackage[utf8x]{inputenc}
\usepackage{amsfonts}
\usepackage{amssymb}
\usepackage{amsmath}
\usepackage{delarray}
\usepackage{graphicx}
\usepackage{hyperref}
\usepackage[all]{xy}
\newtheorem{thm}{Theorem}[section]
\newtheorem{cor}[thm]{Corollary}

\newtheorem{lemma}[thm]{Lemma}
\newtheorem{prop}[thm]{Proposition}
\newtheorem{conj}[thm]{Conjecture}

\newtheorem{defn}[thm]{Definition}
\theoremstyle{remark}
\newtheorem{remark}[thm]{Remark}
\numberwithin{equation}{section}
\newtheorem{example}[thm]{Example}

\newcommand{\cK}{\mathcal K}
\newcommand{\cN}{\mathcal N}

\newcommand{\cH}{\mathcal H}
\newcommand{\cD}{\mathcal D}
\newcommand{\cO}{\mathcal O}

\newcommand{\cU}{\mathcal U}

\newcommand{\cS}{\mathcal S}

\newcommand{\bbP}{\mathbb P}
\newcommand{\bbR}{\mathbb R}
\newcommand{\bbT}{\mathbb T}
\newcommand{\bbC}{\mathbb C}

\newcommand{\bbN}{\mathbb N}

\newcommand{\bbS}{\mathbb S}
\newcommand{\bbZ}{\mathbb Z}

\newcommand{\rank}{{\rm rank\ }}

\newcommand{\QED}{\hfill $\square$\vspace{2mm}}

\begin{document}

\newpage
\title{Geometry of nondegenerate $\bbR^n$-actions on $n$-manifolds}

\author{Nguyen Tien Zung and Nguyen Van Minh}
\address{Institut de Mathématiques de Toulouse, UMR5219, Université Toulouse 3}
\email{tienzung.nguyen@math.univ-toulouse.fr, minh@math.univ-toulouse.fr}

\date{Version 2, February 2013, , to appear in J. Math. Soc. Japan}

\begin{abstract} 
This paper is devoted to a systematic study of the geometry of nondegenerate $\bbR^n$-actions on $n$-manifolds. 
The motivations for this study come from both dynamics, where these actions form a special class of integrable
dynamical systems and the understanding of their nature is important for the study of other Hamiltonian
and non-Hamiltonian integrable systems, and geometry, where these actions are related to a lot of other geometric 
objects, including reflection groups, singular affine structures, toric and quasi-toric manifolds, monodromy 
phenomena, topological invariants, etc. We construct a geometric theory of these actions, and obtain a series of results, including: 
local and semi-local normal forms, automorphism and twisting groups, the reflection principle, the toric degree, the monodromy, 
complete fans associated to hyperbolic domains, quotient spaces, elbolic actions and
 toric manifolds, existence and classification theorems. 
\end{abstract}

\maketitle

{\small \tableofcontents }


\section{Introduction}

This paper is devoted to a systematic study of the geometry of nondegenerate $\bbR^n$-actions on $n$-manifolds.
We develop a theory of these actions, and obtain a series of results, which will be presented
below, and also open questions. The motivations for this study come from both integrable dynamical systems and geometry.

Even though, according to many results, from the classical works of Poincaré and Kovalevskaya to the modern 
theory of differential Galois obstructions  (see, e.g., \cite{MRS-Galois2007}), integrable systems are 
``rare and far in between'' in the world of all dynamical systems, a lot of physical dynamical systems of vital importance to us, e.g. the movement of our solar
system, the internal dynamics of usual molecules like H2O, sinusoidal and soliton waves, and so on, can in fact be 
considered as integrable. This ``contradiction'' is explained by the fact that, integrable systems often possess very strong 
stability, and their ``longivity'' compensates for their ``rarity'', and that's also why they are an important subject of study
in physics and mathematics.

One of the most fundamental results in integrable Hamiltonian systems is  the classical Arnold--Liouville--Mineur theorem,
which gives a semi-local normal form for a regular invariant torus in terms of action-angle variables. This normal form result
is also the starting point for many works on the geometry of integrable systems, which is still a research
subject of great actual interest, see, e.g. 
\cite{Audin-Torus2004, Babelon-Book2003, BolsinovFomenko-IntegrableBook, BFO-Integrable2006, 
Duistermaat-globalaction-angle1980, Eliasson-Normal1990, Kappeler-KAM2003, PelayoSan_Acta2011, Vey-Sur1978}
 and references therein. The first author of this paper  also contributed to the geometric and topological study of 
 integrable Hamiltonian systems, with  results on normal forms, singularities, and global topological 
and geometric invariants of such systems, see, e.g.
\cite{Zung-Symplectic1996,Zung-Integrable2003, MirandaZung-NF2004, Zung-Convergence2005}.

There are also many natural dynamical systems which are non-Hamiltonian due to various reasons 
(nonholonomic constraints or loss of energy for example), but which are still integrable in a natural sense: 
they still possess complete sets of commuting vector fields on invariant submanifolds. Many authors have been
working non-Hamiltonian integrable dynamical systems over the last decades from different
points of view, see, e.g., 
\cite{Bogoyavlenskij-Extended1998, Bates-nonholonomic1998, FJ-Chaplygin2006, CD-NonHamiltonian2001, 
Stolovitch-Singular2000, Zung-Convergence2002,AyoulZung_Galois2010} for a small sample.
However, much less is known about the topology and geometry of non-Hamiltonian integrable systems 
than for Hamiltonian ones. Until very recently, even the notion of nondegenerate singularities didn't 
exist for integrable non-Hamiltonian systems. Our program is to remedy  this situation, to study
the geometry and topology of integrable non-Hamiltonian systems in a systematic way, and to try to 
obtain analogs for the non-Hamiltonian case of all known results in the Hamiltonian case. 

In this paper, which is part of our program, we will study the 
geometry of a subclass of integrable systems, namely the systems
of type $(n,0)$, formed by $n$ commuting vector fields (and 0 function) on a manifold of dimension $n$.
This class is of particular importance, because  if we restrict our attention to ``minimal'' invariant manifolds of 
other  integrable systems, then they will become systems of this subclass, i.e. the number of commuting 
vector fields is exactly equal to the dimension of the invariant (sub)manifold.

From the geometric point of view, an integrable system of type $(n,0)$ is simply an action
of $\bbR^n$ on a $n$-manifold. If the manifold is compact and the action is locally free, then according to
the classical Liouville's theorem, the manifold is a $n$-dimensional torus on which $\bbR^n$ acts by translations.
At first, we also thought that the general situation, when there are singular points of the action, is not much more
complicated than this regular situation. But we were wrong. It turns out that the geometry of $\bbR^n$-actions
on $n$-manifolds with nondegenerate singularities
is extremely rich, and this is the second motivation for our interest in this subject. In particular, we recover, 
for example, all the toric and quasi-toric manifolds in our study. Phenomena like monodromy and 
reflection group actions, which we didn't suspect at first, are also there.

For simplicity, throughout this paper, we will always assume that the 
actions are smooth, though most results are also valid for $C^1$-actions. 

The organization of this paper is as follows:

In section 2, we study the local and semi-local structure of singularities of nondegenerate $\bbR^n$-actions on 
$n$-manifolds. In particular, we obtain local and semi-local normal forms 
(Theorem \ref{thm:LinearNF}, Theorem \ref{thm:semi-localform} and Theorem \ref{thm:semi-localform2}) 
which show that these singularities 
can be linearized and decomposed into an almost direct
product of regular components, hyperbolic components, and the so called {\bf elbolic} components. The word 
\emph{elbolic} means a combination of elliptic and hyperbolic, and is a word that we didn't find in the literature 
but invented for our study. We also describe {\bf adapted bases} of 
the action (Theorem \ref{thm:CanonicalCoordinate}), local automorphism groups
(Theorem \ref{thm:LocalAutomorphism}), {\bf twisting groups} (Definition \ref{defn:TwistingGroup}), 
show the {\bf reflection principle} 
for hyperbolic singularities (Theorem \ref{thm:Reflection}), which is reminiscent
of Schwartz reflection principle in complex analysis, introduce the {\bf HERT-invariant} (Definition \ref{defn:HERT-invariant}), and 
study the variation of this invariant from orbit to orbit 
(Proposition \ref{prop:T-semicontinuity} and Proposition \ref{prop:adjacent-orbits}).

In section 3, we introduce the notion of {\bf toric degree} of an action $\rho: \bbR^n \times M^n \to M^n$,
and obtain a simple formula toric degree$(\rho)=e + t$ for calculating the toric degree from the 
HERT-invariant of an arbitrary point on the manifold (Theorem \ref{thm:HERT-toricdegree}). If $\rho$ has toric degree $t(\rho)$,
then it induces an effective $\bbT^{t(\rho)}$ action on $M^n$, called the {\bf associated toric action}
and denoted by $\rho_\bbT$, and the study of $\rho$ is reduced to the study
of the  torus action $\rho_\bbT$ and of the reduced hyperbolic $\bbR^{n-t(\rho)}$ action $\rho_\bbR$ 
on the quotient space $M^n/\rho_\bbT$. Using this strategy, we obtain a complete 
classification of actions of toric degree $n$ and $n-1$ in this section
(Theorem \ref{thm:class-marked}). As is known from the literature on topology, there are strong obstructions for a manifold to admit
an effective $\bbT^k$-action with $k\geq 1$. So the maximal possible toric degree for an action on a given manifold
is an interesting invariant, which we call the {\bf toric rank} (Definition \ref{defn:toricrank}), 
and which is somehow related to Milnor's rank of a manifold.

In section 4, we introduce a natural global invariant, called the {\bf monodromy} of $(M^n,\rho)$, which is a homomorphism
$$\mu : \pi_1(M^n) \to \bbR^n/Z_\rho$$  
from the fundamental group of $M^n$ to the quotient of $\bbR^n$ by the isotropy group $Z_\rho$ of $\rho$.
An important observation is that the monodromy group $\mu(\pi_1(M^n))$ contains all the twisting groups, i.e.
the twistings are part of the monodromy (Theorem \ref{thm:TwistingMonodromy}).
Theorem \ref{thm:lifting-monodromy} says that one can trivialize the monodromy by taking a normal covering 
$(\widetilde M, \widetilde \rho)$ of $(M, \rho)$. Theorem \ref{thm:trans-monodromy} allows one to modify 
the ``free part'' $\mu_{\text{free}}$ of the monodromy in an arbitrary way, in order to obtain new actions 
(on the same manifold) which are locally isomorphic but 
globally non-isomorphic to the old ones.

Section 5 is devoted to the study of {\bf totally hyperbolic actions}, i.e. action 
$\rho: \bbR^n \times M^n \to M^n$ of toric degree 0. Among other things, we give a classification
of {\bf hyperbolic domains} (i.e. regular orbits of type $\bbR^n$) with compact closure by their associated 
{\bf complete fans} (Theorem \ref{thm:classificationByfan}). A complete fan in our sense is similar to a complete fan associated
to a compact toric variety, except that our vectors are not required to lie in the integral lattice. 
We observe that compact closed hyperbolic domains are contractible manifolds with
boundary and corners (Theorem \ref{thm:closureHyperbolic} and Theorem \ref{thm:hyperbolic-contractible}) 
which look like convex simple polytopes, but which are not 
always diffeomorphic to convex simple polytopes (Theorem \ref{thm:diff-to-polytope}). We show that not every decomposition of a manifold
into simple contractible polyhedral pieces can be realized by a totally hyperbolic action, even in dimension 2
(Proposition \ref{prop:3-6domains}), but think that maybe any smooth manifold admits a totally hyperbolic action 
(Theorem \ref{thm:hyperbolic_dim2} and Conjecture \ref{conjecture:hyperbolic}). The global classification of totally 
hyperbolic actions (Theorem \ref{class-hyp-invariants}) involves the associated fan 
family (one fan for each hyperbolic domain) and the monodromy.

In section 6, we study the reduction $(M,\rho)$ by the action of its associated torus action. The result of the reduction is
$(Q,\rho_\bbR)$, where $Q = M^n/\rho_\bbT$ is the {\bf quotient space}, which is an {\bf orbifold}, and $\rho_\bbR$ is the
{\bf reduced action}, which is a totally hyperbolic action (Theorem \ref{thm:hypAction-quotientSpace})
and which can be classified by the results of Section 5. 
We show in this section that the singular torus fibration $M^n \to Q$ is in a sense topologically trivial, or more precisely,
it always admits a smooth cross section (if there are no twistings) or multi-section (when there are twistings) (Proposition \ref{prop:section}
and Proposition \ref{prop:multisection}). Then we show how to get back $(M,\rho)$ 
from $(Q,\rho_\bbR)$, and a {\bf complete set of invariants} for classifying 
nondegenerate actions $(M,\rho)$ (Theorem \ref{thm:GoingBack}).

In section 7, we restrict our attention to a special subclass of nondegenerate $\bbR^n$-actions, called 
{\bf elbolic actions}, i.e. actions whose singularities contain only elbolic components (no hyperbolic component).
After giving some preliminary results about these actions, we show that manifolds of dimension $n = 2m$ admitting
an elbolic action of toric degree $m$ are exactly the same as {\bf topological toric manifolds} in the sense of
Ishida, Fukukawa, Masuda \cite{Ishida-toric}, which are a very natural generalization of complex toric manifolds.
In particular, we recover Ishida--Fukukawa--Masuda's classification theorem for these topological toric manifolds
from our more general point of view of  nondegenerate $\bbR^n$-actions on $n$-manifolds.

In section 8, the last section of this paper, we study in some detail another subclass of nondegenerate 
$\bbR^n$-actions on $n$-manifolds, namely actions of toric degree $n-2$. We describe these actions via the 
2-dimensional quotient space $M^n/\bbT^{n-2}$, and give the full list of 10 types of possible singularities
in this case (Theorem \ref{thm:n-2a} and Theorem \ref{thm:n-2b}).

The following is a schematic presentation of the sections of the paper:
\begin{center}
\setlength{\unitlength}{1cm}
\begin{picture}(10,9)
\put(0.5,0.3){7: Elbolic actions}
\put(0,0){\line(0,1){1}}
\put(0,0){\line(1,0){4}}
\put(4,0){\line(0,1){1}}
\put(0,1){\line(1,0){4}}

\put(6.7,0.3){8: $\bbT^{n-2}$-case}
\put(6,0){\line(0,1){1}}
\put(6,0){\line(1,0){4}}
\put(10,0){\line(0,1){1}}
\put(6,1){\line(1,0){4}}

\put(0.5,2.3){6: Reduction}
\put(0,2){\line(0,1){1}}
\put(0,2){\line(1,0){4}}
\put(4,2){\line(0,1){1}}
\put(0,3){\line(1,0){4}}

\put(6.2,2.3){5: Hyperbolic actions}
\put(6,2){\line(0,1){1}}
\put(6,2){\line(1,0){4}}
\put(10,2){\line(0,1){1}}
\put(6,3){\line(1,0){4}}

\put(0.6,4.3){3: Toric degree}
\put(0,4){\line(0,1){1}}
\put(0,4){\line(1,0){4}}
\put(4,4){\line(0,1){1}}
\put(0,5){\line(1,0){4}}

\put(6.6,4.3){4: Monodromy}
\put(6,4){\line(0,1){1}}
\put(6,4){\line(1,0){4}}
\put(10,4){\line(0,1){1}}
\put(6,5){\line(1,0){4}}

\put(3.6,6.3){2: Singularities}
\put(3,6){\line(0,1){1}}
\put(3,6){\line(1,0){4}}
\put(7,6){\line(0,1){1}}
\put(3,7){\line(1,0){4}}

\put(3.6,8.3){1: Introduction}
\put(3,8){\line(0,1){1}}
\put(3,8){\line(1,0){4}}
\put(7,8){\line(0,1){1}}
\put(3,9){\line(1,0){4}}

\put(5,7.8){\vector(0,-1){0.7}}
\put(4,5.8){\vector(-1,-1){0.7}}
\put(6,5.8){\vector(1,-1){0.7}}
\put(2,3.8){\vector(0,-1){0.7}}
\put(8,3.8){\vector(0,-1){0.7}}
\put(2,1.8){\vector(0,-1){0.7}}
\put(8,1.8){\vector(0,-1){0.7}}

\put(8,1.8){\vector(-4,-1){3.8}}
\put(2,1.8){\vector(4,-1){3.8}}

\put(8,3.8){\vector(-4,-1){3.8}}
\put(5.8,2.5){\vector(-1,0){1.6}}
\end{picture}
\end{center}
 
\section{Nondegenerate singularities}
\label{section:singularities}

\subsection{Definition of nondegenerate singularities} 

Let $\rho: \bbR^n \times M^n \to M^n$ be a smooth action of $\bbR^n$ on a $n$-dimensional connected 
manifold $M^n$ (which is not necessarily compact). Then it is generated by $n$  commuting vector fields $X_1,\hdots, X_n$ on $M^n$,
which are given by the following formula at each point $z \in M$:
\begin{equation}
X_i (z) = \frac{d}{dt}\rho((0, \hdots,t, \hdots,0),z)|_{t=0}. 
\end{equation}
Conversely, a $n$-tuple of commuting vector fields $X_1, \hdots, X_n$ on $M$ will generate an infinitesimal action of 
$\bbR^n$, which integrates into an action of $\bbR^n$, provided that these vector fields are complete.
Most of the times, we will assume that the vector fields $X_1, \hdots, X_n$ are complete. 
However, the definitions and results of purely local
nature of the paper don't need this completeness condition and remain valid for infinitesimal $\bbR^n$-actions.

If $v = (v^1,\hdots,v^n)  \in \bbR^n$ is an element of $\bbR^n$ (viewed as an Abelian Lie algebra),
then we will put
\begin{equation}
 X_v = \sum_{i=1}^n v^i X_i
\end{equation}
and call it the {\bf generator of the action $\rho$ associated to $v$}. 
In particular, if $v_1,\hdots,v_n \in \bbR^n$ are linearly independent, then the vector fields $Y_1 = X_{v_1},\hdots, Y_n = X_{v_n}$ 
also generate the same $\bbR^n$ action as $\rho$, up to an automorphism of $\bbR^n$.

A point $z \in M^n$ is called a {\bf singular point} with respect to a given 
action $\rho: \bbR^n \times M^n \to M^n$ if the {\bf rank} of $z$ with respect to $\rho$, 
defined by the formula
\begin{equation}
 \rank z = \dim Span_\bbR(X_1 (z),\hdots, X_n(z))
\end{equation}
is smaller than $n$. The number $(n - \rank z)$ will then be called the {\bf corank} of $z$. 
Thus a point is singular
if and only if it has positive corank. If $\rank z = n$ then we say that $z$ is a 
{\bf regular point} of the action. Clearly, a regular point is a point at which the action 
is locally free. If $\rank z = 0$ then we say that $z$ is a {\bf fixed point} of the action.

The definition of nondegenerate singular points that we present in this section is a special case of the definition
of nondegenerate singularities of integrable non-Hamiltonian systems in \cite{Zung-Nondegenerate2012}. In fact,
from a geometric point of view, a complete  integrable non-Hamiltonian system of type $(n,0)$ on a $n$-dimensional
manifold $M^n$ is the same thing as a $\bbR^n$-action on $M^n$.

If $z$ is a fixed point of the action, then $X_1(z) = \hdots = X_n(z) = 0$, and we can talk about the linear part
$X^{(1)}_i$ of $X_i$ at $z$ for each $i=1,\hdots,n$: these are well-defined linear vector fields on the tangent space 
$T_{z}M^n \cong \bbR^n$. Since $[X_i, X_j] = 0$, we also have $[X_i^{(1)}, X_j^{(1)}] = 0$, 
i.e. $(X_1^{(1)}, \hdots, X_n^{(1)})$  generate a linear action of $\bbR^n$ 
on $T_{z}M^n \cong \bbR^n$, which will be denoted by $\rho^{(1)}$ and called {\bf the linear part of $\rho$ at $z$}. 
Recall that the set of linear vector fields
on $\bbR^n$ is naturally isomorphic to the general Lie algebra $gl(n,\bbR)$, which is a reductive Lie algebra of rank 
$n$. In particular, any Abelian subalgebra of $gl(n,\bbR)$ has dimension at most $n$, and $gl(n,\bbR)$ contains
Cartan subalgebras, i.e. Abelian subalgebras of dimension exactly $n$ whose elements are semisimple.

\begin{defn}
A linear action $\rho^{(1)}$ of $\bbR^n$ on a $n$-dimensional vector space $V$ is called a {\bf nondegenerate linear action}
if the Abelian Lie algebra $Span_\bbR (X^{(1)}_1,\hdots, X^{(1)}_n)$ spanned by the generators of $\rho^{(1)}$
is a Cartan subalgebra of $gl(V)$, i.e. it is of dimension $n$ and all of its elements are semisimple (i.e. diagonalizable over
$\bbC$). A fixed point $z$ of a smooth action $\rho: \bbR^n \times M^n \to M^n$ is called a {\bf nondegenerate fixed point}
if the linear part $\rho^{(1)}$ of $\rho$ at $z$ is nondegenerate.
\end{defn}

For non-fixed singular points, the definition of nondegeneracy is similar. Let $z$ be a singular point whose corank is equal to 
$k < n$. Up to an automorphism of $\bbR^n$ we may assume that 
$X_i = \frac{\partial}{\partial x_i}$ for all $i = k+1,\hdots, n$ and
$X_1(z) = \hdots = X_k(z) = 0$ in a  local coordinate system in a neighborhood of $z$. Forgetting about the coordinates 
$x_{k+1},\hdots x_{n}$, we get an infinitesimal action of $\bbR^{k}$, generated by  $X_1,\hdots, X_k$ 
(or more precisely, their projections) on the local $k$-dimensional manifold $\{x_{k+1} = \hdots = x_n = 0\}$, 
which admits $z$ as a fixed point. This infinitesimal action is called the reduced transversal action of $\bbR^k$
at $z$; it is unique up to local isomorphisms, and it can be defined intrinsically.

\begin{defn}
 A singular point $z$ of corank $k$ of an action $\rho: \bbR^n \times M^n \to M^n$ is called {\bf nondegenerate} 
if $z$ is a nondegenerate fixed point
of the reduced transversal $\bbR^k$-action at $z$.
\end{defn}

\begin{defn}
 An action $\rho: \bbR^n \times M^n \to M^n$ is called a {\bf nondegenerate action} if every singular point of $\rho$ on $M^n$
is nondegenerate.
\end{defn}

In this paper, we will consider only nondegenerate actions of $\bbR^n$ on $n$-dimensional manifolds.
Let us give here some explanations on why we impose the above nondegeneracy 
condition:

1) The above nondegeneracy condition is consistent with the general case of  
integrable (Hamiltonian or non-Hamiltonian) dynamical systems 
\cite{Vey-Sur1978,Eliasson-Normal1990, Zung-Symplectic1996,Zung-Nondegenerate2012}.
Nondegenerate singularities are geometrically linearizable, structurally stable, and most singularities of natural integrable
systems coming from  mechanics and physics that have been studied in the literature are nondegenerate.
(The problem of global structural stability of nondegenerate actions $\rho: \bbR^n \times M^n \to M^n$ will be
addressed in a separate paper).

2) Actions which are too degenerate are not very beautiful, do not appear in the real world,
and don't give much information
about the ambient manifolds. It is easy to construct on any $n$-manifold $M$ a degenerate $\bbR^n$ action which is free almost 
everywhere as follows: fill $M^n$, up to a  nowhere dense set of measure 0, 
by a countable disjoint union of open  balls $(B_j;\ j \in J)$. For each $j \in J$ construct a smooth diffeomorphism from $\bbR^n$
onto $B_j$, and denote by $X_1, \hdots, X_n$ the push-forward  of the standard vector fields 
$\frac{\partial}{\partial x_1}, \hdots, \frac{\partial}{\partial x_n}$ on $\bbR^n$ by this map. 
Extend these vector fields to the exterior of the balls
by putting them equal 0 outside the balls. If the maps are chosen well enough, we will get a commuting family of vector fields
$X_1,\hdots,X_n$  which generate an action of $\bbR^n$ on $M^n$, which is very degenerate but almost everywhere free.
The above construction of degenerate commuting vector fields is folkloric, and its variations have appeared, for example, 
in the papers by Camacho \cite{Camacho-MorseSmaleAction1973} and Weinstein \cite{Weinstein-RealSemisimple1987}.
 
3) In \cite{Camacho-MorseSmaleAction1973}, Camacho introduced a notion of Morse--Smale actions of $\bbR^2$ on 2-manifolds. 
It turns out that Camacho's Morse-Smale condition is essentially the same as our nondegeneracy condition. 
Another set of conditions for $\bbR^2$-actions on 2-manifolds was introduced by Sabatini in \cite{Sabatini-commuting1997}. 
Sabatini's conditions (in particular, his condition that every singular point is an isolated fixed point) 
is different from ours, but is a bit similar to the so called \emph{elbolic} case of our actions, 
and leads to some similar arguments and conclusions. In fact, if an action satisfies Sabatini's conditions and
is nondegenerate then it is an elbolic action in the sense of Definition \ref{defn:elbolic}.
As was observed by Camacho, nondegenerate actions are not dense in the space of all actions with respect to a natural
topology: due to the rigidity of the commutativity condition, there are very degenerate actions which cannot be purturbed
into nondegenerate actions. Still, we believe that the nondegeneracy condition is very reasonable, and very degenerate actions
are ``pathological'' and should not  be viewed as representative of integrable systems.

4) Starting from 2 degrees of freedom, in typical integrable Hamiltonian systems, besides nondegenerate singularities, one
finds also degenerate singularities of finite determinacy, see, e.g. \cite{BolsinovFomenko-IntegrableBook}. 
It would be nice to include degenerate (but not too degenerate)
singularities into the study of integrable non-Hamiltonian systems. Our work on the topological classification 
of integrable Hamiltonian systems \cite{Zung-Integrable2003} also includes degenerate singularities. We will treat
typical degenerate singularities of integrable non-Hamiltonian systems in general and of $\bbR^n$-actions on 
$n$-manifolds in particular in some future work. As will be shown in this paper, the class of nondegenerate 
$\bbR^n$-actions on $n$-manifolds is already very rich, with lots of things to say about them, and 
interesting open questions.

\subsection{Local normal form}

Nondegenerate linear actions of $\bbR^n$ on $\bbR^n$ can be classified, up to isomorphisms, by their corresponding Cartan
subalgebras of $gl(n,\bbR)$. In terms of linear vector fields, this classification can be written as follows:

\begin{thm}[Classification of nondegenerate linear actions] \label{thm:LinearNF}
Let $\rho^{(1)}: \bbR^n \times \bbR^n \to \bbR^n$ be a nondegenerate linear action of the Abelian group
$\bbR^n$ on the vector space $\bbR^n$.
Then there exist nonnegative integers  $h, e \geq 0$ such that $h + 2e = n$, a linear coordinate system $x_1, \hdots, x_n$
on $\bbR^n$, and a linear basis $(v_1,\hdots,v_n)$ of the Lie algebra $\bbR^n$ such that the generators $Y_i = X_{v_i}$
of the action $\rho^{(1)}$ associated to the basis $(v_1,\hdots,v_n)$ can be written as follows:
\begin{equation}
\begin{cases}
Y_i = x_i\frac{\partial }{\partial x_i} \quad \forall \quad i = 1,\hdots, h \\
Y_{h+2j-1} = x_{h+2j-1}\frac{\partial }{\partial x_{h+2j-1}} +  x_{h+2j}\frac{\partial }{\partial x_{h+2j}}  \\
Y_{h+2j} = x_{h+2j-1}\frac{\partial }{\partial x_{h+2j}} -  x_{h+2j}\frac{\partial }{\partial x_{h+2j-1}} \quad \forall \quad j = 1,\hdots, e. 
\end{cases} 
\end{equation}
\end{thm}

The numbers $e,h$ in the above theorem form a complete set of invariants of the nondegenerate linear action up to automorphisms
of (the group) $\bbR^n$ and automorphisms of (the vector space) $\bbR^n$ on which $\bbR^n$ acts,
and they are called the number of {\bf elbolic} components and the number of {\bf hyperbolic} components respectively.

The proof of the above theorem is a simple excercise of linear 
algebra: since the linear vector fields $X_i$ commute, they are simultaneously
diagonalizable over $\bbC$. Their joint 1-dimensional real eigenspaces correspond to hyperbolic components, while
joint complex eigenspaces correspond to elbolic components.

\begin{remark}
We didn't find the word \emph{elbolic} in the literature, and so we invented it for this paper: 
elbolic is a contraction of  \emph{elliptic + hyperbolic}, 
to describe a 2-dimensional situation with both an elliptic type 
sub-component and a hyperbolic type sub-component, see Figure \ref{fig:elbolic}.
\end{remark}
\begin{figure}[htb] 
\begin{center}
\includegraphics[width=115mm]{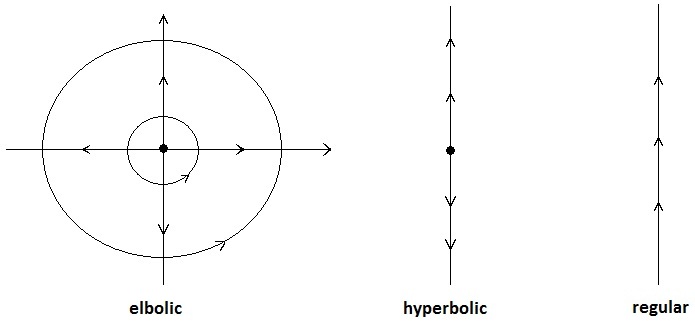}
\caption{Elbolic, hyperbolic, and regular components of $\bbR^n$-actions on $n$-manifolds.}
\label{fig:elbolic}
\end{center}
\end{figure}

Theorem \ref{thm:LinearNF} together with smooth linearization techniques lead
to the following smooth local normal form theorem, which was obtained in \cite{Zung-Nondegenerate2012}:

\begin{thm}[\cite{Zung-Nondegenerate2012}, Local normal form] \label{thm:NormalForm}
Let $p$ be a nondegenerate singular point of co-rank $m \leq n$ of a smooth nondegenerate action 
$\rho: \bbR^n \times M^n \to M^n$. Then there exists a smooth local coordinate 
system $(x_1,x_2,\hdots, x_n)$ in a neighborhood of $p$, non-negative integers $h, e \geq 0$
such that $h + 2e =m$, and a basis $(v_1,\hdots, v_n)$ of $\bbR^n$ such that the corresponding
generators $Y_i = X_{v_i} (i = 1, \hdots, n)$ of $\rho$ have the following form:
\begin{equation}
\begin{cases}
Y_i = x_i\frac{\partial }{\partial x_i} \quad \forall \quad i = 1,\hdots, h \\
Y_{h+2j-1} = x_{h+2j-1}\frac{\partial }{\partial x_{h+2j-1}} +  x_{h+2j}\frac{\partial }{\partial x_{h+2j}}  \\
Y_{h+2j} = x_{h+2j-1}\frac{\partial }{\partial x_{h+2j}} -  x_{h+2j}\frac{\partial }{\partial x_{h+2j-1}} \quad \forall \quad j = 1,\hdots, e \\
Y_k = \frac{\partial }{\partial x_k} \quad \forall \quad k = m+1,\hdots, n.
\end{cases}
\end{equation}
The numbers $(h, e)$ do not depend on the choice of local coordinates.
\end{thm}
\begin{defn} \label{def:HE-invariant}
The couple $(h,e)$ in the above theorem is called the {\bf HE-invariant} of the action $\rho$ at $p$.
The number $e$ is called the number of {\bf elbolic} components, and $h$ is called the number of
{\bf hyperbolic} components of $\rho$ at $p$. The coordinate system $(x_1, \hdots, x_n)$ in this theorem
is called a local {\bf canonical system of coordinates}, and the basis $(v_1,\hdots, v_n)$ of $\bbR^n$
is called an {\bf adapted basis} of the action $\rho$ at $p$.
\end{defn}

Local canonical coordinate systems at a point $p$ and associated adapted bases of $\bbR^n$
are not unique, but they are related to each other by the following theorem:

\begin{thm}[Adapted bases] \label{thm:CanonicalCoordinate}
Let $(x_1, \hdots, x_n)$ be a canonical system of coordinates at a point $p$ of a nondegenerate action $\rho$
together with an associated adapted basis $(v_1,\hdots, v_n)$ of $\bbR^n$ as in Definition \ref{def:HE-invariant}.
Let $(y_1,\hdots, y_n)$ be another canonical system of coordinates at $p$ together with an associated adapted 
basis  $(w_1,\hdots, w_n)$ of $\bbR^n$. Then we have:

i) The vectors $(v_1,\hdots, v_h)$ are the same as the vectors $(w_1,\hdots, w_h)$ up to permutations, where 
$h$ is number of hyperbolic components.

ii) The $e$-tuples of pairs of vectors $((v_{h+1},v_{h+2}),\hdots, (v_{h+2e-1},v_{h+2e}))$ is also the 
same as the $e$-tuples $((w_{h+1},w_{h+2}),\hdots, (w_{h+2e-1},w_{h+2e}))$ up to permutations
and changes of sign of the type 
\begin{equation}
(v_{h+2i-1},v_{h+2i}) \mapsto (v_{h+2i-1},-v_{h+2i})
\end{equation}
(only the second vector, the one whose corresponding generator of $\rho$ is a vector field
whose flow is $2\pi$-periodic, changes sign).

iii) Conversely, if $(x_1, \hdots, x_n)$ and  $(v_1,\hdots, v_n)$ are as in Theorem \ref{thm:NormalForm}, 
and $(w_1,\hdots, w_n)$ is another basis of $\bbR^n$ which satisfies the above conditions i) and ii), then
$(w_1,\hdots, w_n)$ is the adapted basis of $\bbR^n$ for another canonical system of coordinates 
$(y_1,\hdots, y_n)$ at $p$.
\end{thm}

\begin{proof}
Since $(v_1,\hdots, v_n)$ is a basis of $\bbR^n$, for any $w \in \bbR^n$ we can write 
$w = \sum_{i=1}^n \alpha_iv_i$, and hence
\begin{eqnarray*}
X_w &=& \sum_{i=1}^n \alpha_iY_i\\
&=&  \sum_{i=1}^h \alpha_ix_i\frac{\partial}{\partial x_i}
+\sum_{i=1}^e \bigg[\alpha_{h+2i-1}\big(x_{h+2i-1}\frac{\partial}{\partial x_{h+2i-1}}+ x_{h+2i}\frac{\partial}{\partial x_{h+2i}}\big) \\
 &+& \alpha_{h+2i}\big(x_{h+2i-1}\frac{\partial}{\partial x_{h+2i}} - x_{h+2i}\frac{\partial}{\partial x_{h+2i-1}}\big)\bigg]
+\sum_{i=h+2e+1}^n \alpha_i\frac{\partial}{\partial x_i}.
\end{eqnarray*}
If $X_w = y_1\frac{\partial}{\partial y_1}$ in some new coordinate system, then in particular
$X_w(p) = 0$, therefore $\alpha_i = 0 \quad \forall i \geq h +2e+1$.
Moreover, $X_w$ has only one non-zero eigenvalue, which is equal to $1$, and all the other eigenvalues 
(counting with multiplicities) are 0. On the other hand, $\alpha_1,\hdots, \alpha_h $ 
and $\alpha_{h+2i-1} \pm \sqrt{-1}\alpha_{h+2i}$  are eigenvalues of $X_w$. From there it is obvious that
there is an index $j \leq  h$ such that $\alpha_j = 1$, and all the other $\alpha_i$ are 0.
In other words, we have $w=v_j$.

ii) The proof of ii) is absolutely similar to the proof of i).

Remark that there are no conditions on $w_{h+2e+1}, \hdots, w_n$ except that they form together with 
$w_1, \hdots, w_{h+2e}$ a basis of $\bbR^n$.

iii) As for the converse statement, assume that the basis $w_1, \hdots, w_n$ of $\bbR^n$ satisfies the above
conditions i) and ii).

Notice that the change $(v_{h+1},v_{h+2}) \mapsto (v_{h+1},-v_{h+2})$ can be achieved by the permutation 
$(x_{h+1},x_{h+2}) \mapsto (x_{h+2},x_{h+1})$ of the coordinates $x_{h+1}$ and $x_{h+2}$. So, without 
loss of generality, we can assume now that $w_i = v_i$ for all $i= 1, \hdots, h + 2e$. We can write
\begin{equation}
w_{h+2e+i}= \sum_{j=1}^{h+2e} a_{ij}v_j + \sum_{j=1}^{n-h-2e} b_{ij}v_{h+2e+j}
\end{equation} 
for each $i = 1, \hdots, n-h-2e$, where $(b_{ij})$ is an invertible matrix.

Putting $\widetilde{w}_{h+2e+i}= \sum_{j=1}^{n-h-2e} c_{ij}w_{h+2e+j}$ where $(c_{ij})$ is the inverse matrix
of $(b_{ij})$, we have
\begin{equation}
\widetilde{w}_{h+2e+i}= \sum_{j=1}^{h+2e} \widetilde{a}_{ij}v_j +v_{h+2e+i}
\end{equation}
for some $\widetilde{a}_{ij}$, and $w_{h+2e+i} = \sum_{j=1}^{n-h-2e} b_{ij}\widetilde{w}_{h+2e+j}$.

The vector fields 
\begin{equation}
X_{\widetilde{w}_{h+2e+i}} = \frac{\partial}{ \partial x_{h+2e+i}} + \hdots
\end{equation}
are regular vector fields in a neighborhood of $p$, which commute with each other.

Define the new functions $y_1, \hdots, y_{h+2e}$ in the neighborhood of $p$ as follows:
\begin{equation}
y_i(q) = x_i(\varphi_{X_{\widetilde{w}_{h+2e+1}}}^{-x_{h+2e+1}(q)} \circ \varphi_{X_{\widetilde{w}_{h+2e+2}}}^{-x_{h+2e+2}(q)}  \circ \hdots \circ \varphi_{\widetilde{w}_{w_n}}^{-x_n(q)} (q)).
\end{equation}

In other words, move $q$ by the flows of $X_{\widetilde{w}_{h+2e+i}}, \hdots, X_{\widetilde{w}_n}$ to a point $q'$ on the subspace 
$\{x_{h+2e+1} = \hdots = x_n = 0\}$, and then put $y_i(q) = x_i(q')$.

Then in the new coordinate system $(y_1, \hdots,y_{h+2e}, x_{h+2e+1}, \hdots,x_n)$ the vector fields 
$X_{w_1}, \hdots, X_{w_{h+2e}}$ still have the same expression as before:
\begin{equation}
\begin{cases}
X_{w_i}  = y_i\frac{\partial}{\partial y_i} \quad (\forall i = 1, \hdots, h) \\
X_{w_{h+2i-1}}= y_{h+2i-1}\frac{\partial }{\partial y_{h+2i-1}} 
                  +  y_{h+2i}\frac{\partial }{\partial y_{h+2i}}\\
X_{w_{h+2i}}= y_{h+2i-1}\frac{\partial }{\partial y_{h+2i}} 
               -  y_{h+2i}\frac{\partial }{\partial y_{h+2i-1}} \quad \forall \quad i = 1,\hdots, e.
\end{cases}
\end{equation}
Moreover, in this new coordinate system $(y_1, \hdots,y_{h+2e}, x_{h+2e+1}, \hdots,x_n)$ the 
vector fields $X_{\widetilde{w}_{h+2e+i}}$ are rectified, i.e. we have 
$$X_{\widetilde{w}_{h+2e+i}} = \frac{\partial}{\partial x_{h+2e+i}} \quad \forall i = 1, \hdots, n-h-2e.$$
In this coordinate system, the vector fields $X_{w_{h+2e+i}}$ are also constant:
$$X_{w_{h+2e+i}} = \sum_{j=1}^{n-h-2e} b_{ij}\frac{\partial}{ \partial x_{h+2e+i}}.$$
In order to write $X_{w_{h+2e+i}} = \frac{\partial}{ \partial y_{h+2e+i}}$, it remains to make
the following linear change of the coordinates $(x_{h+2e+1}, \hdots,x_n)$:
$$y_{h+2e+i} = \sum_{j=1}^{n-h-2e}c_{ij}x_{h+2e+i}.$$ 
The new coordinate system $(y_1, \hdots,y_n)$ is now a canonical coordinate system associated 
to the basis $(w_1, \hdots,w_n)$.
\end{proof}
\begin{remark}
 The fact that the last vectors (from $w_{h+2e+1}$ to $w_n$) in an adapted basis can be arbitrary 
(provided that they form together with $w_1, \hdots, w_{h+2e}$ a basis of $\bbR^n$)
is very important in the global picture, because it allows us to glue different local canonical 
pieces together in a flexible way.
\end{remark}

As a simple corollary of the local normal form theorem, we have the following preliminary description
of the set of all singular points of a nondegenerate action:
\begin{cor} 
Denote by 
\begin{equation} 
\cS = \{x \in M^n \ |\ \rank x < n\}
\end{equation} 
the set of singular points of a nondegenerate action $\rho: \bbR^n \times M^n \to M^n$. Then we have:

i) $\cS$ is a stratified manifold, where the smooth strata are 
\begin{equation} 
\cS_{h,e} = \{x \in M^n \ |\ \text{HE-invariant of } x \text{ is } (h,e)\}
\end{equation} 
given by the HE-invariant.

ii) $\dim \cS_{h,e} = n -h-2e$ if $\cS_{h,e} \not= \emptyset$.

iii) If $\cS \not= \emptyset$ then $\dim \cS  = n-1$ or $\dim \cS  = n-2$.
When there are hyperbolic singularities (points with $h>0$) then $\dim \cS  = n-1$, and 
when there are only elbolic singularities ($h=0$ for every point) then  $\dim \cS  = n-2$. 
\end{cor}
\begin{proof}
The proof is straightforward.
\end{proof}

\begin{defn} \label{def:AssociatedVector}
 1) If $\cO_p$ is a singular orbit of corank 1 of a nondegenerate action $\rho: \bbR^n \times M^n \to M^n$,
i.e. the HE-invariant of $\cO_p$ is $(1,0)$, 
then the unique vector $v \in \bbR^n$ such that the corresponding generator $X_v$ of $\rho$ can be written as 
 $X_v = x\frac{\partial}{ \partial x}$ near each point of $\cO_p$ is called the {\bf associated vector} of $\cO_p$.

2) If $\cO_p$ is a singular orbit of HE-invariant $(0,1)$ (i.e. corank 2 transversally elbolic) then the couple of vectors 
$(v_1, \pm v_2)$ in $\bbR^n$, where $v_2$ is determined only up to a sign, such that $X_{v_1}$ and $X_{v_2}$ can be locally written as 
$$\begin{cases}
   X_{v_1} = x\frac{\partial}{ \partial x} + y\frac{\partial}{ \partial y}\\
 X_{v_2} = x\frac{\partial}{ \partial y} - y\frac{\partial}{ \partial x}
\end{cases}$$
is called the {\bf associated vector couple} of $\cO_p$.

\end{defn}

\subsection{Local automorphism groups and the reflection principle}

\begin{thm}[Local automorphism groups] \label{thm:LocalAutomorphism}
 Let $p$ be a nondegenerate singular point of HE-invariant $(h,e)$ and 
rank $r$ of an action $\rho: \bbR^n \times M^n \to M^n$ $(n= h + 2e + r)$. Then the group 
of germs of local isomorphisms (= local diffeomorphisms which preserve the action) which fix the point $p$ 
is isomorphic to  $\bbT^e \times \bbR^{e+h} \times (\bbZ_2)^h$. The part $\bbT^e \times \bbR^{e+h}$
of this group comes from the action $\rho$ itself (internal automorphisms given by the action of the isotropy 
group of $\rho$ at $p$).
\end{thm}

\begin{proof}
Using the local normal form theorem, it is clear that for any $w \in Z_\rho(p)$, where
\begin{equation} 
Z_\rho(p) = \{ v \in \bbR^n \ |\ \rho(v,p) = p\}
\end{equation} 
denotes the isotropy group of $\rho$ at $p$, the map $\rho(w,.)$ fixes the point $p$ and
preserves the action, and (the germ of) this map is identity if and only if $w$ belongs 
to the isotropy group
\begin{equation}
Z_\rho(\cU) = \{ v \in \bbR^n \ |\ \rho(v,.) = Id_{\cU}\}
 \end{equation} 
 of $\rho$ in a neighborhood $\cU$ of $p$.
 Thus we have a natural inclusion of 
\begin{equation}
Z_\rho(p)/Z_\rho(\cU)\cong \bbT^e \times \bbR^{e+h} 
 \end{equation} 
into the group of germs of local automorphisms which fix $p$.

Let $(x_1, \hdots,x_n)$ be a canonical coordinate system at $p$
with respect to $\rho$. Then for each $i = 1, \hdots, h$, the involution 
\begin{equation}
\sigma_i : (x_1, \hdots,x_i, \hdots,x_n) \mapsto (x_1, \hdots,-x_i, \hdots,x_n)  
\end{equation} 
is also a local automorphism of the action and $\sigma_i(p) = p$.

The involutions $\sigma_i$ commute with each other and generate an Abelian group isomorphic 
to $(\bbZ_2)^h$. The elements of this group do not come from $Z_\rho(p)/Z_\rho(\cU)$,
and things commute, so together we get a group isomorphic to $\bbT^e \times \bbR^{e+h}\times (\bbZ_2)^h$
of germs of automorphisms.

It remains to show that any (germ of) local automorphism which fixes $p$ is an element of this group.
Indeed, let $\varphi : (\cU,p) \to (\cU,p)$ be a local diffeomorphism which fixes $p$ and
preserves the action. The corner 
\begin{equation}
\cU_+ = \{(x_1, \hdots,x_n) \in \cU \ |  \ x_1 > 0, \hdots, x_h>0\}
\end{equation} 
is a local regular orbit of $\rho$. If $\varphi $ does not preserve this corner, i.e. it sends this corner
to another corner, say for example $\{(x_1, \hdots,x_n) \in \cU \ |\ x_1 < 0, x_2 < 0, x_3 > 0, \hdots, x_h>0\}$,
then $\sigma_1\circ \sigma_2 \circ   \varphi $ preserves the positive corner. To prove that $\varphi $ belongs to the 
above group is equivalent to prove that $\varphi $ composed with some involutions 
$\sigma_i \quad (i = 1, \hdots, h)$ belongs to the above group. 
So without loss of generality we can assume that $\varphi $ preserves the positive corner.

Let $z \in \cU_+$ be a point in the positive corner near $p$. Then $\varphi(z) \in \cU_+$, which implies the existence 
of an element $w \in \bbR^n$ such that $\rho(w,z) = \varphi(z)$. 

Put $\widehat{\varphi} = \rho(-w,.)\circ \varphi $. Then $\widehat{\varphi}(z) = z$. Since $\widehat{\varphi}$ is an automorphism,
it implies that $\widehat{\varphi}$ is identity on the whole corner $\cU_+$. ($\forall y \in \cU_+$ we can write $y =  \rho(v,z)$
and hence $\widehat{\varphi}(y) = \widehat{\varphi}(\rho(v,z)) = \rho(v,\widehat{\varphi}(z))= \rho(v,z) = y$).
Now, for any element $z'$ in any other corner of $\cU$, we will also have $\varphi(z') =z'$, because if not we would have 
$\widehat{\varphi} = \rho(v,.)$ is a linear map different from identity in that corner, which would imply that $\varphi $
is \emph{not} differentiable at $p$.

Thus $\widehat{\varphi}$ is identity, not only in the corner $\cU_+$, but in a neighborhood of $p$, and we can write 
$\varphi = \rho(w,.)$ in a neighborhood of $p$. Remark that $w \in Z_\rho(p)$, otherwise $\varphi $ would not fix $p$.
\end{proof}

The finite automorphism group $(\bbZ_2)^k$ in the above theorem acts not only locally in the neighborhood 
of a singular point $p$ of HE-invariant $(h,e)$, but also in the neighborhood of a smooth closed manifold 
of dimension $n-h-2e$ which contains $p$. More precisely, we have the following {\bf reflection principle}, 
which is somewhat similar to the Schwartz reflection principle in complex analysis:

\begin{thm}[Reflection principle] \label{thm:Reflection}
a) Let $p$ be a point of HE-invariant $(1,0)$ of a nondegenerate $\bbR^n$-action $\rho$ 
on a manifold $M^n$ without boundary. 
Denote by $v \in \bbR^n$ the associated vector of $p$ (i.e. of the orbit $\cO_p$) as in 
Definition \ref{def:AssociatedVector}. Put
\begin{equation} 
\cN_v = \{q \in M^n \ |\ X_v(q) = 0 \text{ and } X_v \text{ can be written as  } 
x_1\frac{\partial}{\partial x_1} \text{ near }  q\}.
\end{equation} 
Then $\cN_v$ is a smooth embedded hypersurface of dimension $n-1$ of $M^n$ 
(which is not necessarily connected), and there is a unique non-trivial involution 
$\sigma_v : \cU(\cN_v) \to \cU(\cN_v)$ from a neighborhood of $\cN_v$ to itself which preserves the 
action $\rho$ and which is identity on $\cN_v$.

b) If the HE-invariant of $p$ is $(h,0)$ with $h>1$, then
we can write 
\begin{equation} 
p \in \cN_{v_1,\hdots,v_h} = \cN_{v_1}\cap \hdots \cap \cN_{v_h}
\end{equation} 
where $\cN_{v_i}$ are defined as in a), $(v_1,\hdots,v_h)$ is a free family of vectors in $\bbR^n$,
the intersection $\cN_{v_1}\cap \hdots \cap \cN_{v_h}$ is transversal and $\cN_{v_1,\hdots,v_h}$
is a closed smooth submanifold of codimension $h$ in $M$. The involutions $\sigma_{v_1},\hdots, \sigma_{v_h} $
generate a group of automorphisms of $(\cU(\cN_{v_1,\hdots,v_h}), \rho)$ isomorphic to $(\bbZ_2)^h$.
\end{thm}

\begin{proof}
It follows easily from Theorem \ref{thm:NormalForm}, Theorem \ref{thm:CanonicalCoordinate} and 
Theorem \ref{thm:LocalAutomorphism}.
\end{proof}

\subsection{Nondegenerate singular orbits}

Consider an orbit $\cO_z = \{ \rho(t,z) \ |\ \ t \in \bbR^n \}$ though a point $z \in M^n$
of a given $\bbR^n$ action $\rho$. Since $\cO_z$ is a quotient of $\bbR^n$, 
it is diffeomorphic to $\bbR^k \times \bbT^l$ for some
nonnegative integers $k,l \in \bbZ_+$. The couple $(k,l)$ will be called the {\bf RT-invariant}
of $z$ or of the orbit $\cO_z$. The sum $k+l$ is the dimension of $\cO_z$.

\begin{defn} \label{defn:HERT-invariant}
The {\bf HERT-invariant} of an orbit $\cO_q$ or a singular point $q$ on it is
the quadruple $(h,e,r,t)$, where $h$ is the number of transversal hyperbolic components, $e$ is the
number of transversal elbolic components, and $\bbR^r \times \bbT^t$ is the diffeomorphism type of the orbit.
\end{defn}
An orbit is compact if and only if $r = 0$, in which case it is a torus of dimension $t$.

\begin{prop}[Semi-continuity of the T-invariant] \label{prop:T-semicontinuity}
If $\cK \cong \bbR^{r(\cK)} \times \bbT^{t(\cK)}$ and $\cH \cong \bbR^{r(\cH)} \times \bbT^{t(\cH)}$ 
 are two different orbits of $(M^n,\rho)$ such that $\cK \subset \bar \cH$, then $t(\cK) \leq t(\cH)$. 
\end{prop}
\begin{proof}
 Remark that 
\begin{equation} 
t(\cK) = \rank_\bbZ (Z_\rho(\cK)/I_\rho(\cK) )
\end{equation} 
where 
\begin{equation} 
Z_\rho(\cK) = \{w \in \bbR^n \ |\ \rho(w,.)|_\cK = Id_\cK\}
\end{equation} 
 is the isotropy group of $\rho$ on $\cK$, and
\begin{equation} 
I_\rho(\cK) = \{w \in \bbR^n \ |\ X_w = 0 \text{ on } \cK\}
\end{equation} 
 is the isotropy group of the infinitesimal action on $\cK$.

In order to show $t(\cK) \leq t(\cH)$, it is enough to show that 
$2Z_\rho(\cK)/I_\rho(\cK) $ is a subgroup of a quotient group of $Z_\rho(\cH)/I_\rho(\cH)$. 

Let $p \in \cK, q \in \cH$ near $p$, $w \in Z_\rho(\cK), w \neq 0$.

Since $\rho(w,p) = p$, we have that $\rho(w,q)$ is close to $q$ and belongs to $\cH$. 
To avoid possible ``twistings'' due to the $(\bbZ_2)^h$ symmetry group as in Theorem \ref{thm:LocalAutomorphism},
look at $\rho(2w,q)$ instead of  $\rho(w,q)$. In any case
$\rho(2w,q)$ lies in the same ``local orbit'' as $q$, and there is an element $\theta_q \in \bbR^n$ close
to $I_\rho(\cK)$ such that 
$$\rho(2w,q)= \rho(\theta_q,q),$$
which implies that $2w - \theta_q \in Z_\rho(\cH)$, or 
$$2w - \theta_q \mod I_\rho(\cK) \in Z_\rho(\cH) \mod I_\rho(\cK).$$
This is true for all $k \in Z_\rho(\cK)$, so we have
$$2Z_\rho(\cK)/I_\rho(\cK) \subset  Z_\rho(\cH)/(I_\rho(\cK) \cap Z_\rho(\cH))$$ 
Notice that $I_\rho(\cK) \supset I_\rho(\cH)$ by continuity, so 
$Z_\rho(\cH)/(I_\rho(\cK) \cap Z_\rho(\cH))$ is a quotient group of $Z_\rho(\cH)/I_\rho(\cH)$.
\end{proof}

We have the following linear model for a tubular neighborhood of a compact orbit of 
HERT-invariant $(h,e,0,t)$:

\begin{itemize}
\item The orbit is
\begin{equation}
\{0\} \times \{0\} \times \bbT^t / (\bbZ_2)^k,
\end{equation}
which lies  in 
\begin{equation}
B^{h} \times B^{2e} \times \bbT^t / (\bbZ_2)^k,
\end{equation}
(where $B^h$ is a ball of dimension $h$), with coordinates $(x_1, \hdots,x_{h+2e})$ 
on $\bbR^{h} \times \bbR^{2e}$ and  $(z_1, \hdots,z_t) \mod 2\pi$ 
on $\bbT^t$, and $k$ is some nonnegative integer such that $k\leq  \min(h,t)$.
\item The (infinitesimal) action of $\bbR^n$   is generated by the vector fields 
\begin{equation}
\begin{cases}
Y_i = x_i\frac{\partial }{\partial x_i} \quad \forall \quad i = 1, \hdots, h \\
Y_{h+2j-1} = x_{h+2j-1}\frac{\partial }{\partial x_{h+2j-1}} +  x_{h+2j}\frac{\partial }{\partial x_{h+2j}}  \\
Y_{h+2j} = x_{h+2j-1}\frac{\partial }{\partial x_{h+2j}} -  x_{h+2j}\frac{\partial }{\partial x_{h+2j-1}} \quad \forall \quad j = 1, \hdots, e \\
Y_{h+2e+i} = \frac{\partial }{\partial z_i} \quad \forall \quad i = 1, \hdots, t
\end{cases}
\end{equation}
like in the local normal form theorem.
  \item The Abelian group $(\bbZ_2)^k$ acts on $B^h \times B^{2e} \times \bbT^t $ freely,
 component-wise, and by isomorphisms of the action, so that the quotient is still a manifold with
 an induced action of $\bbR^n$ on it. The action of $(\bbZ_2)^k$ on $B^{h}$ is 
by an injection from $(\bbZ_2)^k$ to the involution group $(\bbZ_2)^h$ generated by the
reflections $\sigma_i: (x_1,\hdots,x_i,\hdots, x_h) \mapsto (x_1,\hdots, -x_i,\hdots, x_h)$, its
action on $B^{2e}$ is trivial, and its action on $\bbT^t$ is via an injection of $(\bbZ_2)^k$
into the group of translations on $\bbT^t$.
\end{itemize}

\begin{thm}[Semi-local  normal form for compact orbits] \label{thm:semi-localform}
Any compact orbit of a nondegenerate action
$\rho: \bbR^n \times M^n \to M^n$ can be linearized, i.e. there is a tubular neighborhood of it
which is, together with the action $\rho$, isomorphic to a linear model described above.
\end{thm}

\begin{proof}(Sketch). Let $q \in \cO_q$ be a point on a compact orbit of HERT-invariant $(h,e,0,t)$.
It follows from the local normal form theorem that there exists a local submanifold $N$ transverse
to $\cO_q$ at $q$, $N \cap \cO_q = \{q\}$, such that $N$ is tangent to the vector field $X_v$ for any
\begin{equation}
v \in I_\rho(p):= \{v \in \bbR^n \ |\ X_v(p) = 0\}.
\end{equation}

Similarly to the proof of Proposition \ref{prop:T-semicontinuity}, we can find $t$ linearly independent vectors
$w_1, \hdots, w_t \in \bbR^n$ such that $X_{w_1}(p), \hdots, X_{w_t}(p)$ are also linearly independent, and 
$X_{w_1}, \hdots, X_{w_t}$ generate a locally free $\bbT^e$-action in a neighborhood $\cU(\cO_p)$ of $\cO_p$,
which is free almost everywhere. (The existence of $X_{w_1}, \hdots, X_{w_t}$ can also be seen from Theorem 
\ref{thm:HERT-toricdegree} below). Notice that $\dim N + t = \dim N + \dim \cO_p = n$,
and the action of $\bbT^t$ generated by $X_{w_1}, \hdots, X_{w_t}$ is transversal to $N$.

In the case when this $\bbT^t$-action is free, we can decompose $\cU(\cO_p)$ into a direct product $N \times \bbT^t$
by viewing  $\cU(\cO_p)$ as a trivial principal $\bbT^t$-bundle with base $N$. Then the action $\rho$ also splits
in $\cU(\cO_p)$ into a direct sum of an action on $N$  and an action
by translations on $\bbT^t$. On $N$, we have the local canonical coordinates given by the local 
normal form theorem. On $\bbT^t$, we have periodic
coordinates with respect to which the vector fields $X_{w_1}, \hdots, X_{w_t}$ form a standard basis of constant
vector fields. Putting these coordinates together, we have a linearization of the action on $N \times \bbT^{t}$.
If the $\bbT^t$-action is not free, then we can make it into 
a free action by taking a normal $(\bbZ_2)^k$-covering of $\cU(\cO_p)$ for some $1\leq  k \leq h$, then proceed
as in the free case. The theorem is proved.
\end{proof}

\begin{remark}
The above theorem is analogous to Miranda-Zung's linearization theorem for tubular neighborhoods 
of compact nondegenerate singular orbits of integrable Hamiltonian systems \cite{MirandaZung-NF2004}.
\end{remark}

More generally, for any point $q$ lying in a orbit $\cO_q$ of HERT-invariant $(h,e,r,t)$ which is not necessarily
compact (i.e. the number $r$ may be strictly positive), we still have the following linear model:

\begin{itemize}
\item The intersection of the orbit with a tubular neighborhood is
\begin{equation}
 \{0\} \times \{0\} \times \bbT^t / (\bbZ_2)^k \times B^r,
 \end{equation}
which lies in 
\begin{equation}
(B^{h} \times B^{2e} \times \bbT^t / (\bbZ_2)^k) \times B^r
\end{equation}
 with coordinates $(x_1, \hdots,x_{h+2e})$ on $B^{h} \times B^{2e}$, 
$(z_1, \hdots,z_t) \mod 2\pi$ on $\bbT^t$, and $\zeta_1,\hdots, \zeta_r$ on $B^r$
 and $k$ is some nonnegative integer such that $k\leq  \min(h,t)$.
\item The action of $\bbR^n$   is generated by the vector fields 
\begin{equation}
\begin{cases}
Y_i = x_i\frac{\partial }{\partial x_i} \quad \forall \quad i = 1, \hdots, h, \\
Y_{h+2j-1} = x_{h+2j-1}\frac{\partial }{\partial x_{h+2j-1}} +  x_{h+2j}\frac{\partial }{\partial x_{h+2j}}  \\
Y_{h+2j} = x_{h+2j-1}\frac{\partial }{\partial x_{h+2j}} -  x_{h+2j}\frac{\partial }{\partial x_{h+2j-1}} \quad \forall \quad j = 1, \hdots, e, \\
Y_{h+2e+i} = \frac{\partial }{\partial z_i} \quad \forall \quad i = 1, \hdots, t, \\
Y_{h+2e+t+i} = \frac{\partial }{\partial \zeta_i} \quad \forall \quad i = 1, \hdots, r.
\end{cases}
\end{equation}
  \item The Abelian group $(\bbZ_2)^k$ acts on $\bbR^h \times \bbR^{2e} \times \bbT^t $ freely
in the same way as in the case of a compact orbit.
\end{itemize}

\begin{thm}[Semi-local  normal form] \label{thm:semi-localform2}
Any point $q$ of any HERT-invariant $(h,e,r,t)$ with respect to a nondegenerate action 
$\rho: \bbR^n \times M^n \to M^n$ admits a neighborhood which is isomorphic to a linear
model described above.
\end{thm}

\begin{proof}
Theorem \ref{thm:semi-localform2} is simply a parametrized version of Theorem \ref{thm:semi-localform}. It can
also be seen as a corollary of Theorem \ref{thm:semi-localform}, by assuming that the point $q$ lies in a linear
model of a neighborhood of a compact orbit. (If $q$ is far from compact orbits, we can move it by a 
map $\rho(v,.)$ of the action for some appropriate $v \in \bbR$
to a point $\rho(v,q)$ which is close enough to a compact orbit. A model for a neighborhood of $\rho(v,q)$
will become a model for a neighborhood of $q$ by taking the inverse map $\rho(-v,.)$).
\end{proof}

\begin{remark}
 The difference between the compact case and the noncompact case is that, when $\cO_q$ is a compact orbit,
we have a linear model for a whole tubular neighborhood of it, but when $\cO_q$ is noncompact we have
a linear model only for a neighborhood of a ``stripe'' in $\cO_q$.
\end{remark}

To be more precise, the (minimal required) group $(\bbZ_2)^k$ in Theorem \ref{thm:semi-localform}
and Theorem \ref{thm:semi-localform2}  is naturally isomorphic to the group
\begin{equation} \label{eqn:TwistingGroup}
 G_q = (Z_\rho (q) \cap (Z_\rho \otimes  \bbR))/Z_\rho.
\end{equation}

\begin{defn} \label{defn:TwistingGroup}
The group $G_q$ defined by the above formula is called the {\bf twisting group} of the action $\rho$ at $q$ (or at
the orbit $\cO_q$). The orbit $\cO_q$ is said to be {\bf non-twisted} (and $\rho$ is said to be non-twisted at $q$) 
if $G_q$ is trivial, otherwise it is said to be {\bf twisted}.
\end{defn}

\begin{remark}
 The twisting phenomenon also appears in ``real-world'' integrable Hamiltonian systems coming from physics 
and mechanics, and it was observed, for example, by Fomenko and his collaborators in their study of 
integrable Hamiltonian systems with 2 degrees
of freedom. See, e.g., \cite{BolsinovFomenko-IntegrableBook}.
\end{remark}

\begin{prop}[HERT-invariant of adjacent orbits] \label{prop:adjacent-orbits}
1) If $\cO_p$ is an orbit of HERT-invariant $(e,h,r,t)$ with $r> 0$ and $M^n$ is compact, then there is an
orbit of HERT-invariant $(e,h+1,r-1,t)$ or $(e+1,h,r-1,t-1)$ lying in the closure $\bar \cO_p$ of $\cO_p$. 

2) If $M^n$ is compact, then the closure of any orbit contains a compact orbit, i.e. an orbit with $r =0$.

3) If an orbit $\cO_p$ has HERT-invariant $(e,h,r,t)$ with $e \geq 1$, then there is an orbit 
$\cO_q$ of HERT-invariant $(e-1,h+1,r,t+1)$ such that $\cO_p \subset \bar \cO_q$.

4) If an orbit $\cO_p$ has HERT-invariant $(e,h,r,t)$ with $h \geq 1$, then there is an orbit $\cO_q$
of HERT-invariant $(e,h-1,r+1,t)$ such that $\cO_p \subset \bar \cO_q$.

5) Any orbit lies in the closure of a regular orbit, i.e. an orbit of dimension $n$.
\end{prop}

\begin{proof}
 The proof follows directly from the previous results and arguments of this section.
\end{proof}

\section{The toric degree}
\subsection{Definition and determination of toric degree}
Let $\rho: \bbR^n \times M^n \to M^n$ be a smooth action of $\bbR^n$ on a $n$-dimensional 
manifold $M^n$. As before, we will denote by 
\begin{equation}
Z_{\rho} = \{ g \in \bbR^n : \rho(g,.) = Id_{M^n}\}
\end{equation}
the isotropy group of $\rho$ on $M^n$. Since $\rho$ is locally free almost everywhere due to 
its nondegeneracy,
$Z_{\rho}$ is a discrete subgroup of $\bbR^n$, so we have 
\begin{equation}
Z_{\rho} \cong  \bbZ^k.
\end{equation}

The action $\rho$ of $\bbR^n$ descends to an action of 
\begin{equation}
\bbR^n / Z_{\rho} \cong \bbT^k \times \bbR^{n-k}
\end{equation}
on $M$, which we will also denote by $\rho$:
\begin{equation}
\rho : (\bbR^n/Z_\rho) \times M^n \to M^n
\end{equation}

We will also denote by 
\begin{equation} \label{eqn:rhoT-action}
\rho_\bbT: \bbT^k \times M^n \to M^n
\end{equation}
the subaction of $\rho$ given by the subgroup $\bbT^k \subset \bbT^k \times \bbR^{n-k} \cong \bbR^n/Z_\rho$. 
More precisely, $\rho_\bbT$ is an action of $(Z_{\rho}\otimes \bbR)/Z_{\rho}$ on $M^n$, 
which becomes a $\bbT^k$-action after an isomorphism from $(Z_{\rho}\otimes \bbR)/Z_{\rho}$ to $\bbT^k$.

\begin{defn}
The number $k = \rank_\bbZ Z_{\rho}$ is called the {\bf toric degree} of the action $\rho$.
\end{defn}

Clearly, the toric degree of $\rho$ is also the maximal number such that the action $\rho$
descends to an action of $\bbT^k \times \bbR^{n-k}$ on $M^n$.
It can be viewed as the maximal number $k$ such that $\rho$ contains an effective action of 
$\bbT^k$ as its subaction.

If the toric degree is 0 then we say that the action is {\bf totally hyperbolic}. Totally hyperbolic actions
will be studied in Section \ref{section:hyperbolic}. It seems that there are no obstructions for a closed manifold 
to admit totally hyperbolic actions (see Theorem \ref{thm:hyperbolic_dim2} and Conjecture \ref{conjecture:hyperbolic}).
But starting from $k \geq 1$, there are strong topological obstructions for a $n$-manifold to admit a nondegenerate
$\bbR^n$-action of toric degree $k$. This leads us to the following definition:

\begin{defn} \label{defn:toricrank}
 We say that a  manifold $M^n$ has {\bf toric rank} equal to $k$ if $k$ is the maximal number such that
$M$ admits a nondegenerate $\bbR^n$-action of toric degree $k$.
\end{defn}

For example, as will be seen from Subsection \ref{subsection:n_and_n-1} and 
Subsection \ref{subsection:hyperbolic_existence}, it is easy to show that, the toric rank of
$\bbT^2$ is equal to 2, the toric rank of $\bbS^2, \bbR \bbP^2$ and the Klein bottle is equal
to 1, and the toric rank of any other closed 2-dimensional surface is 0.

If $M^n$ has toric degree $k$, then in particular it must admit an effective action of $\bbT^k$.
This condition is a rather strong topological condition. For example, Fintushel \cite{Fintushel-Circle1977} 
showed (modulo Poincaré's conjecture which is now a theorem) that among
simply-connected 4 manifolds, only the manifolds $\bbS^4, \bbC \bbP^2,-\bbC \bbP^2,\bbS^2 \times \bbS^2$ and their connected sums
admit an effective locally smooth $\bbT^1$-action (and so only these manifolds may have toric degree $\geq 1$). 
This list is the same as the list of simply-connected 4-manifolds admitting an effective $\bbT^2$-action,
according to Orlik and Raymond \cite{OR_Torus1}, \cite{OR_Torus2}. A classification of non-simply-connected
4-manifolds admitting an effective $\bbT^2$-action can be found in Pao \cite{Pao-TorusAction1}.

\begin{remark}
 An interesting invariant closely related to toric rank is the so-called Milnor's rank of a manifold, see, e.g. 
\cite{Camacho-Foliations1985}: it is the maximal nonnegative integer $k$ such that the manifold 
admits a  free infinitesimal $\bbR^k$-action, or in other words, a $k$-tuple
of commuting vector fields which are linearly independent everywhere.
\end{remark}
 
We observe that the toric degree can be read off the HERT-invariant of any point on $M$
with respect to the action. More precisely, we have:
\begin{thm}[Toric degree] \label{thm:HERT-toricdegree}
Let $\rho: \bbR^n \times M^n \to M^n$ be a nondegenerate smooth action of $\bbR^n$ on a $n$-dimensional 
manifold $M^n$ and $p \in M$ be an arbitrary point of $M$. If the HERT-invariant of $p$ with respect to $\rho$
is $(h,e,r,t)$, then the toric degree of $\rho$ on $M$ is equal to $e + t$.
\end{thm}

\begin{proof}
We will divide the proof of the theorem into several steps.

\underline{Step 1}: \emph{ Let $p \in M$ be a regular point. Then $\text{toric degree}(\rho) \leq t(p)$}.

Indeed, the orbit $\cO_p$ is of the type $\bbT^{t(p)}\times \bbR^{r(p)}$
and can be viewed as an orbit of an action of $\bbT^k\times \bbR^{n-k}$,
where $k = \text{toric degree}(\rho)$. Since $t(p) + r(p) = n$, the isotropy 
group of the action of $\bbT^k\times \bbR^{n-k}$ on $\cO_p$ is a discrete group. 
It follows immediately that $k \leq t(p)$.

\underline{Step 2}: \emph{ If $\cO_1$ and $\cO_2$ are two arbitrary different 
regular orbits then $Z_\rho(\cO_1) = Z_\rho(\cO_2)$, where $Z_\rho(\cO) \subseteq \bbR^n$
denotes the isotropy group of $\rho$ on $\cO$.} 

By connectedness, it is enough to prove the above statement in the case when
$\bar \cO_1 \cap \bar \cO_2 \not= \emptyset$, where $\bar \cO$ denotes the closure of $\cO$.
Notice that if $\cO_1 \not= \cO_2$ and $\bar \cO_1 \cap \bar \cO_2 \not= \emptyset$
then $\bar \cO_1 \cap \bar \cO_2$ must contain a singular point of rank $n-1$, because
the set of singular points of rank $\leq n-2$ in $M$ is of dimension $\leq n-2$ and
can not separate $\cO_1$ from $\cO_2$. So let  $q \in \bar \cO_1 \cap \bar \cO_2$ be
a singular point of rank $n-1$ and corank 1. Then $q$ is automatically a hyperbolic 
singular point, i.e. $h(q) = 1$ and $e(q) = 0$, because $h(q) + 2e(q) = 1$. 

Consider a canonical coordinate system $(x_1,\hdots,x_n)$ in a neighborhood $\cU$ of $q$
in $M^n$:
\begin{equation}
X_{v_1} = x_1 \frac{\partial}{\partial x_1},
X_{v_2} =  \frac{\partial}{\partial x_2}, \hdots, X_{v_n} =  \frac{\partial}{\partial x_n},
\end{equation}
where $(v_1,\hdots, v_n)$ is a basis of $\bbR^n$. Locally $\cO_1$ and $\cO_2$ lie on the two different 
sides of the singular orbit $\cO_q$, so we can assume that 
$\cO_1 \cap \cU = \{(x_1,\hdots, x_n) \in U \ |\ x_1 < 0\}$ and 
$\cO_2 \cap \cU = \{(x_1,\hdots, x_n) \in U \ |\ x_1 > 0\}$. Let $w \in Z_\rho(\cO_1)$. 
Since the map $\rho(w,.)$ is identity on $\cO_1$, the differential of $\rho(w,.)$ at $q$ is also identity. 
It implies that, for any point
$p_{\epsilon} = (\epsilon,0,\hdots,0) \in \cO_2 \cap \cU$, $\rho(w,p_{\epsilon})$ is close enough to 
$p_{\epsilon}$ so that there is an element $\theta_{\epsilon} \in \bbR^n$ close to zero such that 
$\rho(w,p_{\epsilon}) = \rho(\theta_{\epsilon},p_{\epsilon})$. A-priori, $\theta_{\epsilon}$ may depend
on $\epsilon$, but $\theta_{\epsilon} \to 0$ when $\epsilon \to 0+$. Note that the equality 
$\rho(w,p_{\epsilon}) = \rho(\theta_{\epsilon},p_{\epsilon})$ implies that
$w - \theta_{\epsilon} \in Z_\rho (\cO_2)$, so by taking the limit when $\epsilon$ tends to 0, we have
$w = \lim_{\epsilon \to 0} (w - \theta_{\epsilon}) \in Z_\rho (\cO_2)$. Since  $w \in Z_\rho (\cO_2)$
for any $w \in Z_\rho(\cO_1)$, we have $Z_\rho(\cO_1) \subset Z_\rho(\cO_2)$. By symmetry of arguments, 
the inverse inclusion is also true, i.e. we have $Z_\rho(\cO_1) = Z_\rho(\cO_2)$.

\underline{Step 3}: \emph{$Z_\rho = Z_\rho(\cO)$ for any regular orbit $\cO$. In particular, for any regular point $p$, 
 the toric rank of $\rho$ is equal to $t(p)$ (and $e(p) = h(p) = 0$).} 

Indeed, the inclusion $Z_\rho \subset Z_\rho(\cO)$ is true for any orbit (singular or regular). To prove the inverse inclusion,
let $w \in Z_\rho(\cO)$ be an element of the isotropy group of a regular orbit. Then according to Step 2), the isotropy
group of any other regular orbit also contains $w$. It means that $\rho(w,.)$ is identity on the set of regular points of $M^n$.
Since this set is dense in $M^n$, by continuity we have that $\rho(w,.)$ is identity of $M^n$, i.e. $w \in Z_\rho$.

Remark that if $p$ is a regular point then $t(p) = \dim_\bbZ Z_\rho(\cO_p)$, and since $Z_\rho(\cO_p) = Z_\rho$,
we have that $t(p) = \dim Z_\rho$ is the toric degree of $\rho$.

\underline{Step 4}: \emph{If $q \in M^n$ is a singular point then $e(q) + t(q) \geq $ toric degree $(\rho)$.}

Indeed, consider the induced toric action $\rho_\bbT: \bbT^k \times M^n \to M^n$, where $k$ is the toric degree of $\rho$.
If $q \in M$ and the isotropy group $Z_{\rho_{\bbT}} (q)$  of $\rho_{\bbT}$ at $q$ is of rank $s$ (i.e. its connected
component is isomorphic to $\bbT^s$), then $q$ has exactly $s$ elbolic components (because each elbolic component gives
rise to exactly one ``vanishing cycle'', i.e. a $\bbT^1$-subaction having $q$ as a fixed point), i.e. $e(q) = s$. On the other hand,
the action of the quotient group $\bbT^k/Z_{\rho_{\bbT}} (q) \cong \bbT^{k-s}$ on $\cO_q$ is free, so we have that
$t(q) \geq k-s$. Thus $e(q) + t(q) \geq k =$ toric degree $(\rho)$.

\underline{Step 5}: \emph{The converse inequality is also true: $e(q) + t(q) \leq $ toric degree $(\rho)$.}

The main point in Step 5 is to show that, if $w \in Z_\rho(q) \setminus I_\rho(q)$, 
where $I_\rho(q) = \{v \in \bbR^n \ |\ X_v(q) = 0 \} \subset \bbR^n$ 
is the isotropy vector space of the infinitesimal action at $q$
($I_\rho(q)$ is also the connected component of $Z_\rho(q)$ which contains 0),
then there is an element $\theta \in I_\rho(q)$  such that either $w + \theta \in Z_\rho$ (the non-twisted case)
or $2w + \theta \in Z_\rho$ (the twisted case). Indeed, 
if this fact is true, then $\dim_{\bbZ} (Z_\rho \cap I_\rho(q)) = e(q)$ and
$\dim_{\bbZ} (Z_\rho / (Z_\rho \cap I_\rho(q))) \geq t(q)$, because $t(q) = \dim_\bbZ ( Z_\rho(q) / (Z_\rho(q) \cap I_\rho(q)))$
and there is an injective homomorphism from $2Z_\rho(q) / (2Z_\rho(q) \cap I_\rho(q))$ 
into $Z_\rho / (Z_\rho \cap I_\rho(q))$, therefore the  toric degree of $\rho$ $=
\dim_\bbZ Z_\rho = \dim_{\bbZ} (Z_\rho \cap I_\rho(q)) + \dim_{\bbZ} (Z_\rho / (Z_\rho \cap I_\rho(q)) \geq e(q) + t(q)$.

Let us now prove the existence of an element $\theta \in I_\rho(q)$ such as above for any given
element  $w \in Z_\rho(q) \setminus I_\rho(q)$. Denote by
$(x_1,\hdots,x_n)$ a canonical coordinate system in a neighborhood $\cU$ of $q$ 
as given by the local normal form theorem. In particular, 
we have $\cO_q \cap \cU = \{(x_1,\hdots,x_n) \in \cU \ |\ x_1 = \hdots = x_{2e(q)+h(q)} = 0 \}$, and
the coordinate functions $x_{2e(q)+h(q)+1},\hdots, x_n$ are first integrals of the vector field $X_v$
for any $v \in I_\rho(q)$. Denote by $N = \{(x_1,\hdots,x_n) \in \cU \ |\ x_1=\hdots = x_{2e(q)+h(q)} = 0 \}$
the local transversal manifold to $\cO_q$ at $q$. Then $N$ is invariant with respect to the infinitesimal
action of  $I_\rho(q)$. This action of $I_\rho(q)$ 
divides $N$ into a finite number of local regular orbits, which we call the \emph{corners}
of $N$ (the number of corners is $2^{h(q)}$) and a finite number of singular orbits. Choose an arbitrary
vector subspace $W$ complementary to $I_\rho(q)$ in $\bbR^n$: $\bbR^n = I_\rho(q) \oplus W$.
Let $z \in N$ be a regular point close enough to $q$. Since $\rho(w,q) = q$, the point $\rho(w,z)$ is also
close to $N$, so that there is a unique small element $\gamma(z) \in W$ such that 
$\rho(w+ \gamma(z), z) = \rho(\gamma(z), \rho(w,z))$ belongs to $N$. The local map $P_w: z \mapsto 
\rho(w+ \gamma(z), z)$ from $N$ to itself is called the \emph{Poincaré map} on $N$ associated to $w$.
Notice that this Poincaré map is an automorphism of the infinitesimal action of $I_\rho(q)$ on $N$. This
action has $q$ as a nondegenerate fixed point, and according to the results of Section \ref{section:singularities},
either $P_w$ preserves  each corner of $N$ (the non-twisted case), or $(P_w)^2$ will do so (the twisted case).
For simplicity and without loss of generality, let us assume that $P_w$ preserves  each corner of $N$, i.e.
$P_w(z)$ belongs to the same corner as $z$. It means that we can write $\rho(\theta(z), P_w(z)) = z$ for some 
$\theta(z) \in I_\rho(q)$. Recall that $P_w(z) = \rho(w + \gamma(z),z)$, 
so we have $\rho(w+ \gamma(z) + \theta(z), z) = z$, i.e. $w  + \gamma(z) + \theta(z) \in Z_\rho(z)$. Since $z$ is
a regular point, we have $w  + \gamma(z) + \theta(z) \in Z_\rho$. Recall that $\gamma(z)$ tends to 0 when $z$
tends to $q$, so by taking the limit, we find an element $\theta \in I_\rho(q)$  such that $w + \theta \in Z_\rho$.
The theorem is proved.
\end{proof}

\begin{remark}
 The above theorem is similar to and inspired by some results of \cite{Zung-Symplectic1996} on the discrete invariants
of singular points on a nondegenerate singular fiber of an integrable Hamiltonian system. 
\end{remark}

\subsection{Actions of toric degree $n$ and $n-1$} \label{subsection:n_and_n-1}
When the toric degree of $\rho$ is equal to $n$ we get an effective action of $\bbT^n$ on $M^n$.
Each $n$-dimensional orbit of this action is open and  compact in $M^n$ at the
same time, because $\bbT^n$ is compact. Since $M^n$ is connected by our assumptions, a non-empty 
open compact subset of $M$ must be $M$ itself. It implies that the whole $M^n$ is 
just one orbit of the $\bbT^n$-action, and $M^n$ itself is differmorphic to $\bbT^n$.
In other words, we recover the following well-known result:

\begin{thm} \label{thm:Tn}
Up to diffeomorphisms the only connected $n$-dimensional 
manifold admitting an effective $\bbT^n$-action is the torus $\bbT^n$.
\end{thm}

Consider now a nondegenerate action $\rho$ of toric degree $n-1$ on a compact
connected manifold $M^n$, an orbit 
$\cO_p$ of this action, and denote by $(h,e,r,t) $ the HERT-invariant of $\cO_p$.

According to Theorem \ref{thm:HERT-toricdegree}, we have $e+t = n-1$. On the 
other hand, the total dimension is $n = h + 2e + r + t$. These two equalities imply that 
$h+e + r= 1$, which means that one of the three numbers $h, e, r$ is equal to 1 and the other two 
numbers are 0. So we have only three possibilities: 

1) $r = 1, h = e = 0, t = n-1$, and
$\cO_p \cong \bbT^{n-1} \times \bbR$ is a regular orbit. The action $\rho_\bbT$
of $\bbT^{n-1}$ on such an orbit is free with the orbit space diffeomorphic to an open
interval.

2) $r = e = 0, h = 1, t = n-1$, and $\cO_p \cong \bbT^{n-1}$ is a compact singular
orbit of codimension 1 which is transversally hyperbolic. The action $\rho_\bbT$
of $\bbT^{n-1}$ on such an orbit is locally free; it is either free (the non-twisted case)
or have the isotropy group equal to $\bbZ_2$ (the twisted case).

3) $e = 1, h = r = 0, t = n-2$, and
$\cO_p \cong \bbT^{n-2}$ is a compact singular
orbit of codimension 2 which is transversally elbolic.

The orbit space $S = M^n/\bbT^{n-1}$ of the action 
\begin{equation}
\rho_\bbT: \bbT^{n-1} \times M^n \to M^n
\end{equation}
is a compact one-dimensional manifold with or without
boundary, on which we have an induced action of $\bbR$. The singular points of this $\bbR$-action
on $M^n/\bbT^{n-1}$ correspond to the singular orbits  of $\rho$. Since the toric degree is $n-1$ and not
$n$ and $M$ is compact, $\rho$ must have at least one singular orbit, and so on the quotient space
$S = M^n/\bbT^{n-1}$ there is at least one singular point.

Topologically, $S$ must be a closed interval or a circle. A singular point in the interior of $S$
corresponds to a  transversally hyperbolic non-twisted singular orbit of of $\rho$ (Case 2 non-twisted), while
a point on the boundary of $S$ (in case $S$ is an interval) must correspond to either a
transversally elbolic orbit (Case 3) or a transversally  hyperbolic twisted orbit (Case 2 twisted). We can
combine all these possibilities together to construct $\bbR^n$ actions of toric degree $n-1$ on $n$-manifolds.

Globally, we have the following 4 cases:
\begin{figure}[htb]
\begin{center}
\includegraphics[width=100mm]{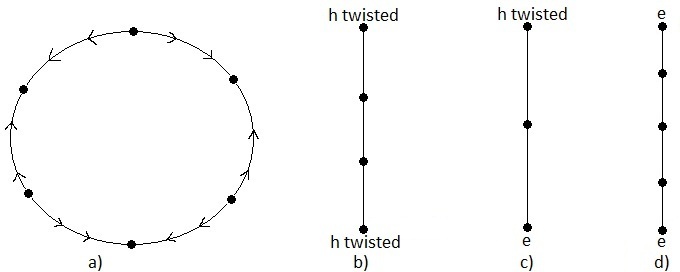}
\caption{The 4 cases of toric degree $n-1$}
\label{fig:n-1}
\end{center}
\end{figure}

\underline{Case a}: \emph{$S$ is a circle, which contains  $m > 0$  hyperbolic points with respect to the induced
$\bbR$-action on it}. 

Notice that, $m$ is necessarily an even number, because the vector field which generates
the hyperbolic $\bbR$-action on $S$ changes direction on adjacent regular intervals, see Figure \ref{fig:n-1}a
for an illustration. The $\bbT^{(n-1)}$-action is free in this case, so $M^n$ is a $\bbT^{n-1}$-principal bundle 
over $S$. Any homogeneous $\bbT^{n-1}$-principal bundle over a circle is trivial, so $M^n$ is diffeomorphic to
$\bbT^n \cong \bbT^{n-1} \times \bbS^1$ in this case.

\underline{Case b}: \emph{$S$ is an interval, and each endpoint of $S$ corresponds to a transversally
elbolic orbit of $\rho$}.

Topologically, in this case, the manifold $M^n$ can be obtained by gluing 2 copies of the ``solid torus'' $D^2 \times \bbT^{n-2}$
together along the boundary. When $n=2$, there is only one way to do it, and $M$ is diffeomorphic to a sphere $S^2$.
When $n \geq 3 $, the gluing can be classified by the homotopy class (up to conjugations) of the two vanishing 
cycles on the common boundary $\bbT^{n-1}$ (the first/ second 
vanishing cycle is the 1-dimensional cycle on the common boundary which becomes trivial on the first/second solid torus). 
When $n=3$, the manifold $M^3$ is either $\bbS^2 \times \bbS^1$
(if the two vanishing cycles are equal up to a sign) or a so called \emph{lens space} in 3-dimensional topology.
  
\underline{Case c}: \emph{$S$ is an interval, one endpoint of $S$ corresponds to a twisted transversally
hyperbolic orbit of $\rho$, and the other endpoint corresponds to a transversally elbolic orbit of $\rho$}.

Due to the twisting, the ambient manifold is non-orientable in this case. But $(M^n,\rho)$
admits a double covering $\widetilde{(M^n,\rho)}$ which belongs to Case b. If $n=2$ then $M^2 = \bbR \bbP^2$ in this case.
 
\underline{Case d}: \emph{$S$ is an interval, and each endpoint of $S$ corresponds to a twisted transversally
hyperbolic orbit of $\rho$}.

Again, in this case, $M$ is non-orientable, but $(M^n,\rho)$ admits a normal $(\bbZ_2)^2$-covering $\widetilde{(M^n,\rho)}$ 
which is orientable and belongs to Case a. If $n=2$ then $M^2$ = Klein bottle in this case.
 
We can classify actions of toric degree $n-1$ on closed manifolds as follows:

View $S$ as a (non-oriented) graph, with singular points (i.e. points which correspond to singular orbits of $\rho$) as vertices.
Mark each vertex of $S$ with the vector or the vector couple of $\bbR^n$ associated to the corresponding orbit
of $\rho$ (in the sense of Definition \ref{def:AssociatedVector}). Then $S$ becomes a marked graph, which is an invariant of $\rho$, and which
will be denoted by $S_{\text{marked}}$. The isotropy group $Z_\rho \subseteq \bbR^n $ (which is isomorphic 
to $\bbR^{n-1}$) is also an invariant of $\rho$. 
Note that $S_{\text{marked}}$ and $Z_\rho$ satisfy the following conditions ($C_i$)-($C_{iv}$):

\begin{itemize}
 \item[$C_i$)] $S$ is homemorphic to a circle or an interval.
 If $S$ is a circle then it has an even positive number of vertices. If $S$
 is an interval then it has at least 2 vertices, which are at the two ends of $S$.
 \item[$C_{ii}$)] Each interior vertex of $S$ is marked with a vector in $\bbR^n$. 
Each end vertex of $S$ may be marked with either a vector or a couple of vectors of the type 
$(v_1, \pm v_2)$ in $\bbR^n$ (the second vector in the couple is only defined up to a sign).
 \item[$C_{iii}$)] $Z_\rho$ is a lattice of rank $n-1$ in $\bbR^n$. 
 \item[$C_{iv}$)] If $v \in \bbR^n$ is the mark at a vertex of $S$, then
\begin{equation}
\bbR.v \oplus (Z_\rho \otimes \bbR) = \bbR^n.
\end{equation}
 If $(v, \pm w)$ is the mark at a vertex of $S$, then we also have
\begin{equation}
\bbR.v \oplus (Z_\rho \otimes \bbR) = \bbR^n
\end{equation}
while $w$ is a primitive element of $Z_\rho$. Moreover, if $v_i$ and $v_{i+1}$ are two consecutive marks
(each of them may belong to a couple, e.g. $(v_i, \pm w_i)$), then they lie on different 
sides of $Z_\rho\otimes \bbR$ in $\bbR^n$.
\end{itemize}

It is clear that $S_{\text{marked}}$ and $Z_\rho$ are invariants of the action $\rho$. In the case when $S$ is a 
circle (case a)), then there is another invariant called the \emph{monodromy} and defined as follows:

Denote by $F_1, \hdots, F_m$ the $(n-1)$-dimensional orbits of $(M^n, \rho)$ in cyclical order (they correspond to
vertices of $S$ in cyclical order). Denote by $\sigma_i$ the reflection associated to $F_i$. Let $z_1$ be an arbitrary 
regular point which projects to a point in the regular orbit lying between $F_m$ and $F_1$.

Put $z_2 = \sigma_1(z_1)$ (which is a point lying on the regular orbit between $F_1$ and $F_2$), 
$z_3 = \sigma_2(z_2),\hdots,z_{m+1} = \sigma_m(z_m)$. Then $z_{m+1}$ lies on the same regular orbit as $z_1$,
and so there is a unique element $\mu \in \bbR^n/Z_\rho$ such that 
\begin{equation}
z_{m+1} = \rho(\mu, z_1).
\end{equation}
This element $\mu$ is called the {\bf monodromy} of the action. Notice that $\mu$ does not depend on the choice of $z_1$
nor on the choice of $F_1$ (i.e. which singular orbit is indexed as the first one), but only on the choice of the orientation
of the cyclic order on $S$: If we change the orientation 
of $S$ then $\mu$ will be changed to $-\mu$ (modulo $Z_\rho$). So a more correct way 
to look at the monodromy is to view it as a homomorphism from $\pi_1(S) \cong \bbZ$ to $\bbR^n/Z_\rho$.
\begin{figure}[htb]
\begin{center}
\includegraphics[width=80mm]{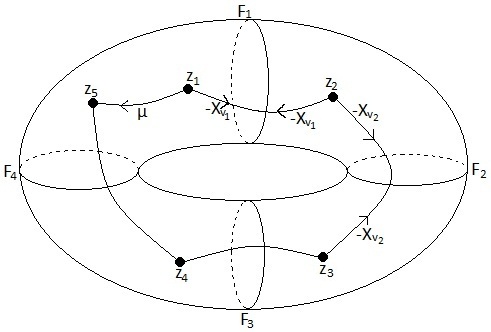}
\caption{Monodromy $\mu$ when $S \cong \bbS^1$.}
\label{fig:monodromyS1}
\end{center}
\end{figure}

\begin{thm}[Classification by marked $S$-graph] \label{thm:class-marked}
1) If $(S_{\text{marked}}, Z)$ is a pair of marked graph and lattice which satisfies the conditions ($C_i$)-($C_{iv}$) above,
then they can be realized as the marked graph and the isotropy group of a nondegenerate action of $\bbR^n$
of toric degree $n-1$ on a compact $n$-manifold. Moreover, if $S$ is a circle then any monodromy 
element $\mu \in \bbR^n/Z$ can also be realized. 

2) a) In the case when $S$ is an interval, then any two such actions having the same $(S_{\text{marked}}, Z)$-invariant are isomorphic. b) In the case when $S$ is a circle, then any two actions having the same ($S_{\text{marked}}, Z, \mu$) are isomorphic.
\end{thm}

\begin{proof}
1) The proof is by surgery, i.e. gluing of linearized pieces given by Theorem \ref{thm:semi-localform}. There is no 
obstruction to doing so.

2a) If there are 2 different actions $(M_1, \rho_1)$ and $(M_2, \rho_2)$ with the same marked graph $(S_{\text{marked}}, Z)$, 
then one can construct an isomorphism $\Phi$ from $(M_1, \rho_1)$ to $(M_2, \rho_2)$ as follows.

Take $z_1 \in M_1$    and $z_2 \in M_2$ such that $z_1$ and $z_2$ project to the 
same regular point on $S_{\text{marked}}$. Put $\Phi (z_1) =z_2$. Extend $\Phi$ to $\cO_{z_1}$ by the formula 
$$\Phi (\rho_1(\theta,z_1)) = \rho_2(\theta, z_2).$$
Then extend $\Phi$ to rest of $M_1$ by the reflection principle and the continuity principle.

2b) The proof is similar to that of assertion 2a).

\end{proof}

\section{The monodromy}
In the classification of actions of toric degree $n-1$ in Subsection \ref{subsection:n_and_n-1}, 
we have encountered a global 
invariant called the monodromy. It turns out that the monodromy can also be defined for any nondegenerate 
action $\rho: \bbR^n \times M^n \to M^n$ of any toric degree, and is one of the main invariants of the action.

Choose an arbitrary regular point $z_0 \in (M^n, \rho)$, and a loop 
$\gamma : [0,1] \to M^n, \gamma(0) = \gamma(1) = z_0$. By a small perturbation 
which does not change the homotopy class of $\gamma$, we may assume that
$\gamma$ intersects the $(n-1)$ singular orbits of $\rho$ transversally (if at all), 
and does not intersect orbits of dimension $\leq n -2$.

Denote by $p_1, \hdots, p_m$ ($m \geq 0$) the singular points of corank 1 on the loop $\gamma$,
and $\sigma_1, \hdots, \sigma_m$ the associated reflections of the singular hypersurfaces which contain 
$p_1, \hdots, p_m$ respectively as given by Theorem \ref{thm:Reflection}. 

Put $z_1 =\sigma_1(z_0), z_2 =\sigma_2(z_1), \hdots, z_m =\sigma_m(z_{m-1})$. 
(The involution $\sigma_0$ can be extended from a small neighborhood of $p_1$
to $z_0$ in a unique way which preserves $\rho$, and so on). Then $z_m$ lies in 
the same regular orbit as $z_0$, so there is a unique element
$\mu = \mu(\gamma ) \in \bbR^n/Z_\rho$ such that $z_m = \rho(\mu(\gamma), z_0)$.

\begin{thm} \label{thm:defn-monodromy}
With the above notations, we have:

1) $\mu(\gamma )$ depends only on the homotopy class of $\gamma$.

2) The map $\mu : \pi_1(M^n,z_0) \to \bbR^n/Z_\rho$ is a group homomorphism from the 
fundamental group of $M^n$ to $\bbR^n/Z_\rho$.

3) $\mu$ does not depend on the choice of $z_0$, and can be
viewed as a homomorphism from the first homology group 
$H_1(M^n, \bbZ) \cong \pi_1(M^n)/[\pi_1(M^n),\pi_1(M^n)]$
to $\bbR^n/Z_\rho$, which will also be denoted by $\mu$:
\begin{equation}
\mu: H_1(M^n, \bbZ) \to \bbR^n/Z_\rho.
\end{equation}
\end{thm}
\begin{proof}
1) Consider a homotopy $\Gamma : [0,1] \times [0,1] \to M^n$ from a loop 
$\gamma_0 = \Gamma(0,.)$ to a loop $\gamma_1 = \Gamma(1,.)$. By a small 
perturbation, we may assume that locally only the following three kinds of situations
can happen when moving from $\gamma_t= \Gamma(t,.)$ to $\gamma_t'=\Gamma(t',.)$ with 
$t'$ close to $t$ (see Figure \ref{fig:homotopy}):
\begin{figure}[htb]
\begin{center}
\includegraphics[width=100mm]{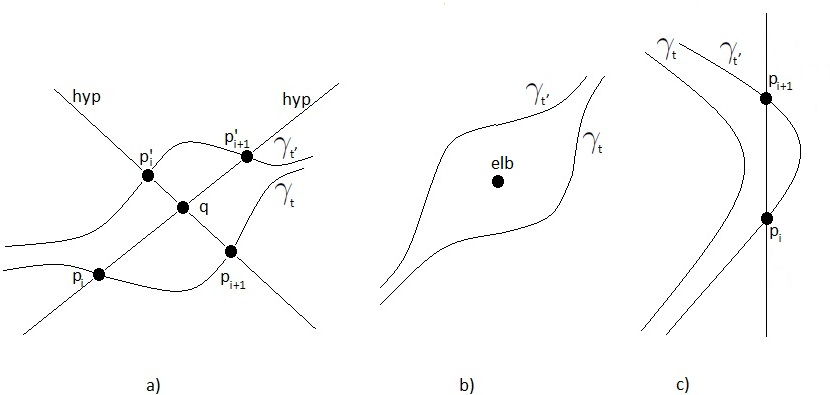}
\caption{Homotopy of $\gamma$.}
\label{fig:homotopy}
\end{center}
\end{figure}

a) $\gamma$ crosses a corank-2 transversally hyperbolic-hyperbolic 
orbit at a point $q$. After this crossing, two consecutive singular points
$p_i, p_{i+1}$ on $\gamma_t$ will be replaced by two other singular points $p_i'$
and $p_{i+1}'$ on $\gamma_{t'}$, with $\sigma_i = \sigma_{i+1}'$ (i.e. 
$\sigma_{p_i} = \sigma_{p_{i+1}'}$) and $ \sigma_{i+1}' = \sigma_i$.
Due to the commutativity of $\sigma_i$ and $\sigma_{i+1}$ as can be seen from 
the point $q$, we have $\sigma_i \circ \sigma_{i+1} = \sigma_i' \circ \sigma_{i+1}'$,
which implies that $\mu(\gamma_t)=\mu(\gamma_t')$ (the other $\sigma_j, j \neq i, i+1$, remain the same).

b) $\gamma$ crosses a singular corank-2 transversally elbolic orbit. In this case
$\gamma_t$ and $\gamma_t'$ gives rise to exactly the same sequence of involutions $\sigma_i$'s, so  $\mu(\gamma_t)=\mu(\gamma_t')$.

c) $\gamma_t$ enters (or exists) a new regular orbit by crossing a transversally hyperbolic
corank-1 orbit. Then 2 new consecutive singular points $p_i, p_{i+1}$ are created (or
disappear), with $\sigma_i = \sigma_{i+1}$. Since $\sigma_i$ is an involution, 
$\sigma_i \circ \sigma_{i+1} = \sigma_i^2 = Id$, we also have $\mu(\gamma_t)=\mu(\gamma_t')$ in this case. 

2) Since $\rho$ commutes with all the involutions $\sigma_i$, it is obvious that
\begin{equation}
\mu (\gamma_1 \circ \gamma_2) = \mu (\gamma_1) +\mu(\gamma_2)
\end{equation}
for any loops $\gamma_1, \gamma_2$ starting at $z_0$.

Thus we obtain a homomorphism $\mu$ from $\pi_1(M^n,z_0)$ to $\bbR^n/Z_\rho$.

3) Since $\bbR^n/Z_\rho$ is commutative, any homomorphism from $\pi_1(M^n)$ to  
 $\bbR^n/Z_\rho$ descends to a homomorphism from its Abelianization 
\begin{equation}
H_1(M^n, \bbZ) \cong \pi_1(M^n)/[\pi_1(M^n),\pi_1(M^n)]
\end{equation}
to $\bbR^n/Z_\rho$. Since we can go from any regular point of $M$ to any other regular point by the 
reflections $\sigma_i$ and the action $\rho$, and $\rho$ commutes with these $\sigma_i$'s,
 it is also clear that $\mu$ does not depend on the choice of $z_0$ in $M^n$.
\end{proof}

\begin{defn}
The homomorphisms $\mu : H_1(M^n, \bbZ) \to \bbR^n/Z_\rho$ and $\mu : \pi_1(M^n) \to \bbR^n/Z_\rho$
(which are both denoted by $\mu$ for simplicity)
given by Theorem \ref{thm:defn-monodromy} are called the {\bf monodromy}
of the action $\rho: \bbR^n \times M^n \to M^n$.
\end{defn}

\begin{remark}
For proper integrable Hamiltonian systems and their associated Lagrangian torus
fibrations, there is also the notion of monodromy, introduced in the regular case 
by Duistermaat \cite{Duistermaat-globalaction-angle1980} and extended to the singular case
by Zung \cite{Zung-Integrable2003}. Base spaces of singular Lagrangian torus
fibrations admit a natural stratified singular affine structure, and the monodromy of 
singular Lagrangian torus is the monodromy of that affine structures \cite{Zung-Integrable2003}. 
Such singular affine structures also play an important role, for example, in the study of mirror symmetry
of Calabi-Yau manifolds, see, e.g. Gross-Siebert \cite{Gross-Siebert} and Kontsevich-Soibelman \cite{Kontsevich-Soibelman}. In our situation here, one can also view the action $\rho: \bbR^n \times M^n \to M^n$
as giving rise to a kind of special singular affine structure on $M^n$, and the monodromy
$\mu : \pi_1(M^n,z_0) \to \bbR^n/Z_\rho$ as a monodromy of this special singular affine structure.
\end{remark}
A simple remark is that the twisting groups are subgroups of the monodromy group, i.e.
the image of $\pi_1(M^n)$ by $\mu$ in $\bbR^n/Z_\rho$:

\begin{thm}[Twistings and monodromy] \label{thm:TwistingMonodromy}
For any $q \in M^n$, we have
\begin{equation}
G_q \subseteq Im(\mu),
\end{equation}
where $G_q = (Z_\rho(q) \cap Z_\rho \otimes \bbR)/Z_\rho$ is the twisting group
of the action $\rho$ at $q$, and $Im(\mu) = \mu(\pi_1(M^n)) \subseteq \bbR^n/Z_\rho$
is the image of $\pi_1(M^n)$ by the monodromy map $\mu$. In particular,
if $M^n$ is simply-connected, then $Im(\mu)$ is trivial, and $\rho$ has no twisting.
\end{thm}
\begin{proof}
Let $q \in M^n, (w \mod Z_\rho) \in G_q$, and $z_0$ be a regular point close enough to $q$.
Consider the loop $\gamma: [0,1] \to M^n$  defined as follows:
\begin{itemize}
\item $\gamma(t) = \rho(2tw,z_0) \quad \forall 0 \leq  t \leq \frac{1}{2}$
\item $\gamma(t)$ for $\frac{1}{2} \leq  t \leq 1$ is a path from $\rho(w,z_0)$ to $z_0$
in a small neighborhood of $q$.
\end{itemize}
The one verifies, using the definition of the monodromy and the semi-local normal form theorem,
that $\mu([\gamma]) = w \mod Z_\rho$.
\end{proof}
The proof of the above theorem shows that the monodromy map $\mu : H_1(M^n,\bbZ) \to \bbR^n/Z_\rho$
satisfies the following compatibility condition (*) with the isotropy groups:

\vspace{0.3cm}
\begin{quote} (*) \quad \emph{ If $[\gamma] \in H_1(M^n,\bbZ)$  can be represented by a loop of the type 
$\{\rho(tw,p) | t \in [0,1]\}$ where $p \in M^n, w \in Z_\rho(p) \cap Z_\rho\otimes \bbR$, then
$\mu([\gamma])= w \mod Z_\rho$.}
\end{quote}

\vspace{0.3cm}
In particular, If $[\gamma] \in H_1(M^n,\bbZ)$  can be represented by a loop of the type 
$\{\rho(tw,p) | t \in [0,1]\}$ where $w \in Z_\rho$, then $\mu([\gamma])=0$.

Take an arbitrary regular point $z_0 \in M$. Then the map 
\begin{equation}
Z_\rho \to H_1(M^n,\bbZ),
\end{equation}
which associates to $w \in Z_\rho$ the homology class of the loop 
$\{\rho(tw,z_0) | t \in [0,1]\}$ in $H_1(M^n,\bbZ)$, is a homomorphism
which does not depend on the choice of $z_0$. Denote the image of this map
by $Im(Z_\rho)$. Then we can also view the monodromy 
of $(M^n, \rho)$ as a homomorphism, which we will also denote by
\begin{equation}
\mu : H_1(M^n,\bbZ)/Im(Z_\rho) \to \bbR^n/Z_\rho,
\end{equation}
from $H_1(M^n,\bbZ)/Im(Z_\rho)$ to $\bbR^n/Z_\rho$.

According to the structural theorem for finitely generated Abelian groups, we can write
\begin{equation}
H_1(M^n,\bbZ)/Im(Z_\rho) = G_{\text{torsion}}\oplus G_{\text{free}},
\end{equation}
where $G_{\text{torsion}} \subseteq H_1(M^n,\bbZ)/Im(Z_\rho)$ is its torsion part, and
$G_{\text{free}} \cong \bbZ^k$, where $k = \rank_\bbZ \big(H_1(M^n,\bbZ)/Im(Z_\rho)\big)=
 \dim_\bbR \big(\big(H_1(M^n,\bbZ)/Im(Z_\rho)\big)\otimes \bbR\big)$,
is a free part complementary to $G_{\text{torsion}}$.

This decomposition of $H_1(M^n,\bbZ)/Im(Z_\rho)$ gives us a decomposition of $\mu$:
\begin{equation}
\mu = \mu_{\text{torsion}}\oplus \mu_{\text{free}},
\end{equation}
where $\mu_{\text{torsion}} : G_{\text{torsion}} \to \bbR^n/Z_\rho$
is the restriction of $\mu$ to the torsion part $G_{\text{torsion}}$, and
$\mu_{\text{free}}$ is the restriction of $\mu$ to $G_{\text{free}}$.

Notice that $\mu_{\text{torsion}}$ is not arbitrary, but must satisfy the above 
compatibility condition (*) with the twisting groups. On the other hand, $\mu_{\text{free}}$
can be arbitrary. More precisely, we have the following theorem:

\begin{thm}[Changing of monodromy] \label{thm:trans-monodromy}
With the above notations, assume that $\mu_{\text{free}}' : G_{\text{free}} \to \bbR^n/Z_\rho$
is another arbitrary homomorphism from $G_{\text{free}}$ to $\bbR^n/Z_\rho$. Put 
\begin{equation}
 \mu' = \mu_{\text{torsion}}\oplus \mu_{\text{free}}':  H_1(M^n,\bbZ)/Im(Z_\rho) \to \bbR^n/Z_\rho.
\end{equation}
Then there exists another nondegenerate action $\rho': \bbR^n \times M^n \to M^n$, 
which has the same orbits as $\rho$ and the same isotropy group at each point of $M^n$ 
as $\rho$, but whose monodromy is $\mu'$.
\end{thm}
Before proving the above theorem, let us notice that we can kill the 
monodromy by an appropriate covering of $(M^n, \rho)$. More precisely, we have:

\begin{thm} \label{thm:lifting-monodromy}
Denote by $\widetilde{M^n}$ the covering of $M^n$ corresponding to the kernel $\ker \mu$
of the monodromy homomorphism $\mu: \pi_1(M^n) \to \bbR^n/Z_\rho$, i.e. $\pi_1(\widetilde{M^n}) = \ker \mu$. 
Denote by $\widetilde{\rho}$ the lifting of $\rho$ on $\widetilde{M^n}$. Then we have:

1) $(\widetilde{M^n},\widetilde{\rho})$ has the same toric degree as $(M^n,\rho)$, and
the same isotropy group: $Z_{\widetilde{\rho}} = Z_\rho$.

2) The monodromy of $(\widetilde{M^n},\widetilde{\rho})$ is trivial.
\end{thm}
\begin{proof}
The proof is straightforward.
\end{proof}

\underline{Proof of Theorem \ref{thm:trans-monodromy}}

The covering $\widetilde{M^n}$ of $M^n$ in the above theorem is the minimal covering
which trivializes the monodromy of $\widetilde{\rho}$. In order to 
prove  Theorem \ref{thm:trans-monodromy}, we will use a kind of universal covering $\widehat{M^n}$ (among 
all coverings on which the toric action $\rho_\bbT$ associated to $\rho$ can be lifted):
$\widehat{M^n}$  is the covering of $M^n$ such that $\pi_1(\widehat{M^n}) = Im(Z_\rho)$, where 
$Im(Z_\rho) \subset \pi_1(M^n,z_0)$ now denotes the image of 
$Z_\rho \cong \pi_1(\bbT^{\text{toric degree}(\rho)})$ in $\pi_1(M^n,z_0)$ via an isotropy-free orbit of $\rho_\bbT$ on $M^n$.

The action $\rho$ can also be naturally lifted to an action on $\widehat{M^n}$, 
which we will denote by $\widehat \rho$. Similarly to the above theorem, we have $Z_{\widehat \rho}=Z_\rho$ 
and the monodromy of $\widehat \rho$ on $\widehat{M^n}$ is trivial.

Remark that $Im(Z_\rho)$ is a normal subgroup of $\pi_1(M^n,z_0)$. 
The covering $\widehat{M^n} \to M^n$ is a normal covering, the quotient 
group $\pi_1(M^n,z_0)/ Im(Z_\rho)$ acts on $\widehat{M^n}$ freely, and $(M^n, \rho)$ is
isomorphic to the quotient of $(\widehat{M^n}, \widehat\rho)$ by this action.

In order to obtain another action $\rho'$ with monodromy $\mu'$, it suffices to modify
the action of $\pi_1(M^n,z_0)/ Im(Z_\rho)$ on $\widehat{M^n}$, in such a way that 
the action remains free and preserves $\widehat\rho$, the quotient manifold $M'$ is diffeomorphic to $M^n$, but 
the induced $\bbR^n$-action $\rho'$ on $M'$ has monodromy equal to $\mu'$ instead 
of $\mu$. In fact, $\mu'$ and  $\mu$ indicate how to define the new action of $\pi_1(M^n,z_0)/ Im(Z_\rho)$ on $\widehat M^n$:

Let $z \in (\widehat{M^n}, \widehat\rho)$ be a regular point and let a loop $\gamma : [0,1] \to M^n$,
$\gamma(0) = \gamma(1) = proj(z)$, represent an element 
$[\gamma] \in \pi_1(M^n,z_0)/ Im(Z_\rho)$. Denote 
by $\widehat \gamma : [0,1] \to \widehat{M^n}, \widehat \gamma(0) = z$ the lifting 
of $\gamma$ from $M^n$ to $M'$. Put
\begin{equation}
A_{[\gamma]}(z) = \widehat\rho (\mu'([\gamma])-\mu([\gamma]),\gamma(1)).
\end{equation}
One verifies easily that the map 
$A: \pi_1(M^n,z_0)/ Im(Z_\rho) \times \widehat{M^n} \to \widehat{M^n} $, defined 
by the above formula, is a free action of $\pi_1(M^n,z_0)/ Im(Z_\rho)$ on $\widehat{M^n}$ which commutes 
with $\widehat \rho$, the quotient $M' = \widehat{M^n}/A$ is still 
diffeomorphic to $M^n$, and the induced action $\rho' = \widehat \rho/A$ 
on $M'$ can be thought of as having the same orbits and 
isotropy as $\rho$ up to a diffeomorphism, but with the monodromy equal to $\mu'$. \QED

\begin{remark}
a) In Theorem \ref{thm:trans-monodromy}, it is possible to change $\mu_{\text{torsion}}$ also to another
homomorphism $\mu_{\text{torsion}}' : G_{\text{torsion}} \to \bbR^n/Z_\rho$. Then the construction of 
the proof still works, but the new action $\rho'$ will not have the same isotropy groups as $\rho$ at 
twisted singular orbits in general, and even the diffeomorphism type of $M'$ may be different from $M$,
because the new action of $\pi_1(M^n,z_0)/Im(Z_\rho)$ will not be isotopic to the old one.

b) We don't know yet if $G_{\text{torsion}}$ is completely generated by the twisting elements or not in general.
\end{remark}

Another way to look at the monodromy is as follows:

Given a nondegenerate action $\rho: \bbR^n \times M^n \to M^n$, we will look at its 
{\bf 1-skeleton}, which is a graph denoted by $Skelet_1(M^n,\rho)$, and 
defined as follows:
\begin{itemize}
 \item Each vertex of $Skelet_1(M^n,\rho)$ corresponds to exactly one regular orbit of $\rho$.
 \item Each edge of $Skelet_1(M^n,\rho)$  corresponds to one corank-1 singular orbit of $\rho$.
The two ends of the edge are glued to the vertices corresponding to adjacent regular orbits.
(If the singular orbit is twisted, i.e. if it is adjacent to only one regular orbit, then
the two ends of the edge are glued to the same vertex).
\end{itemize}

By lifting, we get a natural homomorphism 
\begin{equation}
l: H_1(Skelet_1(M^n,\rho),\bbZ) \to H_1(M^n,\bbZ)/Im(Z_\rho).
\end{equation}
The map 
\begin{equation}
\widehat \mu: H_1(Skelet_1(M^n,\rho),\bbZ) \to \bbR^n/Z_\rho
\end{equation}
defined by $\widehat \mu = \mu \circ l$ will also called the {\bf monodromy} of $\rho$
(on $Skelet_1(M^n,\rho)$).

Observer that $\widehat \mu$ is not arbitrary, but must satisfy the following commutativity
and compatibility conditions M1 and M2.

M1) If $q \in M^n$ is a singular point of HE-invariant (2,0) and
 $\sigma_1,\sigma_2,\sigma_3,\sigma_4$ denote the edges on $Skelet_1(M^n,\rho)$ 
corresponding to 4 local corank-1 orbits adjacent to $q$ (we may have for example 
 $\sigma_1 = \sigma_3$ in the twisted case), then 
\begin{equation}
\widehat \mu([\sigma_1\sigma_2\sigma_3\sigma_4])=0.
\end{equation}
(See Figure \ref{fig:loop-skeleton1}a).

\begin{figure}[htb]
\begin{center}
\includegraphics[width=90mm]{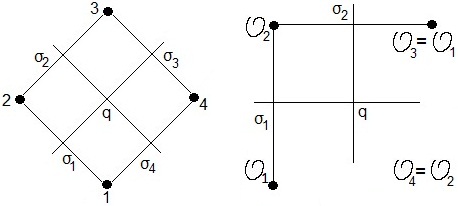}
\caption{Loops in $Skelet_1(M^n,\rho)$ on which $\widehat \mu $ 
must satisfy commutativity and compatibility conditions.}
\label{fig:loop-skeleton1}
\end{center}
\end{figure}

M2) If $q \in M^n$ is a singular point of HE-invariant $(h,e)$ with $h \geq 1$ 
which is twisted, i.e. there is an element $w \in Z_\rho(q)$ such that $w \notin Z_\rho$ but 
$2w \in Z_\rho$, then for any sequence of edges $\sigma_1,\hdots, \sigma_m$ 
corresponding to local corank-1 orbits adjacent to $z$ which forms a loop in $Skelet_1(M^n,\rho)$,
and such that when viewed as reflections with respect to the corresponding corank-1 orbit,
then $\sigma_m \circ \sigma_{m -1} \circ \hdots \circ \sigma_1(z)$ lies in the same local
regular orbit as $\rho(w,z)$, where $z$ is a regular point near $q$, then we also have 
\begin{equation} 
\widehat \mu([\sigma_1\hdots \sigma_m])=w \mod Z_\rho.
 \end{equation}
 (See Figure \ref{fig:loop-skeleton1}b for an illustration).

\section{Totally hyperbolic actions}
\label{section:hyperbolic}

\subsection{Hyperbolic domains and complete fans}

\begin{defn}
A {\bf hyperbolic domain} is an orbit of dimension $n$  of a totally hyperbolic action 
$\rho: \bbR^n \times M^n \to M^n$, i.e. the toric degree of $\rho$ is 0. 
\end{defn}

An equivalent definition is: a hyperbolic domain is an orbit of type  $\bbR^n$, i.e. 
the action $\rho$ is free on it. Remark that, if a nondegenerate action  $\rho: \bbR^n \times M^n \to M^n$
admits an orbit of type $\bbR^n$, then $\rho$ is necessarily totally hyperbolic, according to the toric degree
formula (Theorem \ref{thm:HERT-toricdegree}). Moreover, every orbit of a totally hyperbolic action is of type
$\bbR^k$ for some $0 \leq k \leq n$. In particular, the boundary $\partial \cO$ of a hyperbolic domain
$\cO$ consists of singular orbits of types $\bbR^k$ with $0 \leq k < n$.

\begin{prop}
Let $\cO$ be a hyperbolic domain of a totally hyperbolic action
 $\rho: \bbR^n \times M^n \to M^n$. Assume that the closure $\bar{\cO}$  of $\cO$
 in $M^n$ is compact (this condition is automatic if $M^n$ itself is compact). Then
we have: \\
i) $\bar \cO$ contains a fixed point of the action $\rho$. \\
ii) For each $0 \leq k \leq n$, $\bar \cO$ contains an orbit of dimension $k$ of the action $\rho$. \\
iii) The partition of $\bar \cO$ into the orbits of the action  $\rho$ is a cell decomposition of
$\bar \cO$, where each $k$-dimensional orbit is also a $k$-dimensional cell. 
\end{prop}

\begin{proof}
i) Let $p \in \bar \cO$
be a point of lowest rank of the action on $\bar \cO$. The orbit $\cO_p$ through $p$ of the
action lies in $\bar \cO$, and is of the type $\bbR^k$ where $k = \rank p$. If $k > 0$ then
$\cO_p$ is not compact, but $\bar{\cO_p} \subset \bar{\cO}$ is compact, so there exists
a point $q$ on the boundary $\partial \cO_p = \bar{\cO_p} \setminus \cO_p$ of 
$\cO_p$, which necessarily has lower rank than $p$, which is a contradition. Thus the rank of
$p$ is 0, i.e. $p$ is a fixed point of the action.

ii) It follows directly from the local normal form theorem. Indeed, let $p \in \bar \cO$ be a fixed point.
Then there is a local canonical coordinate system $(x_1, \hdots, x_n)$ in a neighborhood $\cU (p)$ of $p$ 
and an adapted basis $\alpha_1, \hdots, \alpha_n$ of $\bbR^n$,
so that  we have 
\begin{equation}
X_{\alpha_1} = x_1\frac{\partial}{\partial x_1}, 
\hdots,X_{\alpha_n} = x_n\frac{\partial}{\partial x_n},
\end{equation}
where $X_{v}$ is the generator of $\rho$ associated to $v \in \bbR^n$. 
Since $p \in \bar \cO$, there is $q \in \cO$ which lies in $\cU (p)$. 
Notice that $x_i(q) \not= 0\ \forall \ i = 1, \hdots, n$, otherwise $q$ would be singular.
For each $0 \leq k \leq n$, the point $q_k = (x_1(q),\hdots, x_k(q), 0,\hdots, 0)$ in the above
coordinate system belongs to $\bar \cO$ and is of rank $k$, 
which implies that the orbit $\cO_{q_k}$ through $q_k$ lies in  $\bar \cO$ and 
is of type $\bbR^k$.

iii) This statement follows directly from the local structure 
of singularities as in ii), and the fact that any orbit of a totally hyperbolic action 
is of type $\bbR^k$ for some $1 \leq k \leq n$.
 \end{proof}

Recall that a {\bf maninold with boundary and corners} is something which near each point looks like a neighborhood 
of 0 in $\{ (y_1, \hdots,y_n) \in \bbR^n \ |\ y_i \geq 0 \ \forall \ i \geq k+1 \}$ for some $k$. 

Below, we will show that the closure $\bar \cO$ of a hyperbolic domain $\cO$ is a manifold
with boundary and corners, and $(\bar \cO, \rho)$ can be classified by the so-called complete fans.

First let us look at the $(n-1)$-dimensional faces (i.e. orbits) of a closed hyperbolic domain $\cO$. 
Let $q \in \bar \cO$ be a point of rank $n-1$, and denote by $\cO_q$ the $(n-1)$-dimensional
orbit of the action $\rho$ through $q$, and by $v_q \in \bbR^n$ the vector associated to $\cO_q$.

\begin{lemma}\label{lemma:limit}
With the above notations, for any point $z \in \cO$ we have that $\phi_{X_v}^t (z) = \rho(-tv_q, z)$ 
tends to a point on $\cO_q$ when $t$ tends to $+\infty$.
\end{lemma}

\begin{proof} In a neighborhood of $q$ in which we have the normal form  $X_{v_q} = x \frac{\partial}{\partial x}$,
it is obvious that there is a point $z_0$ in $\cO$ such that $\rho(-tv_q, z_0)$ tends to $q$ when $t$ tends to $+\infty$.
Let $z$ be an arbitrary point of $\cO$. Then $z = \rho (w,z_0)$ for some $w \in \bbR^n$, and by commutativity
we have $\rho(-tv_q,z) = \rho(w, \rho(-tv_q,z_0))$ tends to $\rho(w,q) \in \cO_q$ when $t$ tends to $+\infty$.
\end{proof}

We can say that the orbit $\cO$ tends to $\cO_q$ in the direction $-v$ of the flow of the action. In particular, It shows
that if $\cO_q$ and $\cO_{q'}$ are two different $(n-1)$-dimensional orbits on the boundary of $\cO$, then the 
associated vectors $v$ and $v'$ in $\bbR^n$ must be different. 

\begin{thm} \label{thm:closureHyperbolic}
 Let $\cO \subset M^n$ be a hyperbolic domain of a totally hyperbolic action $\rho$. Then the closure $\bar \cO$ of $\cO$
 is a manifold with boundary and corners.
\end{thm}

\begin{proof}
 The main point is to show that no overlapping of the boundary $\partial \cO = \bar \cO \backslash \cO$ is possible.
In other words, let $q \in \partial \cO$ be a singular point of corank $k$, $\cU$ be a neighborhood of $q$ in $M$ 
toghether with a canonical coordinate system $(x_1, \hdots, x_n)$ and adapted basis $(w_1,\hdots, w_n) \in \bbR^n$
of the action such that $X_{w_i} = x_i\frac{\partial}{\partial x_i} \quad \forall i \leq k$ and 
$X_{w_i} = \frac{\partial}{\partial x_i} \quad \forall i > k$ in $\cU$.
Assume that the corner $\{(x_1, \hdots, x_n) \in \cU \ |\ x_i > 0 \quad \forall i \leq k \}$ lies in $\cO$.
We have to show that no other corner of $\cU$ lies in $\cO$, i.e. if $(x_1, \hdots, x_n) \in \cU \cap \cO$ then 
$x_i >0 \quad \forall i \leq k$. We will prove it by induction on $k$.

\underline{The case $k=1$}

Assume that $k=1$, i.e. the rank of a point $q \in \partial \cO$ is $n-1$, and that $\cO$
approaches the $(n-1)$-dimensional orbit $\cO_q$ from both side of $\cO_q$. Denote by 
$v \in \bbR^n$ the vector associated to $\cO_q$, as in Definition \ref{def:AssociatedVector}. 
Then $\cU \backslash \cO_q \subset \cO$ by our assumptions, and 
there exist points $z_1, z_2 \in \cU \backslash \cO_q$ such that $x_1(z_1) > 0, x_1(z_2) <0$ and
$\lim_{t \to \infty} \rho(-tv,z_1) = \lim_{t \to \infty} \rho(-tv,z_2) = q$.
Since $z_1, z_2 \in \cO$, there is a unique $w \in \bbR^n$ such that $z_2 = \rho(w, z_1)$. 
   
Note that $w$ is not collinear to $v$, because, if $w = sv$ with $s >0$ for example then
$\rho(w,z_1) \in \cU$ and 
$x_1(\rho(w,z_1)) = x_1(\rho(sv,z_1)) = e^{-s}x_1(z_1) >0$,
hence $\rho(w,z_1) \neq z_2$ (if $w = sv$ with $s <0$ then we use the formula $z_1 = \rho(-w,z_2)$ instead).

By commutativity of the action we have 
$$q = \lim_{t \to \infty}\rho(-tv,z_2) =  \rho(w, \lim_{t \to \infty}\rho(-tv,z_1))  = \rho(w,q),$$
i.e. $w$ belongs to the isotropy group of the action at $q$.

On the other hand, due to the hyperbolicity of the action, the isotropy group at $q$ is $\bbR v$,
so we have a contradiction. So it is impossible for $\cO$ to approach $\cO_q$ from both sides.
   
\underline{The case $k>1$}

We take $z_1, z_2 \in \cO \cap \cU$  such that $x_1(z_1) > 0, x_2(z_1) >0$ and 
$x_1(z_2) < 0, x_2(z_2) <0$.

Denote by $v$ the element of $\bbR^n$ such that $X_v = x_1\frac{\partial}{\partial x_1}$ in $\cU$, 
and by $F$ the $(n-1)$-dimensional orbit in $\bar \cO$
coresponding to $v$, i.e. $X_v$ has the type $x\frac{\partial}{\partial x}$ near every point of $F$.
Then $F$ contains both the point $p_1 = \lim_{t \to \infty}\rho(-tv,z_1)$ and $p_2 = \lim_{t \to \infty}\rho(-tv,z_2)$.
Notice that $x_2(p_1) = x_2(z_1) >0$ and $x_2(p_2) = x_2(z_2) <0$.
On the submanifold $M_v = \{ z \in M \ |\ X_v(z) =0\}$, we have a totally hyperbolic action of $\bbR^{n-1} \cong \bbR^n/\bbR v$
for which $F$ is a hyperbolic domain, and $q$ lies on the boundary of $F$.
The corank of $q$ with respect to the action of $M_v$ is $k-1$, and $F$ has at least 2 different corners at $q$. 
By induction on $k$, we have a contradiction, and so this situation is impossible. 

\end{proof}

\begin{lemma} \label{lemma:converge}
Assume that the closure $\bar \cO$ of a hyperbolic domain $\cO$ is compact.
Let $z \in \cO$ and $w \in \bbR^n$ be arbitrary. Then the curve $\rho(-tw,z)$ 
(i.e. the flow of the action through $z$ in the direction of $-w$) converges to a point in $\bar \cO$
when $t$ tends to $+\infty$.
\end{lemma}

\begin{proof}
Since $\bar \cO$ is compact and $\rho(-tw,z) \in \bar \cO$ for all $t \in \bbR$, there exists
a point $q \in \bar \cO$ and a sequence $(t_\nu )_{\nu \in \bbN}$ of positive numbers such that 
$t_\nu \to +\infty$ and $\rho(-t_\nu w, z) \to q$ when $\nu \to +\infty$. We will show that
$\rho(-tw, z) \to q$ when $t \to +\infty$.

If $q \in \cO$ then $w =0$ and the statement is obvious. Indeed, if $q \in \cO$ then the fact that 
$\rho(-t_\nu w,z) \underset{\nu \to +\infty}{\longrightarrow} q \ $ means that there 
exist $w_\nu \in \bbR^n, w_\nu \to 0$ when $\nu  \to \infty$
such that $\rho(-t_\nu w, z) = \rho(w_\nu, q) \quad \text{for all } \nu \in \bbN$.
It implies that:
\begin{multline}
\rho((t_\mu -t_\nu )w, z) = \rho(t_\mu w, \rho(-t_\nu w, z)) \\
= \rho(t_\mu w, \rho(w_\nu, q)) = \rho(t_\mu w+w_\nu  -w_\mu, \rho(w_\mu, q))\\
= \rho(t_\mu w+w_\nu  -w_\mu, \rho(-t_\mu w, z)) = \rho(w_\nu  -w_\mu,z).
\end{multline} 
Since the action on $\cO$ is free, we have:
\begin{equation}
(t_\mu -t_\nu )w = w_\nu -w_\mu.
\end{equation}
But $w_\nu -w_\mu \to 0$ when $\nu, \mu  \to \infty$ on one hand, and $t_\mu -t_\nu$ 
can be arbitrarily large on the other hand, therefore $w=0$.

We can assume now that $w \not= 0$ and $q \notin \cO$, i.e. $q \in \partial \cO$ is a singular point of corank $k \geq 1$. 
Then there is a neighborhood $\cU$ of $q$ in $M$ with a canonical coordinate system $(x_1, \hdots, x_n)$
in $\cU$ and elements $v_1,\hdots,v_n \in \bbR^n$ such that 
\begin{equation}
\cO_q \cap \cU = \{ (x_1, \hdots, x_n) \in \cU \ |\ x_1 = 0, \hdots, x_k = 0\}
\end{equation}
and $v_i$ is associated to the $(n-1)$-face $\{x_i =0\}$, i.e. $X_{v_i} = x_i\frac{\partial}{\partial x_i}
 \text{ near } \cU \cap \{x_i =0\}$. 
 
We can assume that $\cO \cap \cU$ is the corner 
\begin{equation}
\{ (x_1, \hdots, x_n) \in \cU \ |\ x_i > 0 \quad \forall i = 1, \hdots,k\}.
\end{equation}
Observe that $w \in \bbR \langle v_1, \hdots, v_k\rangle$. Indeed, assume that 
$w \notin \bbR \langle v_1, \hdots, v_k\rangle$, and $\nu, \mu \in \bbN$ such that 
$|t_\nu  - t_\mu |$ is large and $\rho(-t_\nu w, z), \rho(-t_\mu w, z) \in \cU \cap \cO$
very close to $q$.

Then the curve $\rho(-tw,z)$ for $t$ near $t_\nu $ cuts the $k$-dimensional set 
\begin{equation}
\{ (x_1, \hdots, x_k, 0, \hdots, 0) \in \cU \ |\ x_i > 0, \quad \forall i \leq k\} \subset \cO
\end{equation}
transversally at a point $z_1$, i.e. $z_1 = \rho(-s_1w,z)$ for some $s_1$ near $t_\nu $, and similarly 
\begin{equation}
z_2 = \rho(-s_2w,z) \in \{ (x_1, \hdots, x_k, 0, \hdots, 0) \in \cU \ |\ x_i > 0, 
\quad \forall i \leq k\} \subset \cO
\end{equation}
for some $s_2$ near $t_\mu $. Then $z_2 = \rho((s_2-s_1)w, z_1)$
and $s_2-s_1 \not= 0$.

On the other hand, since $z_1 = \{ (x_1^1, \hdots, x_k^1, 0, \hdots, 0) \in \cU \cap \cO \}$ and 
$z_2 = \{ (x_1^2, \hdots, x_k^2, 0, \hdots, 0) \in \cU \cap \cO \}$, we have $\displaystyle z_1 = \rho(\sum_{i=1}^k\log(\frac{x_i^1}{x_i^2})v_i, z_2)$,
which implies that 
$\displaystyle \rho(\sum_{i=1}^k\log(\frac{x_i^1}{x_i^2})v_i + (s_2-s_1)w, z_1) = z_1$,
which is a contradiction, since $\rho$ is free on $\cO$ and $w$ is linearly independent of $(v_1, \hdots, v_k)$.

Thus we must have $\displaystyle w = \sum_{i=1}^k s_iv_i$ for some $s_i \in \bbR$.

Our next step is to show that $s_i > 0 \quad \forall i = 1, \hdots, k$. Indeed, 
if for example $s_1 \leq 0$, then the vector field $-X_w$ is neutral or expulsive in the coordinate 
$x_1$ in $\cU$, and hence its flow starting from a point can never get arbitrarily close to $q$:
in fact if the flow (in the positive time direction) passes through at a point in $\cU$  but 
outside $V$, where  
\begin{equation}
V = \{ (x_1, \hdots, x_n) \in \cU, |x_1| \leq \varepsilon \},
\end{equation}
then it can never enter $V$ because $|x_1|$ can't decrease. The rest of the proof is straightforward. 
\end{proof}

Let us now fix a point $z_0 \in \cO$. For each orbit $\cH$ in $\bar \cO$ of any dimension, denote by
\begin{equation}
C_{\cH} = \{ w \in \bbR^n \ |\ \lim_{t \to +\infty}\rho(-tw,z_0) \in \cH \} 
\end{equation}
the set of all elements $w \in \bbR$ such that the flow of the action through $z_0$ in the direction $-w$ tends to a point in $\cH$.

\begin{prop} \label{prop:simpli_cone} Assume that $\bar \cO$ is compact, with the above notations, we have: \\
1) $C_\cH$ does not depend on the choice of $z_0\in \cO$. \\
2) $C_\cO = \{0\}$ and $C_{F_i} = \bbR_{>0}.v_i$ for each $(n-1)$-dimensional orbit $F_i \subset \bar \cO$, 
where $v_i \in \bbR^n$ is the vector associated to $F_i$ with respect to the action $\rho$.\\
3) If $w \in C_\cH$ then $X_w = 0$ on $\cH$.\\
4) $\bar C_\cH$ is a simplicial cone in $\bbR^n$ (i.e. a convex cone with a simplicial base) and $\dim C_\cH + \dim \cH = n$. \\
5) The family ($C_\cH; \cH$ is an orbit in $\bar \cO$)
is a partition of $\bbR^n$. \\
6) $C_\cK \subset  \bar C_\cH$ if and only if $\cH \subset \bar \cK$
and in that case $C_\cK$ is a face of $\bar C_\cH$. \\
 
\end{prop}

\begin{proof}
 1) Let $z_1 = \rho(\theta,z_0)$ be another point of $\cO$. Then by commutativity we have 
\begin{equation}
\lim_{t \to +\infty}\rho(-tw,z_1) =  \rho(\theta,\lim_{t \to +\infty}\rho(-tw,z_0)),
\end{equation}
and therefore $\underset{t \to +\infty}{\lim}\rho(-tw,z_1) \in \cH$ if and only if $\underset{t \to +\infty}{\lim}\rho(-tw,z_0) \in \cH$.

2) It follows directly from Lemma \ref{lemma:limit}.

3) If $\underset{t \to +\infty}{\lim}\rho(-tw,z_0) = p \in \cH$ then by commutativity we have
\begin{equation}
\rho(sw,p) = \lim_{t \to +\infty}\rho((-t+s)w,z_0) = p \quad \forall s \in \bbR,
\end{equation}
which implies that $X_w(p) = 0$. For any other point $q = \rho(\theta, p) \in \cH$ we will also have 
$X_w(q) = \rho(\theta,.)_*X_w(p) = 0$ by commutativity.

4) Consider a canonical coordinate system $(x_1, \hdots, x_n)$ in a neighborhood $\cU$ of a point 
$p \in \cH$ of HE-invariant $(h,0)$, i.e. $\dim \cH = n -h$.

We can assume 
\begin{equation}
\cO \cap \cU = \{(x_1, \hdots, x_n) \in \cU \ |\ x_1>0, \hdots, x_h >0 \}.
\end{equation}
Denote by 
\begin{equation}
F_i = \{(x_1, \hdots, x_n) \in \cU \cap \cO \ |\ x_i = 0 \} \quad (i = 1, \hdots, h)
\end{equation}
the $h$ facets of $\cO$ adjacent to $\cH$, and by $v_1, \hdots, v_h \in \bbR^n$ 
their associated vectors in $\bbR^n$, 
i.e. $X_{v_i} = x_i\frac{\partial}{\partial x_i}$ near $F_i$.

If $w = \sum_{i=1}^h\alpha_iv_i$ with $\alpha_i >0 \quad \forall i = 1, \hdots, h$ then
\begin{multline}
 \rho(-tw, (\varepsilon_1, \hdots, \varepsilon_h, 0,\hdots,0)) = (e^{-t}\varepsilon_1, \hdots, e^{-t}\varepsilon_h, 0,\hdots,0)\\
\underset{t \to +\infty}{\longrightarrow} p = ( 0,\hdots,0) \in \cU,
\end{multline}
 and so $w \in C_\cH$ by definition. 

Conversely, if $w \in C_\cH$ then $w$ must be of the type $\sum_{i=1}^h\alpha_iv_i$ with $\alpha_i >0$.
Indeed, if $w \notin Span_{\bbR}(v_1, \hdots, v_h)$ then $X_w(p) \not= 0$ which contradicts 3), so
we must have $w \in Span_{\bbR}(v_1, \hdots, v_h)$. If $w = \sum_{i=1}^h\alpha_iv_i$ with $\alpha_i \geq 0$
for some $i$, then $x_i(\rho(-tw,z_0))$ does not decrease to 0 when $t \to +\infty$, and hence
$\rho(-tw,z_0)$ cannot tend to 0 when $t \to +\infty$ either. Thus, in order for $\rho(-tw,z_0)$ to tend to a point 
in $\cH$ when $t \to +\infty$, we must have $w = \sum_{i=1}^h\alpha_iv_i$ with $\alpha_1 >0, \hdots, \alpha_h >0$.

In conclusion, we have the following formula:
\begin{equation}
 C_\cH = \{ \sum_{i=1}^h\alpha_iv_i \ |\ \alpha_1 >0, \hdots, \alpha_h >0 \},
\end{equation}
and $\bar C_\cH = \{ \sum_{i=1}^h\alpha_iv_i \ |\ \alpha_1 \geq 0, \hdots, \alpha_h \geq 0 \}$
 is a $h$-dimensional simplicial cone, because $\alpha_1, \hdots, \alpha_h $ are linearly independent.

5) The fact that $\underset{\cH}{\bigcup} C_\cH = \bbR^n$ follows from Lemma \ref{lemma:converge}.
The fact that $C_\cH \cap C_\cK = \emptyset$ when $\cH \not= \cK$ (and therefore $\cH \cap \cK = \emptyset $)
is obvious from the definition. 

6) It follows directly from the above formula for $C_\cH$.
\end{proof}

In the literature, a partition 
$$\bbR^n = \bigsqcup_\cH C_\cH$$
together with a family of vectors $v_i$, with the properties as listed in the above proposition,
is called a {\bf complete (simplicial) fan} (over $\bbR$), see, e.g. \cite{Ewald-Combinatorial1996}, \cite{Ishida-toric}.
More precisely, we have the following definition:
\begin{defn} \label{def:fan}
A {\bf complete fan} in $\bbR^n$ is a set of data $(C_\cH, v_i)$ (here $\cH$ and $i$ are indices) such that:

i) $(C_\cH)$ is a finite partition of $\bbR^n$, i.e. $\bbR^n$ is the disjoint union of this family  $(C_\cH)$.

ii) Each $\bar  C_\cK$ (i.e. the closure of  $C_\cK$) is a convex simplicial cone in $\bbR^n$ 
and $\bar  C_\cK \backslash C_\cK$ is the boundary of the cone $\bar  C_\cK$.

iii) If $\bar  C_\cK \backslash C_\cK \not= \emptyset$ (i.e. $C_\cK \not= \{0\}$) then each face of 
$\bar C_\cK$ is again an element of the family $(C_\cH)$.

iv) Each 1-dimensional $C_{\cK_i}$ contains exactly one element $v_i : C_{\cK_i} = \bbR_{>0}.v_i$.
In particular, the number of $v_i$'s is equal to the number of 1-dimensional components (half-lines) in the partition $(C_\cH)$.
\end{defn}

Figure \ref{fig:thefan} is an illustration of the construction of the associated fan for a hyperbolic domain.
\begin{figure}[htb]
\begin{center}
\includegraphics[width=90mm]{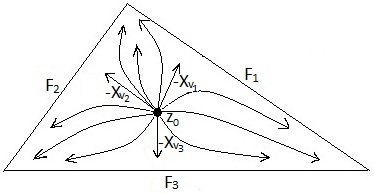}
\caption{The fan at $T_{z_0}M^n \cong \bbR^n$.}
\label{fig:thefan}
\end{center}
\end{figure}

Proposition \ref{prop:simpli_cone} tells us that to each compact closed hyperbolic domain $\bar \cO$ there is
a naturally associated complete fan of $\bbR^n$, which is an invariant of the action. The following theorem shows that, 
conversely, any complete fan can be realized, and is the full invariant of the action on a compact closed hyperbolic domain.

\begin{thm}[Classification of hyperbolic domains by fans] \label{thm:classificationByfan}
 1) Let $(C_\cH, v_i)$ be a complete fan of $\bbR^n$. Then there exists a totally hyperbolic action  
$\rho: \bbR^n \times M^n \to M^n$ on a compact closed manifold $M^n$ (without boundary)
with a hyperbolic domain $\cO$ such that the associated fan to $(\bar \cO, \rho)$ is $(C_\cH, v_i)$.

2) If there are two compact closed hyperbolic domains $(\bar \cO_1, \rho_1)$ and $(\bar \cO_2, \rho_2)$ of two
actions $\rho_1$ and $\rho_2$, which have the same associated complete fan $(C_\cH, v_i)$, then there is 
a diffeomorphism from $\bar \cO_1$ to $\bar \cO_2$ which intertwines $\rho_1$ and $\rho_2$.
\end{thm}
\begin{proof}
 1) We can use the following gluing method to construct $(\bar \cO, \rho)$:

 For each component $C_\cH$ of the fan, denote by $v_1, \hdots, v_h$ the vectors of the fan which lie on the edges
of $\bar C_\cH$ (the number $h$ is equal to the dimension of $C_\cH$). Complete $v_1, \hdots, v_h$ in a arbitrary way 
to obtain a basis $(v_1, \hdots, v_h, w_{h+1}, \hdots, w_n)$ of $\bbR^n$. Put
\begin{equation}
\cD_\cH \cong \{(x_1, \hdots, x_n) \in \bbR^n \ |\ x_1 \geq 0, \hdots, x_h\geq 0\}
\end{equation}
and denote by $\rho_\cH$ the $\bbR^n$-action on $\cD_\cH$ such that the corresponding generators
 $X_{v_1}, \hdots, X_{v_h}, X_{w_{h+1}}, \hdots, X_{w_n}$ are 
\begin{equation}
 X_{v_1} = x_1\frac{\partial}{\partial x_1}, \hdots, X_{v_h} = x_h\frac{\partial}{\partial x_h},
X_{w_{h+1}}= \frac{\partial}{\partial x_{h+1}}, \hdots, X_{w_n} = \frac{\partial}{\partial x_n}.
\end{equation}
Fix an arbitrary point $z_\cH$ in the interior of $\cD_\cH$, we get a ``local model'' $(\cD_\cH, \rho_\cH, z_\cH) $
of our construction. For the moment, these models are disjoint, i.e. $\cD_\cH \cap \cD_\cK = \emptyset$ if $\cH \not= \cK$.

We will now glue all these local models together, by the following equivalence relationship $\sim$:

i) $\rho_\cH (\theta, z_\cH) \sim \rho_\cK (\theta, z_\cK)\quad \forall \theta \in \bbR^n, \quad \forall \cH, \cK$

ii) If $y_t \to y_\infty$ in $\cD_\cH$, $y_t' \to y_\infty'$ in $\cD_\cK$ and $y_t \sim y_t' \quad \forall t$, 
then $y_\infty \sim y_\infty'$ (continuity principle).

Put 
\begin{equation}
\bar \cO = (\bigsqcup_\cH \cD_\cH)/\sim.
\end{equation}
Obviously, the actions $(\rho_\cH)$ are compatible and induce an action $\rho$ of $\bbR^n$ on $\bar \cO$. It is
easy to verify that $(\bar \cO, \rho)$ has the required properties: $\bar \cO$ is a manifold with boundary and corners, $\rho$ is totally
hyperbolic, and the complete fan associated to $(\bar \cO, \rho)$ is nothing but our fan $(C_\cH, v_i)$.

In order to construct $(M, \rho)$ without boundary which contains $(\bar \cO, \rho)$, we can use the reflection principle
(Theorem \ref{thm:Reflection}). Indeed, one can glue together $2^m$ copies $\bar \cO_\alpha$ 
of $\bar \cO$, indexed by the elements of the group
$(\bbZ_2)^m$, where $m$ is the number of facets of $\bar \cO$ (i.e. the numbers of $v_i$'s), by the following rule:

Glue the facet number $i$ of $\bar \cO_\alpha$ to the facet number $i$ of $\bar \cO_\beta$  (by the identity map) 
if and only if $\alpha - \beta = (0, \hdots, 1, 0, \hdots, 0)$ is the $i$-th generator of $(\bbZ_2)^m$.

The result is a compact manifold $M^n$ without boundary, on which $(\bbZ_2)^m$ acts by involutions, 
such that $\cO$ is a fundamental domain:
\begin{equation}
M^n/ (\bbZ_2)^m \cong \bar \cO.
\end{equation}
Then we can pull back the $\bbR^n$ action from $\bar \cO$ to $M^n$ via the projection map $M^n \to \bar \cO$
in order to get a totally hyperbolic action on $M$ which has $\bar \cO$ as a closed hyperbolic domain.

2) Take any two points $z_1 \in \cO_1$ and $z_2 \in \cO_2$. Define 
\begin{equation}
\Phi (z_1) =z_2, \Phi(\rho_1(\theta,z_1)) = \rho_2(\theta,z_2)
\end{equation}
for all $\theta \in \bbR^n$, and then extend $\Phi$ to the boundary of $\bar \cO$ 
by continuity. The fact that $(\bar \cO_1, \rho_1)$ and $(\bar \cO_2, \rho_2)$ have the same associated 
complete fan ensures that the constructed map $\Phi : \bar \cO_1 \to \bar \cO_2$ is a diffeomorphism, which sends 
$\rho_1$ to $\rho_2$. 
 \end{proof}
\noindent {\it Remark.} It was pointed out to us by a referee that the manifold $M^n$ in the above proof is in fact 
the so called {\it real moment-angle manifold} associated to the underlying simplicial complex of the fan, see Section 6.6 
of  \cite{BP_Torus2002}. The adjective ``real'' here means that we have an action of $(\bbZ/2\bbZ)^m$ (instead of $\Pi^m$).

\begin{thm} \label{thm:hyperbolic-contractible}
If $\bar \cO$ is a compact closed hyperbolic domain of a totally hyperbolic action $\rho$
then $\bar \cO$ is contractible.
\end{thm}
\begin{proof}
$\bar \cO$ can also be partioned into a \emph{nolinear compact fan} similar to its associated
complete fan as follows:

Fix a point $z_0 \in \bar \cO$, and for each orbit $\cH \subset \bar \cO$  put
\begin{equation}
\cD_\cH = \{ z \in \bar \cO \ |\ z = \rho(-w,z_0) \text{ or } z = \lim_{t \to -\infty} \rho(tw,z_0)
\text{ for some } w \in C_\cH\}.
\end{equation}
Clearly, each $\bar \cD_\cH$ is diffeomorphic to a $h$-dimensional cube, where $h$ = corank $\cH$,
and we can contract $\bar \cO$ to $z_0$ by contracting it ``cell by cell'' (each $\cD_\cH$ is a cell):
First kill the highest-dimensional cells, then kill the next-to-highest dimensional cells, and so on, see Figure \ref{fig:contracting}
for an illustration.
\end{proof}
\begin{figure}[htb]
\begin{center}
\includegraphics[width=100mm]{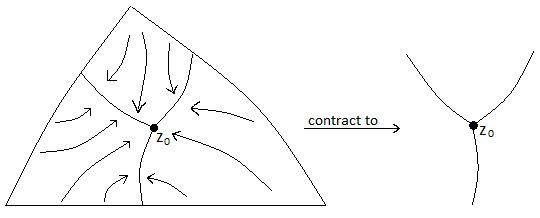}
\caption{Contracting $\bar \cO$ to a point.}
\label{fig:contracting} 
\end{center}
\end{figure}

\subsection{Polytope or not polytope?}

Recall that, a simple convex polytope is a convex polytope such that each vertex is simplicial,
i.e. has exactly $n$ adjacent edges, where $n$ is the dimension. 

A simplicial polytope is a polytope such that every facet is a simplex. If $P \subseteq \bbR^n$
is a simple convex polytope, then the dual polytope 
\begin{equation}
P^* = \{ x \in (\bbR^n)^* \ |\ \langle x,y \rangle \leq 1 \quad \forall y \in P\}
\end{equation}
is a convex simplicial polytope and vice versa.

The results of the previous subsection show that any compact closed hyperbolic domain $\bar \cO$
looks like a convex simple polytope: $\bar \cO$ has boundary and corners, each vertex of $\bar \cO$
is simplicial, and moreover $\bar \cO$ is contractible, and the same is true for each face of $\bar \cO$.
It is easy to see that $\bar \cO$ can be diffeomorphically embedded into $\bbR^n$ in a non-linear way 
(i.e. with non-linear boundary). So we may say that $\bar \cO$ is a ``curved polytope''.

The question now is: Is $\bar \cO$ diffeomorphic to a convex polytope in $\bbR^n$?

Surprisingly, the answer is not always Yes, though it is obviously Yes in dimension 2:

\begin{thm} \label{thm:diff-to-polytope}
Any compact closed hyperbolic domain $\bar \cO$ of dimension $n \leq 3$ is diffeomorphic to a convex
simple polytope. If $n \geq 4$ then there exists a compact closed hyperbolic domain $\bar \cO$ of dimension $n$
which is not diffeomorphic to a polytope. 
\end{thm}

\begin{proof}
The case $n=2$ is obvious. The case $n=3$ is a consequence of the classical Steinitz theorem. 
When $n=4$ of higher, there are counterexamples:
The first known counterexample comes from the so-called Barnette's sphere \cite{Barnette-Diagram1970}, which
was pointed out by Ishida, Fukukawa and Masuda in \cite{Ishida-toric} for a similar problem. 
The idea of our proof here is also taken from \cite{Ishida-toric}.

The Barnette's sphere is a simplicial complex whose ambient space is a 3-dimensional sphere $S^3$,
but which cannot be realized as the boundary of a convex simplicial polyhedron in $\bbR^3$ for some reasons of 
combinatorial nature.

It is known \cite{Ewald-Combinatorial1996} that Barnette's sphere can be realized as the base of a complete fan in $\bbR^4$,
which we will call the Barnette fan. Take the hyperbolic domain $\cO$ given by this Barnette fan. Then $\bar \cO$ cannot be
diffeomorphic to a complex simple 4-dimensional polytope, because if there is such a polytope, 
then the boundary of the simplicial polytope dual to it will be a realization of the Barnette's sphere, which is a contradiction. 
\end{proof}
   
When $n=2$, one can construct a convex polytope diffeomorphic to $\bar \cO$ by the following cutting method:
\begin{figure}[htb]
\begin{center}
\includegraphics[width=50mm]{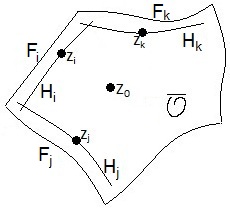}
\caption{Trimming $\bar \cO$ into a polytope.}
\end{center}
\end{figure}

For each face $F_i$ of $\bar \cO$ denote by $z_i$ a point in $\cO$ 
``very near'' $F_i$, and by 
\begin{equation}
H_i = \{ \rho(w,z_i) \ |\ w \in \bbR^2, \langle w, v_i\rangle = 0  \}
\end{equation}
the ``hyperplane'' in $\cO$ of the action through $z_i$ which is ``orthogonal''
to $v_i$. Here $\langle.,.\rangle$ denotes a standard scalar product in $\bbR^2$, and $v_i \in \bbR^2$ is
the vector associated to $F_i$.

Then $H_i$ is ``nearly parallel'' to $F_i$ in the sense that the points $\rho(w,z_i)$ remain close
to $F_i$ when $w$ is not too large. $H_i$ cuts $\cO$ into 2 pieces which we will denote by $\cO_{i+}$, $\cO_{i-}$,
where $\cO_{i+}$ is the piece which contains the chosen point $z_0 \in \cO$. Put 
\begin{equation}
\bar \cO_+ = \bigcap_i \cO_{i+}.
\end{equation} 
It is then not difficult to verify that $\bar \cO$ is diffeomorphic to $\bar \cO_+$,
and on the other hand, $\bar \cO_+$ is a convex polygone with respect to the affine structure
on $\cO$ given by the $\bbR^2$-action $\rho$.

This cutting method probably still works in dimension $n=3$, but clearly it fails in dimension $n\geq 4$
because there are counterexamples like the Barnette's sphere.

\subsection{Existence  of totally hyperbolic actions} \label{subsection:hyperbolic_existence}
In the case of dimension 2, the existence of a totally hyperbolic actions on any closed
2-manifold was known to Camacho \cite{Camacho-MorseSmaleAction1973}, who used the term ``Morse-Smale 
$\bbR^2$-flows on a 2-manifold'' for what we call a nondegenerate action. We have here
a sightly improved result, wich includes the non-orientable case and has
the minimal number of hyperbolic domains.

\begin{thm} \label{thm:hyperbolic_dim2}
1) On $\bbS^2$ there exists a totally hyperbolic $\bbR^2$ action, which has
exactly 8 hyperbolic domains. The number 8 is also the minimal number possible:  
any totally hyperbolic action of $\bbR^2$ on $\bbS^2$ must have at least 8 hyperbolic domains.

2) For any $g \geq  1$, on a closed orientable surface of genus $g$ there 
exists a totally hyperbolic action of $\bbR^2$ which has exactly 4 hyperbolic domains. 
The number 4 is also minimal possible.

3) Any non-orientable closed surface also admits a totally hyperbolic action with 4 
hyperbolic domains, and the number 4 is also minimal possible.
\end{thm}
\begin{proof}
\emph{Existence:}

1) Cut $\bbS^2$ into 8 pieces by 3 loops so that each piece is a trigone, as shown 
in Figure \ref{fig:S2Sigma2}a. 
According to Theorem \ref{thm:NormalForm}, any simplest complete fan of $\bbR^2$ (partition of $\bbR^2$
into 3 convex cones, together with an arbitrary choice of 3 vectors on the 3 boundary
directions) will correspond to a hyperbolic $\bbR^2$-action on a trigone (as illustrated in
Figure \ref{fig:S2Sigma2}a). Using reflections as in Theorem \ref{thm:Reflection} and Theorem \ref{thm:classificationByfan}, 
we can pull back this action to an action of $\bbS^2$ via the projection map: $\bbS^2 \to \bbS^2/ (\bbZ_2)^3 = \text{trigone}$.
\begin{figure}[htb]
\begin{center}
\includegraphics[width=100mm]{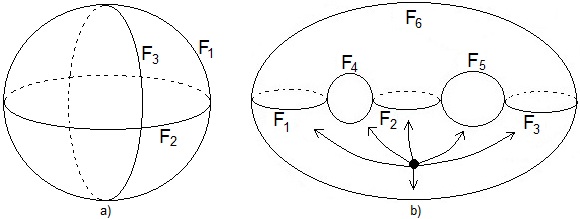}
\caption{Cutting $\bbS^2$ into 8 trigones and cutting $\Sigma_2$ into 4 domains.}
\label{fig:S2Sigma2}
\end{center}
\end{figure}

2) In the case of an orientable surface $\Sigma_g$ of genus $\geq  1$, we just need 2 involutions to cut it 
into four $(2g+2)$-gones as in Figure \ref{fig:S2Sigma2}b. Construct an action on one of these $(2g+2)$-gones, 
and extend it to the whole surface by the reflection principle 
(so that it becomes invariant with respect to the 2 involutions). 

3) Embed  $\Sigma_g$  (where $g\geq 0$) into $\bbR^3$ in such a way that
it is symmetric with respect to the 3 planes $\{x=0\}$, $\{y=0\}$, $\{z=0\}$, and is cut into $8$ polygones
by these planes (each polygone has $g+3$ edges). Like in the case of $\bbS^2$, we can construct a hyperbolic
action on one of these 8 polygones, and extend it to the other polygones by the reflections $\sigma_x$,
$\sigma_y$, $\sigma_z$ (where $\sigma_x$ is the reflection: $(x,y,z) \mapsto (-x,y,z)$
to get an action $\rho$ on $\Sigma_g$. Since $\rho$ is invariant with respect to the free involution
$\sigma = \sigma_x \circ \sigma_y \circ \sigma_z : (x,y,z) \mapsto (-x,-y,-z)$ on $\Sigma_g$, 
it projects to a hyperbolic action on the non-orientable
surface $\Sigma_g / \sigma$, which has 4 hyperbolic domains.

\emph{Minimality:}

1) Each loop on $\bbS^2$ cuts it into 2 disks with smooth boundary.
2 loops $\rightarrow $ there are at least 4 pieces which are 2-gones. Each 2-gone needs 
to be cut in order to obtain polygones (with at least 3 edges), so 
in total we have at least $4\times 2=8$ pieces, i.e. 
8 orbits of dimension 2 of the action.

2) and 3) The number 4 is minimal, due to the non overlapping of boundary of a hyperbolic domain: 
near a fixed point of the action, the 4 corners must belong to 4 different domains. 
More generally, for a totally hyperbolic action on a compact $n$-dimensional manifold without boundary,
we must have at least $2^n$ different hyperbolic domains.
\end{proof}

\noindent {\it Remark.} It was pointed out to us by a referee that the constructions in the above theorem correspond to the so 
called {\it real topological toric manifolds} introduced by Ishida-Fukukawa-Masuda \cite{Ishida-toric} in the case of dimension 2.

In the general case of dimension $n \geq 3$, we have the following conjecture:
\begin{conj} \label{conjecture:hyperbolic}
Any closed smooth $n$-dimensional manifold admits a completely hyperbolic nondegenerate action of $\bbR^n$.
\end{conj}
If Conjecture \ref{conjecture:hyperbolic} is not true, i.e. there are obstructions for a
manifold to admit a totally hyperbolic action, then there might be an obstruction in the torsion 
part of the first homology group, due to the monodromy. For example, we don't even know yet 
if any lens space (which is a rational homology 3-sphere) different from $\bbS^3$ admits a totally 
hyperbolic $\bbR^3$-action or not.

A related interesting question is: given a manifold $M^n$, what is the minimal number 
that a totally hyperbolic $\bbR^n$-action on it must have, and how is this number related to the
other topological invariants of $M^n$?

If the above conjecture is true, it would mean that:

a) Any closed smooth manifold can be decomposed into ``curved polytopes'' by embedded closed hypersurfaces
$F_i$ which intersect transversally.

b) Moreover, one can associate to each hypersurface $F_i$ in a) an element $v_i \in \bbR^n$
such that each ``curved polytope'' corresponds to a complete fan of $\bbR^n$ compatible 
with the $v_i$'s of its faces.

The property a) is easy to achieve, one can cut any manifold into polyhedral pieces by smooth 
surfaces. But the property b) is highly non-trivial, because not any polyhedral decomposition 
by smooth tranversally intersecting hypersurfaces can be realized by totally hyperbolic $\bbR^n$-action,
as the counterexamples in the following proposition show.

\begin{prop} \label{prop:3-6domains}
There does not exist a totally hyperbolic action which contains 3 domains $\cO_1, \cO_2,\cO_3$ 
as in Figure \ref{fig:3domain}a or 6 domains $\cO_1, \hdots,\cO_6$ as in Figure \ref{fig:3domain}b.
\end{prop} 
\begin{figure}[htb] 
\begin{center}
\includegraphics[width=120mm]{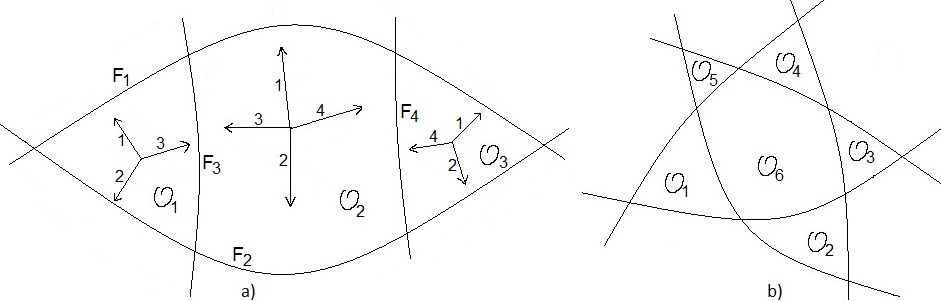}
\caption{Impossible configurations.}
\label{fig:3domain}
\end{center}
\end{figure}

\begin{proof}

Assume that there is a totally hyperbolic action which contains three domain $\cO_1, \cO_2,\cO_3$ as in 
Figure \ref{fig:3domain}a. 
Denote by $v_1,v_2,v_3,v_4$ the vectors associated to the curves $F_1,F_2,F_3,F_4$ respectively.
Since $v_1,v_2,v_3$ form the fan of $\cO_1$, we must have: $v_3 = \alpha v_1 + \beta v_2$ for some $\alpha, \beta  < 0$.
Similarly, looking at the fan of $\cO_3$, we have: $v_4 = \gamma  v_1 + \delta  v_2$ for some $\gamma, \delta < 0$.
But looking at the fan of $\cO_2$, we have that either $v_3$ or $v_4$ is a positive linear combination of $v_1$ and $v_2$.
This is contradiction.

The proof of impossibility of the configuration on Figure \ref{fig:3domain}b is similar.
\end{proof}
An interesting question of combinatorial nature is: what are the necessary and sufficient conditions for a graph on a surface $\Sigma$
to be the singular set of a totally hyperbolic action of $\bbR^2$ on $\Sigma$?
\subsection{Classification of totally hyperbolic actions}

For a nondegenerate totally hyperbolic action $\rho: \bbR^n \times M^n \to M^n$, we have the following
set of invariants:

I1) Smooth invariant hypersurfaces of $M^n$ which intersect transversally and which cut $M^n$
into a finite number of ``curved polytopes'', which are hyperbolic domains of the action.

I2) The family of fans: a fan for each domain. (Two fans of two adjacent domains will share a common
vector, which is the vector associated to the corresponding hypersurface).

I3) The monodromy.

The above set of invariants also completely determines $\rho$ up to isomorphisms:

\begin{thm}[Classification of totally hyperbolic actions] \label{class-hyp-invariants}
Nondegenerate totally hyperbolic actions of $\bbR^n$ on connected $n$-manifolds
(possibly with boundary and corners) are completely determined by their invariants I1, I2, I3 listed above.
In other words, assume that $(M_1^n, \rho_1)$ and $(M_2^n, \rho_2)$ are totally hyperbolic actions,
such that there is a homeomorphism $\varphi : M_1^n \to M_2^n$ which sends hyperbolic domains of $(M_1^n, \rho_1)$
to hyperbolic domains of $(M_2^n, \rho_2)$, such that the monodromy and the associated fans 
are preserved by $\varphi$, then there is a diffeomorphism $\Phi : M_1^n \to M_2^n$ which sends $\rho_1$ to $\rho_2$.
\end{thm}

\begin{proof}
 The action-preserving diffeomorphism $\Phi$ from $M_1^n$ to $M_2^n$ can be constructed as follows:
 
 i) Fix an arbitrary regular point $z_0 \in M_1^n$, and put 
 $$\Phi(z_0) = \varphi (z_0)$$
 
 ii) Extend $\Phi$ to $\cO_{z_0}$ by the formula $\Phi(\rho_1(v,z_0))= \rho_2(v,\Phi(z_0))$
 for any $v \in \bbR^n$, and then extend it to $\bar \cO_{z_0}$ by continuity. The fact that the fans
 associated to $\cO_{z_0}$ and $\cO_{\varphi (z_0)}$ are the same assures that 
$\Phi: \bar \cO_{z_0} \to \bar \cO_{\varphi (z_0)}$ is a diffeomorphism, according to 
Theorem \ref{thm:classificationByfan}.

iii) Extend $\Phi$ to the other domains of $M_1^n$ by the reflection principle. The fact that we have the same
fans on $M_2^n$ as on $M_1^n$, and also the same monodromy, assures that this extension is well-defined and smooth.
Thus we obtain the required smooth isomorphism $\Phi : (M_1^n,\rho_1) \to  (M_2^n,\rho_2)$.
\end{proof}

\section{Reduction by  associated torus action}
\label{section:QuotientSpace}

\subsection{Quotient space and reduced action}
We will denote by 
\begin{equation}
 t(\rho)
\end{equation}
the toric degree of a nondegenerate action $\rho: \bbR^n \times M^n \to M^n$,
and also
\begin{equation}
r(\rho) = n-  t(\rho).
\end{equation}
Recall that $(r(\rho),t(\rho))$ is also the RT-invariant of a regular orbit of $\rho$.

In this section, we will look into the structure of the associated torus action $\rho_\bbT$
of $\rho$, and the quotient space 
\begin{equation}
 Q = M^n/\rho_\bbT 
\end{equation}
of $M^n$ by $\rho_\bbT$. Recall that $\rho_\bbT$
is an action of the group $(Z_\rho \otimes \bbR)/Z_\rho \cong \bbT^{t(\rho)}$.

For each point $z \in M$, we will denote by
\begin{equation}
Z_{\rho_\bbT}(z) = \{\gamma \ |\ \rho_\bbT(\gamma,z) = z\} 
\end{equation}
the isotropy of $\rho_\bbT$ at $z$. In particular, if $z$ is a regular point of $(M,\rho)$,
then $Z_{\rho_\bbT}$ is trivial, according to the proof of Theorem \ref{thm:HERT-toricdegree}.
Denote by $(h,e,r,t)$ the HERT-invariant of $z$.

\begin{prop} \label{prop:describe-sing-Tk}
With the above notations, we have: \\
1) The isotropy group $Z_{\rho_\bbT}(z)$ is of the type
\begin{equation}
Z_{\rho_\bbT}(z) \cong \bbT^{e} \times G_z,
\end{equation}
where $G_z$ is the twisting group of $\rho$ at $z$ given by Definition \ref{defn:TwistingGroup}. \\
2) There is a $\rho_\bbT$-invariant neighborhood $\cU$ of $z$ in $M^n$
such that  $(\cU, \rho_\bbT)$ is isomorphic, 
after an identification of $(Z_\rho \otimes \bbR)/Z_\rho$ with $\bbT^{t(\rho)}$, 
to the following linear model:

i) $\cU \cong D_1^2 \times \hdots \times D_{e}^2\times (\bbT^{t}\times B^{h}/G_z) \times B^r$,
where $D_1^2, \hdots, D_{e}^2$ are 2-dimensional disks, $B^{h}$ and $B^r$ are balls of respective dimensions.

ii) The action of $\bbT^{t(\rho)} = \bbT_1^1 \times \hdots \times \bbT_{e}^1\times \bbT^{t}$ 
on $\widetilde \cU = D_1^2 \times \hdots \times D_{e}^2\times \bbT^{t(s)}\times B^{h} \times B^r$ 
is the direct product of the actions of  $ \bbT_1^1, \hdots, \bbT_{e}^1, \bbT^{t}$ on 
$\widetilde \cU$. Each action is diagonal, i.e. it acts simultaneously on each component of $\widetilde \cU$.

iii) $\bbT_i^1$ acts on $D_i^2$ by the standard rotation, and on the other components trivially.
$\bbT^{t}$ acts on $\bbT^{t}$ (itself) by translations, and on the other  
components trivially.

iv) $G_z$ also acts freely and diagonally on $\widetilde \cU$; its actions on $D_i^2$'s and on $B^r$
are trivial, its action on $\bbT^{t(z)}$ is by translations, and its action on $D^{h}$
is by the inclusion of $G_z$ into $(\bbZ_2)^{h}$ generated by the involutions 
$\sigma_i: (x_1,\hdots,x_i,\hdots,x_{h}) \mapsto (x_1,\hdots,-x_i,\hdots,x_{h}) $
on $D^{h}$. (The action of $\bbT^{t(\rho)}$ on $\widetilde \cU$ commutes with the action of 
$G_z$, so it projects to an action on $\cU$). \\
3) The quotient space $M^n/\rho_\bbT$ of $M^n$ by $\rho_\bbT$ is locally isomorphic to
\begin{equation}
 (D^2_1 / \bbT^1_1) \times \hdots \times (D^2_{e} / \bbT^1_{e}) \times (B^{h} / G_z) \times B^r.
\end{equation}

 \end{prop}
 
 \begin{proof}
 The proof follows immediately from the semi-local normal form theorem (Theorem \ref{thm:semi-localform2}) 
 and  the fact that the isotropy of $\rho_\bbT$ at regular points of $(M, \rho)$ is trivial.
 \end{proof}
 
We can identify the quotient space $D^2_i / \bbT^1_i$ of each disk 
$D^2_i = \{(x_i,y_i) \in \bbR^2 \ |\ x_i^2 + y_i^2 < \epsilon^2 \}$ by the standard circle action with the
half-closed interval $\{(x_i,0) \in \bbR^2 \ |\ 0 \leq x_i < \epsilon \}$ via the intersection of this interval
with the orbits of the circle action. Then $(D^2_1/  \bbT^1_1) \times \hdots \times (D^2_e/\bbT^1_e)$
can be identified with the positive corner $\{(x_1\hdots,x_e) \in \bbR^e \ |\ 0 \leq x_i < \epsilon \ \forall \ i\}$.

When the twisting group $G_z$ is trivial, then the local model
$$(D^2_1 / \bbT^1_1) \times \hdots \times (D^2_e / \bbT^1_e) \times (B^h / G_z) \times B^r$$
of the quotient space $Q$ is a manifold with boundary and corners. When $G_z \neq 0$, then 
$B^{h(z)} / G_z$  is an orbifold (which can or cannot be viewed as a manifold with boundary and corners,
depending on how $G_z$ acts), and so the local model of the quotient space $Q$
is an orbifold with boundary and corners. 

\begin{example}
a) $h=2, G_z = (\bbZ_2)^2$ acting on $B^2$ with local coordinates $(y_1,y_2)$ by two
involutions $(y_1,y_2) \mapsto (-y_1,y_2)$ and $(y_1,y_2) \mapsto (y_1,-y_2)$.
Then $(B^2/G_z)$ is an orbifold which can also be viewed as a positive corner
$\{(y_1,y_2) \in \bbR^2 \ |\ y_1 \geq 0, y_2 \geq 0\}$.

b) $h=3, G_z = \bbZ_2$ acting on $B^3(y_1,y_2,y_3)$ by the involution $(y_1,y_2,y_3) \mapsto (-y_1,-y_2,-y_3)$.
Then $B^3/G_z$ is an orbifold which cannot be viewed as a manifold with boundary and corners. 
\end{example}
To summarize, we have:

\begin{prop}
The quotient space $Q  = M^n / \rho_\bbT$ is (homeomorphic to) an orbifold with boundary and corners. If
the twisting group $G_z$ is trivial for every point $z \in M^n$, then $Q$ is a manifold with boundary and corners.
\end{prop}

Remark that the action $\rho: \bbR^n \times M^n \to M^n$ naturally projects down to an action of
$\bbR^n / (Z_\rho \otimes \bbR) \cong \bbR^{r(\rho)}$  on the quotient space $Q$, which we will denote by $\rho_\bbR:$
\begin{equation}
 \rho_\bbR: \bbR^{r(\rho)} \times Q \to Q
\end{equation}
after an identification of $\bbR^n / (Z_\rho \otimes \bbR)$ with $\bbR^{r(\rho)}$ and call it the 
{\bf reduced action} of $\rho$. Note that the dimension of $Q$ is also
equal to $r(\rho) = n - t(\rho)$.

\begin{prop} \label{prop:NoTwisting}
If there is no twisting in $(M^n,\rho)$, i.e. all twisting groups arr trivial, 
then $Q$ is a manifold with boundary and corners, and
the induced action $ \rho_\bbR$ on $Q$ is nondegenerate totally hyperbolic.
\end{prop}

\begin{proof}
 The proof is also an immediate consequence of the semi-local normal form theorem.
\end{proof}

When there is twisting, $Q$ is only an orbifold, but we still want to say that the action $\rho_\bbR$ is nondegenerate
totally hyperbolic. So we have to generalize the notion of totally hyperbolic actions to orbifolds.

\begin{defn}  \label{defn:hyperbolic_orbifold}
 Let $Q$ be an orbifold which can be modeled as $Q = \widetilde{Q}/G$ where $\widetilde{Q}$ is a manifold with boundary and corners,
and $G$ is a discrete group which acts properly on $\widetilde{Q}$ in such a way that the isotropy group of the action at every point is finite.
Then an action $\rho$ of $\bbR^r$ on $Q$, where $r$ is the dimension of $Q$, 
will be called {\bf totally hyperbolic} if it can be lifted to a nondegenerate totally hyperbolic action $\widetilde{\rho}$
of $\bbR^r$ on $\widetilde{Q}$.
\end{defn}

\begin{thm}[Reduction to totally hyperbolic action] \label{thm:hypAction-quotientSpace}
Let $\rho: \bbR^n \times M^n \to M^n$ be a nondegenerate action of toric degree $t(p)$
on a connected manifold $M^n$, and put $r = r(\rho) = n - t(\rho)$. Then the quotient
space $Q = M^n/\rho_\bbT$ of $M^n$ by the associated  torus action $\rho_\bbT$ is an
orbifold of dimension $r$, and the reduced action $\rho_\bbR$ of $\bbR^n/(Z_\rho \otimes \bbR) \cong \bbR^r$
on $Q$ is totally hyperbolic.
\end{thm}
\begin{proof}
It is a combination of Theorem \ref{thm:lifting-monodromy} (Existence of a covering $\widetilde{M^n}, \widetilde{\rho}$ 
of $M^n, \rho$ such that $Z_{\widetilde{\rho}} = Z_\rho$ but the monodromy $\widetilde{\rho}$  is trivial),
Theorem \ref{thm:TwistingMonodromy} (which says that if the monodromy is trivial then there is no twisting), and 
Proposition \ref{prop:NoTwisting} (which says that the assertion of the theorem is true when there is no twisting).
 \end{proof}

Even though $Q = M^n / \rho_\bbT$ is just an orbifold in general, we can still define the monodromy map
$\mu_{\rho_{\bbR}} : H_1(Q,\bbZ) \to \bbR^n/ (Z_\rho \otimes \bbR) \cong \bbR^{n-t(\rho)}$ of the action $\rho_\bbR$
on $Q$, just like the case of actions on manifolds.
\begin{prop}
 We have the following natural commutative diagramme of monodromy maps
\begin{equation}\label{eq:diag-monodromy}
\xymatrix{ 
H_1(M^n,\bbZ) \ar^{\mu_{\rho}}[r]\ar^{proj.}[d]& \bbR^n/Z_\rho \ar^{proj.}[d]\\
H_1(Q,\bbZ) \ar^{\mu_{\rho_{\bbR}}}[r]        &    \bbR^n/ (Z_\rho \otimes \bbR)\\
} 
\end{equation}
where $proj.$ denotes the natural projection maps.
\end{prop}
\begin{proof}
The proof follows directly from the definition of monodromy. 
\end{proof}

\subsection{Existence of cross multi-sections for  $M^n \to Q$}
Assume for the moment that $(M^n, \rho)$ has no twistings. In this case, $Q$ is a manifold with boundary and corners,
and one can talk about cross sections of the singular torus fibration $M^n \overset{\bbT^{t(\rho)}}{\longrightarrow} Q = M^n/\rho_\bbT$
over $Q$. We will say that an embeded submanifold with boundary and corners $Q_c \subset M^n$ is a smooth {\bf cross section} of the 
singular fibration $M^n \to Q$ if the projection map $proj. : Q_c \to Q$ is a diffeomorphism. The existence of a cross section
is equivalent to the fact that the \emph{desingularization via blowing up} of $M^n \to Q$ is a trivial principal $\bbT^{t(\rho)}$-bundle.
(The blowing up process here does not change the quotient space of the action $\rho_\bbT$ on $M^n$, but changes every singular orbit 
of $\rho_\bbT$ into a regular orbit, and changes $M^n$ into a manifold with boundary and corners, see Figure \ref{desingularization} 
for an illustration. This blow-up process is a standard one, and it was used for example, by Dufour and Molino \cite{DufourMolino-AA}, in the construction
of action-angle variables near elliptic  singularities of integrable Hamiltonian systems).

\begin{figure}[htb] 
\begin{center}
\includegraphics[width=80mm]{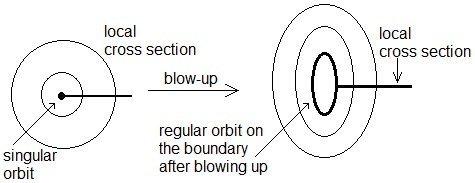}
\caption{Desingularization of $M^n \to Q$ by blowing up.}
\label{desingularization}
\end{center}
\end{figure}

\begin{prop}\label{prop:section}
Assume that $(M^n, \rho)$ has no twistings. Then the singular torus fibration $M^n \to M^n/\rho_\bbT = Q$ admits a smooth cross section $Q_c$.
\end{prop}
\begin{proof}(Sketch)
Consider first the case when the monodromy $\mu: \pi_1(M^n) \to \bbR^n/Z_\rho$ of $\rho$ is trivial. It means that we can choose in each 
regular orbit $\cO_i$ of $(M^n, \rho)$ a point $z_i \in \cO_i$, such that the family of point $\{z_i\}$ satisfies the following
symmetry condition: If $F$ is an $(n-1)$-dimensional orbits, $\sigma_F$ is the reflection map associated to $F$, and 
$\cO_i$ and $\cO_j$ are the two regular orbits adjacent to $F$, then $\sigma_F(z_i) = z_j$.
Starting from the points $z_i$'s, we will construct a (continuous, but not smooth in general) cross section $\Sigma$
to the singular fibration $M \to Q$ as follows:

For each $i \in I$ (where $I$ indexes the set of all regular orbits), denote by
\begin{equation}
 C_i = \{w \in \bbR^n \ |\ \exists \lim_{t \to \infty}\rho(tw,z_i)\} 
\end{equation}
and
\begin{equation}
\Sigma_i = \{\rho(w,z_0) \ |\ w \in C_i\}. 
\end{equation}
Similarly to the results of Section \ref{section:hyperbolic}, one can verify that      
$\bar \Sigma_{i}$  is a continuous section of $\bar \cO_{z_i}$ over $(\bar \cO_{z_i}/\rho_\bbT ) \subset Q$.
Moreover, everytime when $\bar \cO_i$ and $\bar \cO_j$ share a $(n-1)$-dimensional orbit, 
then $\bar \Sigma_i$ and $\bar \Sigma_j$ also share a common piece of boundary. It implies that the union
\begin{equation}
\Sigma = \bigcup_{i \in I}\bar \Sigma_i 
\end{equation}
is a continuous cross section of $M^n$ over $Q$. This section is not smooth, but its existence implies the triviality 
of the desingularization by blowing-up $\check{M^n}$ of $M \to Q$ ($\check{M^n} \to Q$ 
is a principal $\bbT^{t(\rho)}$-fibration), which in turn implies the existence of a smooth section.

When the monodromy is non-trivial, the above arguments can be applied to a covering $(\widetilde{M^n}, \widetilde \rho)$ 
of $(M,\rho)$ which trivializes the monodromy. It means that the desingularization $\check{\widetilde{M^n}}$ of
$\widetilde{M^n}$ is a trivial principal $\bbT^{t(\rho)}$-bundle, and we still have a smooth global 
cross section of $M^n$ over $Q$.     
\end{proof}

Consider now the case when $(M^n, \rho)$ has twistings. Then a-priori $Q$ is only an orbifold and we cannot 
have a submanifold $Q_c$ in $M^n$ diffeomorphic to $Q$. In this case, instead of a section, we will look for a  
\emph{multi-section} of $M^n \to Q$: a smooth {\bf multi-section} of $M^n \to Q$ is a smooth embedded 
submanifold with boundary and corners $Q_c$ in $M^n$, together with a finite subgroup $G \subset(Z_\rho \otimes \bbR)/Z_\rho$
such that $Q_c$ is invariant with respect to $G$ (i.e. if $z\in Q_c$ and 
$w \in G$ then $\rho(w,z) \in Q_c$), and $Q_c/G \cong Q$ via the projection.

\begin{prop}
\label{prop:multisection}
Assume that $(M^n,\rho)$ has twistings. Then the singular torus fibration $M^n \to M^n/\rho_\bbT = Q$
admits a smooth multi-section $(Q_c,G)$, where $G \subset(Z_\rho \otimes \bbR)/Z_\rho $ is generated by the 
twisting groups $G_z \quad (z \in M)$ of $(M^n,\rho)$.
\end{prop}

\begin{proof}
It results from Proposition \ref{prop:section} and an appropriate covering of $(M^n,\rho)$.
\end{proof}

\begin{remark}
Multi-sections also appear in many other places in the literature. 
For example, Davis and Januskiewicz in \cite{Davis-Convex1991} used them in their study of quasi-toric manifolds. Zung also used them in \cite{Zung-Symplectic1996}
in the construction of partial action-angle coordinates for singularities of integrable Hamiltonian systems.
\end{remark}

\begin{cor} \label{cor:EquivariantUniqueness}
 Assume that $(M_1^n,\rho_1)$ and $(M_2^n,\rho_2)$ have the same quotient space $M^n_1/\rho_{1 \bbT} = M_2^n/ \rho_{2 \bbT} = Q$,
and moreover they have the same isotropy  at every point of $Q$: $Z_{\rho_1}(q) = Z_{\rho_2}(q)\ \forall\ q \in Q$, where
$Z_{\rho_1}(q)$ means the isotropy group of $\rho_1$ on the $\rho_{1 \bbT}$-orbit corresponding to $q$. Then there is a
diffeomorphism $\Phi: M_1^n \to M_2^n$ which sends $\rho_{1 \bbT}$ to $\rho_{2 \bbT}$.
\end{cor}

\begin{proof}
 Simply send a multisection in $M_1^n$ over $Q$ to a multi-section 
in $M_2^n$ over $Q$ by a diffeomorphism which projects to the identity
map on $Q$, and extend this diffeomorphism to the whole $M_1^n$ in the unique equivariant way with respect to the associated
torus actions. The fact that the isotropy groups are the same allows us to do so. Notice that the obtained diffeomorphism $\Phi$
intertwines $\rho_{1 \bbT}$ with $\rho_{2 \bbT}$, but does not intertwine $\rho_{1}$ with $\rho_{2}$ in general.
\end{proof}

\subsection{Going back from $(Q,\rho_\bbR)$ to $(M^n,\rho)$}

In order to recover (or to construct) $(M^n,\rho)$ from its reduction $(Q,\rho_\bbR)$, we need (or can choose)
the following additional data:

1) Isotropy groups. Specify the isotropy group $Z_\rho(q)$ for each $q \in Q$. 

2) Associated vectors (for corank 1 singular orbits) and vector couples (for corank 2 transversally elbolic orbits)
in the sense of Definition \ref{def:AssociatedVector}. Note that these associated vectors and vector couples can be
attached to corank-1 orbits of $\rho_\bbR$ in $Q$. (They are images in $Q$ of  corank-1 hyperbolic 
and corank-2 elbolic orbits of $\rho$ in $M^n$). A corank-1 orbit in $Q$ marked ``hyperbolic'' will be given an 
associated vector, while a corank-1 orbit in $Q$ marked ``elbolic'' will be given a vector couple.

3) Lifting of the monodromy from $H_1(Q,\bbZ)  \to  \bbR^n/(Z_\rho \otimes \bbR) \cong \bbR^{\dim Q} $ to 
$H_1(M^n,\bbZ) / Im(Z_\rho) \to \bbR^n/Z_\rho $, which makes the diagramme (\ref{eq:diag-monodromy}) shown 
in the previous subsection commutative. Remark that, even if $M^n$ is unknown and to be found, 
the isotropy data together with $(Q,\rho_\bbR)$ determines $Z_\rho$ (which is equal to $Z_\rho
(q)$ for any regular
point $q \in Q$) and $H_1(M^n,\bbZ) / Im(Z_\rho)$ completely, so it makes sense 
to talk about this monodromy lifting.

Of course, these data are not arbitrary, but must satisfy a series of obvious  conditions, so that
they can be realized locally, i.e. a sufficiently small neighborhood of any point in $Q$ together with the action 
$\rho_\bbR$ and the above data (restricted to that neighbodhood) can be realized by some local model of 
$(M^n,\rho)$. If it is the case, then we will say that our data of isotropy groups, associated vectors and
vector couples, and monodromy lifting satisfy the {\bf local compatibility conditions}.

\begin{thm}  \label{thm:GoingBack}
Assume that $(Q,\rho_\bbR)$ is equipped with a full set of additional data consisting of the isotropy groups, 
the associated vectors and vector couples, and the monodromy lifting,
which satisfy the local compatibility conditions. Then, up to
isomorphisms, there exists a unique $(M^n,\rho)$ which admits these data and has $(Q,\rho_\bbR)$ 
as its reduction with respect to the associated torus action.
\end{thm}

The proof of the above theorem can be obtained easily by the same gluing method, 
as used  in the proof of some earlier theorems of this paper.

\section{Elbolic actions and toric manifolds}

\begin{defn} \label{defn:elbolic}
A nondegenerate action $\rho: \bbR^n \times M^n \to M^n$ is called { \bf elbolic}, if it does not admit any
hyperbolic singularity, i.e. all singular points have only elbolic components.
\end{defn}

\begin{thm}
Let $\rho: \bbR^n \times M^n \to M^n$ be an elbolic action. Then we have: 

1) $\rho$  has exactly one $n$-dimensional orbit. This orbit is open dense in $M^n$, and is of the
type $\bbT^{m+s} \times \bbR^m $ for some $s,m \geq 0$ such that $2m+s = n$. In particular the toric degree 
$t(\rho) = m+s$ is greater or equal to $n/2$. 

2) The monodromy of $\rho$ is trivial, and the quotient space $Q =M^n/\rho_\bbT$ of $M^n$ by the associated
torus action $\rho_\bbT$ is a contractible manifold with boundary and corners (and which is compact 
if and only if $M^n$ is compact), on which the reduced action $\rho_\bbR$ is 
nondegenerate totally hyperbolic and has only one regular orbit. If moreover $M^n$ is compact without boundary,
then $(Q,\rho_\bbR)$ is a contractible closed hyperbolic domain. 

3) If $(M^n,\rho)$ admits a fixed point, then $s=0$, $n=2m$ is an even number, and the toric degree $t(\rho)$
is equal to half of the dimension of $M^n$. 
\end{thm}

\begin{proof}

1) The singular set $S$ of the action $\rho$ on $M^n$ is of codimension at least 2 in
$M^n$, which implies that the regular set is connected, so it cannot be more than
one $n$-dimensional orbit. Let $p \in M^n$ be a singular point of highest corank of the action. Then the orbit 
$\cO_p$ though $p$ must be compact (otherwise the points on the boundary of this orbit
would be more singular than $p$). If  $\cO_p$ is of the type $\bbT^s$ ($s \geq  0$) then the regular orbit is of the type 
 $\bbT^{m+s} \times \bbR^m$ where $m$ is the number of elbolic components at $p$.

2) The monodromy is trivial because the group $H_1(M^n,\bbZ)/ Im(Z_\rho)$ itself is trivial in this case: any loop
in $M^n$ is homotopic to a loop in the regular orbit. Since the monodromy is trivial, there is no twisting, and
so $Q$ is a manifold with boundary and corners,
according to the results of Section \ref{section:QuotientSpace}. Since $\rho$ has only one regular
orbit, $\rho$ also has only one regular orbit in $Q$. The fact that $Q$ is contractible now follows from Theorem
\ref{thm:hyperbolic-contractible} (which remains true, and with the same proof, also in the non-compact case).

3) If $p$ is a fixed point, then the H,R,T components of its HERT-invariant vanish, and $n=2m$ and $t(\rho) = m$, where
$m$ is the E-component of its HERT-invariant $(m,0,0,0)$.
\end{proof}

The case of elbolic actions with a fixed point, i.e. the last case in the above theorem, 
is of special interest in geometry, because of its connection to the so-called toric manifolds.

Recall that, a toric manifold in the sense of complex geometry
is a complex manifold (which is often equipped with a K\"alerian structure, or equivalently, a compatible symplectic
structure) of complex dimension $m$ together with a holomorphic action 
of the complex torus $(\bbC^*)^m$ which has an open dense orbit.
See e.g. \cite{Audin-Torus2004,Cox-Toric2011} for an introduction to toric manifolds. 
From our point of view, such a toric manifold has real dimension
$n = 2m$, and the action of $(\bbC^*)^m \cong  \bbR^m \times \bbT^m$ is an elbolic 
nondegenerate $\bbR^{2m}$-action. Thus, elbolic 
actions are a natural generalization of complex toric manifolds.
Complex toric manifolds are classfied by their associated fans. So our classification 
of hyperbolic domains (and of the quotient spaces of elbolic actions) are very similar 
to the classification of complex toric manifolds, except that, unlike the complex case,
the vectors of our fans are not required to lie in an integral lattice.

For real manifolds, there are at least 3 different well-studied generalizations of the notion
of toric manifolds, namely:

1) {\bf Quasi-toric manifolds} as defined by Davis and Januskiewicz in \cite{Davis-Convex1991}. 
Orginally these  manifolds were also called \emph{toric}, but later on people adopted the name 
\emph{quasi-toric} to avoid confusion with complex toric manifolds. A quasi-toric manifold is a real
$2m$-dimension manifold $M^{2m}$ with a almost-everywhere-free action of $\bbT^m$ such that:

i) The orbit space $M^{2m}/\bbT^m $ is hemeomorphic to a simple convex polytope $P^m$
and the preimage of each point on a $k$-dimensional face of $P^m$ via projection $M^{2m} \to P^m$
is a $k$-dimensional orbit of the $\bbT^m$-action on $M$. In particular,
vertices of $P$ correspond to fixed point of the action on $M$.

ii) Near each fixed point the action is locally isomorphic (up to automorphisms of $\bbT^m$) 
to a standard action of $\bbT^m$ on $\bbC^m$ given by
\begin{multline}
(\alpha _1, \hdots, \alpha_m).(z_1,\hdots, z_m) = (\alpha_1.z_1, \hdots, \alpha_m.z_m)\\
\cong \{ (\alpha _1, \hdots, \alpha_m) \in \bbC^m \ |\ |\alpha_i| = 1 \quad \forall i\}.
\end{multline}
2) {\bf Torus manifolds} as defined by Hattori and Masuda in \cite{Hattori-fan2003}.
A torus manifold is simply a closed connected orientable smooth manifold $M$
of dimension $2m$ with an effective smooth action of $\bbT^m$ having a fixed point.

3) {\bf Topological toric manifolds} as defined by Ishida, Fukukawa and Masuda \cite{Ishida-toric}. A
topological toric manifold is a closed smooth manifold $M$ of dimension $n = 2m$ 
with an almost-everywhere-free smooth action of $(\bbC^*)^m \cong \bbT^m \times \bbR^m$
which is covered by finitely many invariant open subsets each equivariantly  diffeomorphic to a 
direct sum of complex 1-dimensional linear representation of $\bbT^m \times \bbR^m$.

We observe that Ishida--Fukukawa--Masuda's notion of topological toric manifolds is equivalent to
our notion of manifolds admitting an elbolic action whose toric degree is half the dimension of the
manifolds. The proof of the following proposition is a simple verification that their conditions and
our conditions are the same:

\begin{prop}
 A manifold $M^{2m}$, together with a smooth  action of $(\bbC^*)^{m} \cong \bbR^m \times \bbT^m$,
is a topological toric manifold if and only if the action (which may be viewed as an action of $\bbR^{2m}$) 
is elbolic of toric degree $m$.
\end{prop}

Thus, we recover topological toric manifolds from our more general considerations of nondegenerate
$\bbR^n$-actions on $n$-manifolds. 

We refer to the paper of Ishida, Fukukawa and Masuda 
\cite{Ishida-toric} and some related recent works  
\cite{ChoiMasudaSuh-Toric2011,Ishida-Complex2011,Yu-Toric2011} 
for a detailed study of topological toric manifolds. Let us just mention here that, according to the results
of  \cite{Ishida-toric}, topological toric manifolds are \emph{the} right generalization of the notion of toric
manifolds to the category of real manifolds; they have very nice homological properties similarly to
toric manifolds (see Section 8 of  \cite{Ishida-toric}), and they are classifified by the so-called 
{\bf complete non-singular topological fans}. 

The complete non-singular topological fan of a topological toric manifold, in the sense
of  \cite{Ishida-toric}, encodes the following data: the complete fan in $\bbR^n$
associated to the reduced totally hyperbolic action $\rho_\bbR$ on the quotient space $Q = M^{2m}/\rho_\bbT$,
and the vector couples associated to corank-2 transversally elbolic orbits (see Definition \ref{def:AssociatedVector}).
These vector couples tell us how to build back $(M^{2m},\rho)$ from $(Q, \rho_\bbR)$.
So one can recover Ishida--Fukukawa--Masuda's classification theorem for topological toric manifolds
from our point of view of general nondegenerate $\bbR^n$-actions on $n$-manfiolds: one can prove this theorem
in the same way as the proof of Theorem \ref{thm:classificationByfan}, by gluing together local pieces equipped with
canonical coordinates and adapted bases. Another very interesting proof, based on the quotient method,
which represents the  topological toric manifold $(M^{2m},\rho)$ as a 
quotient of another global object, is given in \cite{Ishida-toric}. (The quotient method is also discussed in
\cite{Audin-Torus2004,Cox-Toric2011} for the construction of toric manifolds).
 
\section{Actions of toric degree $n-2$}

\subsection{Three-dimensional case}

Consider an action $\rho: \bbR^3 \times M^3 \to M^3$ of toric degree 1.
Let $q \in \cO_q$ be a point in a singular orbit of $\rho$. Denote the
HERT-invariant of $q$ by $(h,e,r,t)$, and by $k=\rank_{\bbZ_2}G_q$ the rank over $\bbZ_2$
of the twisting group $G_q$ of $\rho$ at $q$. According to the results of the previous sections,
we have following constraints on the nonnegative integers $h,e,r,t,k$:
\begin{equation}
h + 2e + r + t = 3,\ \ e+t = 1,\ \ e + h \geq 1,\ \ k \leq \min (h,t). 
\end{equation}
In particular, we must have $k\leq 1$, i.e. the twisting group $G_q$
is either trivial or isomorphic to $\bbZ_2$.

Taking the above constraints into account, we have the following 
full list of possibilities for the singular point $q$, together with their abbreviated names:

I. $(h) \quad h = 1, e=0, r=1, t=1, G_q = \{0\}$

II. $(h_t) \quad h = 1, e=0, r=1, t=1, G_q = \bbZ_2$

III. $(e) \quad h = 0, e=1, r=1, t=0, G_q = \{0\}$

IV. $(h-h) \quad h = 2, e=0, r=0, t=1, G_q = \{0\}$

V. $(h-h_t) \quad h = 2, e=0, r=0, t=1, G_q = \bbZ_2$
acting by the involution $(x_1,x_2)\mapsto (-x_1,x_2)$

VI. $((h-h)_t) \quad h = 2, e=0, r=0, t=1, G_q = \bbZ_2$
acting by the involution $(x_1,x_2)\mapsto (-x_1,-x_2)$

VII. $(e-h) \quad h = 1, e=1, r=0, t=0, G_q = \{0\}$

In the above list, $(h)$ means hyperbolic non-twisted, $(h-h)_t$ means a joint twisting
of a product of 2 hyperbolic components, and so on.

The local structure of the corresponding 2-dimensional quotient space $Q^2=M^3/\rho_{\bbT}$
(together with the traces of singular orbits on $M^3$) is described in Figure \ref{fig:7types}.
\begin{figure}[htb] 
\begin{center}
\includegraphics[width=120mm]{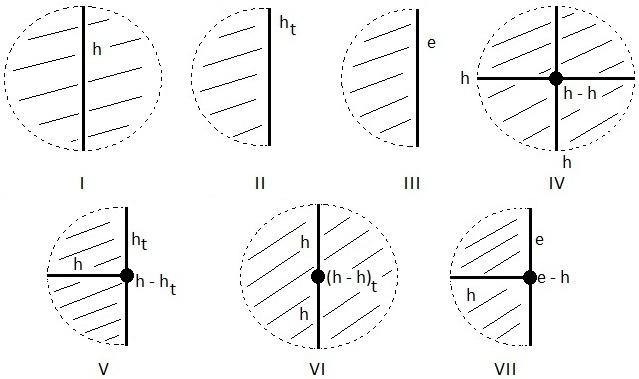}
\caption{The 7 types of singularities of $\bbR^3$-actions of toric degree 1 on 3-manifolds.}
\label{fig:7types}
\end{center}
\end{figure}

Remark that, in Case VI, locally $Q\cong D^2/\bbZ_2$ is homemorphic but not diffeomorphic to a disk.
In the other cases, $Q$ can be viewed locally as either a disk (without boundary) or a half-disk (with boundary)
but it cannot be a corner.

Globally, the quotient space $Q$ can be obtained by gluing copies of the above 
7 kinds of local pieces together, in a way which respects the letters (e.g. an edge marked $e$
will be glued to an edge marked $e$, an edge marked $h_t$ will be glued to an edge marked $h_t$).

Notice, for example, that Case II and Case III in the above list are different but have diffeomorphic quotient spaces. 
To distinguish such situations, we must attach letters to the singularities, which describe the corresponding types
of singularities coming from $(M^3,\rho)$. The quotient space $Q$ together with these letters on its graph of singular
orbits will be called the {\bf typed quotient space} and denoted by $Q_{typed}$.

\begin{thm} \label{thm:n-2a}
 1) Let  $(Q_{typed},\rho_\bbR)$ be the quotient space of $(M^3,\rho)$, where $\rho$ is of toric degree 1
and $M^3$ is a 3-manifold without boundary. Then each singularity of $Q_{typed}$ belongs to one of the seven types I--VII
listed above.

2) Conversely, let $(Q_{typed},\rho_\bbR)$ be a 2-orbifold together with a totally hyperbolic action $\rho_\bbR$ on it, and
together with the letters on the graph of singular orbits, such that the singularities of $Q_{typed}$ belong to the above
list of seven types I--VII. Then there exists $(M^3,\rho)$ of toric degree 1 which admits  $(Q_{typed},\rho_\bbR)$
as its quotient. Moreover, the $\bbT^1$-equivariant diffeomorphism type of $M^3$ is completely determined by $Q_{typed}$.
\end{thm}

\begin{proof}
 1) It was shown above that the list I--VII is complete in the case of dimension 3, due to dimensional constraints.

2) When the toric degree is 1, assuming that $Z_\rho \cong \bbZ$ is fixed in $\bbR^3$, 
because $Z_\rho$  has only 1 dimension and doesn't allow multiple choices, 
we have existence and uniqueness for the choice of isotropy groups in this case. 
The second part of the theorem now follows from Theorem \ref{thm:GoingBack} and
Corollary \ref{cor:EquivariantUniqueness}.
\end{proof}

\begin{figure}[htb] 
\begin{center}
\includegraphics[width=120mm]{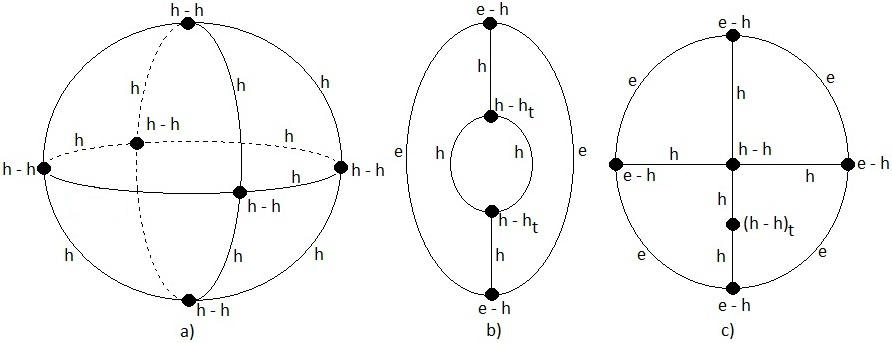}
\caption{Example of $Q^2$ for $n=3, t(\rho) =1$.}
\label{fig:Q2-7kinds}
\end{center}
\end{figure}

\begin{example}
Some examples of realizable $Q_{typed}$ which can be obtained by gluing the above 7 kinds of pieces
are shown in Figure \ref{fig:Q2-7kinds}. Notice that $Q_{typed}$ may be without boundary 
(as in Figure \ref{fig:Q2-7kinds}a) or with boundary 
(Figure \ref{fig:Q2-7kinds}b and \ref{fig:Q2-7kinds}c). 
The boundary components of $Q_{typed}$ corresponds to the orbits of type $e$ (elbolic)
and $h_t$ (hyperbolic twisted). In the interior of $Q_{typed}$, one may have edges of type $h$ (hyperbolic non-twisted)
and singular points of type $h-h$ or $(h-h)_t$. In Figure \ref{fig:Q2-7kinds}c,
$Q_{typed}$ is not a smooth manifold (though it is homeomorphic to a disk). The branched 2-covering of 
Figure \ref{fig:Q2-7kinds}c is shown in Figure \ref{fig:branched-Double} 
($\bbZ_2$ acts by rotating 180° around 0).
It is easy to see that, the 3-manifolds corresponding to the situations a), b) c) in this example are $\bbS^2 \times \bbS^1$, $\bbR \bbP^2 \times \bbS^1$, and $\bbR \bbP^3$ respectively.
\end{example}

\begin{figure}[htb] 
\begin{center}
\includegraphics[width=80mm]{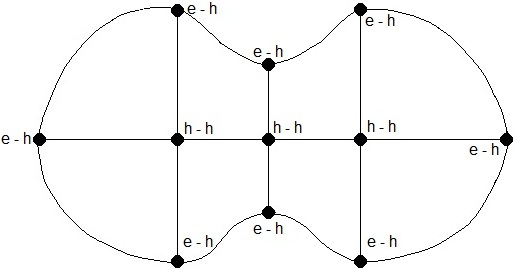}
\caption{Branched double covering of Figure \ref{fig:Q2-7kinds}c.}
\label{fig:branched-Double}
\end{center}
\end{figure}

\begin{remark} Let us mention also that,
 if $M$ admits an action of $\bbR^3$ of toric degree 1, 
then $M^3$ is a graph-manifold in the sense of Waldhausen, which generalizes the notion of Seifert fibrations. 
Actually, any 3-manifold admitting a nontrivial circle action is a graph-manifolds (with some additional
properties), and graph-manifolds form a very special and well-studied class of 3-manifolds 
in topology, see, e.g. \cite{JacoShalen-GraphManifolds}. As was observed by Fomenko
\cite{Fomenko-MorseTheory1986}, graph-manifolds are also precisely
those manifolds which can appear as isoenergy 3-manifolds in an integrable Hamiltonian system
with 2 degrees of freedom.
\end{remark}

\subsection{The case of dimension $n \geq 4$}

When the dimension $n$ is at least 4, the toric degree is $n-2 \geq 2$, we have the following 3 new types
of singularities, in addition to the 7 types listed in the previous subsection:

VIII. $(h_t-h_t) \quad h = 2, e=0, r=0, t=n-2, G_q = \bbZ_2 \times \bbZ_2$ acting separately on the two hyperbolic components.

IX. $(e-h_t) \quad h = 1, e=1, r=0, t=n-3, G_q = \bbZ_2$.

X. $(e-e) \quad h = 0, e= 2, r=0, t=n-4, G_q = \{0\}$.

(See Figure \ref{fig:add3types}).
\begin{figure}[htb] 
\begin{center}
\includegraphics[width=100mm]{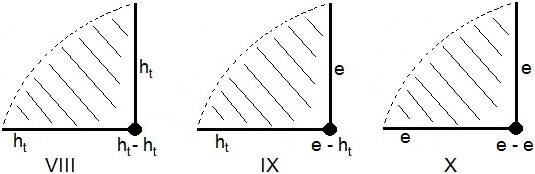}
\caption{The additional 3 possible types of singularities for actions of toric degree $n-2$ when $n\geq 4$.}
\label{fig:add3types}
\end{center}
\end{figure}

\begin{thm} \label{thm:n-2b}
1) Let  $(Q_{typed},\rho_\bbR)$ be the quotient space of $(M^n,\rho)$, where $\rho$ is of toric degree 1
and $M^n$ is a n-manifold without boundary and $n \geq 4$. Then each singularity of $Q_{typed}$ belongs to one of the ten types I--X
listed above.

2) Conversely, let $(Q_{typed},\rho_\bbR)$ be a 2-orbifold together with a totally hyperbolic action $\rho_\bbR$ on it, and
together with the letters on the graph of singular orbits, such that the singularities of $Q_{typed}$ belong to the above
ten types I--X. Then for any $n \geq 4$ there exists $(M^n,\rho)$ of toric degree 
n-2 which admits  $(Q_{typed},\rho_\bbR)$ as its quotient. 
\end{thm}

\begin{proof}
The main point of the proof is to show that  one can choose compatible isotropy groups, but it is a simple excercise.
Remark that, unlike the case of dimension 3, when $n \geq 4$, the typed quotient $Q_{typed}$ 
does not determine the diffeomorphism type of
manifold $M$ completely, because there are now multiple choices for the isotropy groups.
\end{proof}

\begin{figure}[htb] 
\begin{center}
\includegraphics[width=50mm]{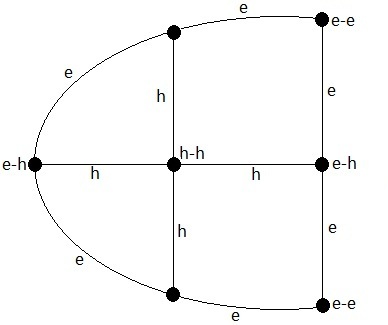}
\caption{Example of $Q = M^4/\bbT^2$.}
\label{fig:exampleM4T2}
\end{center}
\end{figure}

\begin{example}
 An example of the quotient space $Q$, which can't appear for $n=3$ but can appear for $n \geq 4$, 
is shown in  Figure \ref{fig:exampleM4T2}.
\end{example}

\noindent {\bf Acknowledgements.} We would like to thank the referees for many useful remarks on the paper.

\end{document}